\documentclass[aap,preprint]{imsart}

\RequirePackage[OT1]{fontenc}
\RequirePackage{amsthm,amsmath,amsfonts}
\RequirePackage[numbers]{natbib}
\RequirePackage[colorlinks,citecolor=blue,urlcolor=blue]{hyperref}

\usepackage{blindtext}
\usepackage{enumitem}

\startlocaldefs
\numberwithin{equation}{section}
\theoremstyle{plain}
\newtheorem{theorem}{Theorem}[section]
\newtheorem{lemma}[theorem]{Lemma}

\theoremstyle{definition}
\newtheorem{definition}[theorem]{Definition}
\newtheorem{remark}[theorem]{Remark}

\endlocaldefs

\allowdisplaybreaks

\begin{document}

\begin{frontmatter}

\title{Stochastic  parabolic evolution  equations in M-type 2 Banach spaces}
\runtitle{Stochastic  parabolic evolution equations}

\begin{aug}

\author{T\^{o}n Vi$\hat{\d{e}}$t  T\d{a}\ead[label=e1]{tavietton[at]agr.kyushu-u.ac.jp}}


\runauthor{T$\hat{\rm o}$n Vi$\hat{\d{e}}$t T\d{a}}

\affiliation{Kyushu University}

\address{T\^{o}n Vi$\hat{\d{e}}$t T\d{a}\\
Center for Promotion of International Education and Research \\
 Kyushu University \\
 Hakozaki Fukuoka 812-8581, Japan\\
\printead{e1}}
\end{aug}

\begin{abstract}
This paper is devoted to studying  stochastic parabolic evolution  equations with additive noise in  Banach spaces of M-type 2. We construct both strict and mild solutions possessing very strong regularities. First, we consider  the  linear case. We prove  existence and uniqueness of strict  and mild solutions and show their maximal regularities.  Second, we explore the semilinear case. Existence, uniqueness and regularity of mild and strict solutions are shown. Regular dependence of mild solutions on initial data is also investigated. Finally, some applications to stochastic partial differential equations are presented.
\end{abstract}


\end{frontmatter}

\section {Introduction}
We study the Cauchy problem for a stochastic parabolic  evolution  equation
\begin{equation} \label{E2}
\begin{cases}
dX+AXdt=[F_1(t)+F_2(X)]dt+ G(t)dW(t), \hspace{1cm} 0<t\leq T,\\
X(0)=\xi
\end{cases}
\end{equation}
in a complex separable Banach space $E$ with  norm $\|\cdot\|$ and the Borel $\sigma$-field $\mathcal B(E)$. 
Here, $W$ is a  cylindrical Wiener process on a separable Hilbert space $H,$ and is defined on a complete filtered  probability space  $(\Omega, \mathcal F,\mathcal F_t,\mathbb P)$; $G(t), 0\leq t\leq T,$ are $\gamma$-radonifying operators from $H$ to $E$; 
$F_1$ and $F_2$ are  measurable functions from $(\Omega_T, \mathcal P_T)$ and $(\Omega\times E, \mathcal F\times \mathcal B(E))$  into $(E,\mathcal B(E)),$ respectively (where $\Omega_T=[0,T]\times \Omega$ and $\mathcal P_T$ is the predictable $\sigma$-field on $\Omega_T)$; and $A$ is a densely defined, closed linear operator which generates an analytic semigroup on $E$.


It is known that many interesting models introduced from the real world can be formulated by deterministic linear or nonlinear evolution equations of parabolic type. These equations generally  generate not only  a dynamical system but also enjoy a global attractor even a finite-dimensional attractor. Existence of such an attractor then suggests that the phenomena whose processes are described by these parabolic evolution equations enjoy some robustness in a certain abstract sense. Some may be the pattern formation and others may be the specific structure creation.
In these cases, robustness of final states of process is one of main issues to be concerned. It is therefore quite natural in order to investigate the robustness to consider stochastic parabolic evolution equations.

Stochastic parabolic evolution equations  have already been  studied since 1970s. The main interest was to construct unique mild solutions, see Brze\'{z}niak   \cite{Brzezniak}, Da Prato-Zabczyk \cite{prato},   van Neerven et al.  \cite{van4}, and references therein. Some researchers, however, were devoted to constructing stronger solutions, strict solutions. 

Da Prato-Zabczyk \cite{prato}  presented a theory of stochastic equations in infinite dimensions. 
In the framework where the space $E$ is a Hilbert space, $W$  is a $Q$-Wiener process on $H\equiv E$ with   a  symmetric nonnegative nuclear operator $Q$ in $ L(E)$, the authors studied the linear of  \eqref{E2}, i.e. 
\begin{equation} \label{E0.1}
\begin{cases}
dX+AXdt=F(t)dt+ I dW(t),\hspace{1cm} 0<t\leq T,\\
X(0)=\xi,
\end{cases}
\end{equation}
where 
 $F$ is progressively measurable and integrable process on $[0,T]$,  and $I$ is the identity operator in $E$.
Under the  conditions:
\begin{itemize}
 \item [{\rm (i)}] 
  The trace of $  Q  $ is finite,
  \item [{\rm (ii)}]  $A^\beta Q^{\frac{1}{2}} $ for some $ \frac{1}{2}<\beta\leq 1  $ is a Hilbert-Schmidt operator,
\end{itemize}
 existence of strong solutions to \eqref{E0.1} has been shown. Clearly, the condition  {\rm (ii)}  is very restrictive. In the case of standard Brownian motions in $\mathbb R^d$, $Q$ becomes  the identity matrix of size $d.$  Then, {\rm (ii)} implies that $A$ is a bounded linear operator.

Brze\'{z}niak  \cite{Brzezniak}-\cite{Brzezniak3} constructed a theory of stochastic partial differential equations in  Banach spaces of M-type 2.  
In \cite{Brzezniak}, the author treated  a  stochastic linear evolution equation 
\begin{equation} \label{E0.5}
\begin{cases}
dX+AXdt = F(t)dt+ \sum_{j=1}^d B_j Xdw^j(t), \hspace{1cm}  0<t\leq T, \\
X(0)=\xi
\end{cases}
\end{equation}
in an M-type 2  separable Banach space $E$. Here $A$ is as above, 
 $F$ is progressively measurable and square integrable  process on $[0,T]$, 
 $B_j  (j=1,\dots,d)$ are unbounded linear operators in  $E$, 
and $w^j (j=1,\dots,d)$ are  independent standard Brownian motions.
It has been assumed that  
  $$\sum_{j=1}^d\|B_j u\|_{\mathcal D_A(\frac{1}{2},2)}^2 \leq C_1 \|u\|_{\mathcal D(A)}^2 + C_2\|u\|_{\mathcal D_A(\frac{1}{2},2)}^2, \hspace{1cm} u\in \mathcal D(A),$$ 
where $\mathcal D_A(\frac{1}{2},2)=\{u\in E; \int_0^\infty \|Ae^{tA}u\|^2 dt<\infty\} $ and that  $$\xi \in L^2(\Omega, \mathcal D_A(\frac{1}{2},2)).$$ 
Existence of a unique strict solution to \eqref{E0.5}  has been then proved in the space: 
$$X\in L^2([0,T]\times \Omega;\mathcal D(A)) \cap C([0,T];L^2(\Omega;\mathcal D_A(\frac{1}{2},2))).$$

 In addition,   
 regularity of mild solutions is also presented. Under the conditions:
 \begin{itemize}
   \item [\rm (i)]
 $A^{-\beta} G \in \mathcal M^p([0,T];L(U;E))$ 
\item [\rm (ii)] $\mathbb E \|A^\delta \xi\|^p<\infty$
\item [\rm (iii)] $F_1\in \mathcal M^p([0,T];\mathcal D(A^{\alpha-1}))$ for some $0\leq \delta<\gamma<\delta+\frac{1}{p}<\min\{\alpha,\frac{1}{2}-\beta\}$
 and $p\geq 2$, where $ \mathcal M^p([0,T];\mathcal S)$ is the space of all progressive measurable functions $u$ from $[0,T]$ to $\mathcal S$ such that $$\mathbb E \int_0^T \|u(s)\|_{\mathcal S}^pdu<\infty,$$
\end{itemize}
 existence of a unique mild solution  to \eqref{E0.5} has been  proved in the space: 
 $$X\in \mathcal M^p([0,T];\mathcal D(A^\gamma))\cap L^p(\Omega;\mathcal C([0,T];\mathcal D(A^\delta)).$$

  The equation  \eqref{E2} is already considered in  \cite{Brzezniak1}.
Under the local Lipschitz condition on $F_2 $
\begin{equation}  \label{E0.6}
 \|F_2(u)-F_2(v)\|_{\mathcal D(A^{\alpha-1})} \leq C \|u-v\|_{\mathcal D(A^\gamma)}^\nu  \|u-v\|_{\mathcal D(A^\delta)}^{1-\nu}
\end{equation}
for some $0\leq \nu <1$, 
  existence of  maximal local mild solutions of  \eqref{E2} has been shown. 

van Neerven et al. \cite{van1} constructed a theory of stochastic integration in UMD Banach spaces. The authors then systematically studied  stochastic evolution equations in UMD Banach spaces (see \cite{van2}-\cite{van6} and references therein).

A theory of abstract (deterministic) parabolic evolution equations was introduced by Yagi  \cite{yagi}. In this theory, strict solutions to the deterministic version of  \eqref{E2} (i.e. $G\equiv 0$ on  $[0,T]$) were  investigated. For the linear case (i.e. $F_2\equiv 0$),  
when $ F_1$ belongs to a weighted H\"{o}lder continuous function space $\mathcal F^{\beta, \sigma}((0,T];E))$ (see the  definition in Section \ref{section2}) and  $\xi\in E$ is arbitrary,  then existence of a unique  solution  was  proved in the space (see  \cite{yagi0}):
$$X\in \mathcal C^1((0,T];E)\cap \mathcal C([0,T];E) \cap \mathcal C((0,T];\mathcal D(A)).$$
Furthermore, the author obtained the maximal regularity for both initial value $\xi \in \mathcal D(A^\beta)$ and external force function $F_1\in \mathcal F^{\beta, \sigma}((0,T];E)$  (see \cite{yagi1}-\cite{yagi}):
$$A^\beta X\in \mathcal  C([0,T];E),$$
$$\frac{dX}{dt}, AX \in \mathcal F^{\beta, \sigma}((0,T];E).$$
For the  semilinear case ($F_2\not\equiv 0$), in  \cite{Osaki} it was  assumed that 
$\xi\in \mathcal D(A^\beta), F_1\in  \mathcal F^{\beta, \sigma}$ $((0,T];E).$ In addition, $F_2$ is a nonlinear operator from $\mathcal D(A^\eta)$ into $E$  and satisfies  a Lipschitz condition of the form
\begin{align*}
&\|F_2(u)-F_2(v)\|\leq \varphi (\|A^\beta u\|+A^\beta v\|)\\
&\times [\|A^\eta(u-v)\|+ (\|A^\eta u\|+\|A^\eta v\|)\|A^\beta (u-v)\|], \quad u,v\in \mathcal D(A^\eta),
\end{align*}
where $0<\beta\leq \eta<1, 0<\sigma<\min\{\beta, 1-\eta\}$ and $\varphi$ is some increasing function. 
Existence of a unique local  solution was then proved in the function space:
$$X\in \mathcal C((0,\bar T];\mathcal D(A))\cap  \mathcal C([0,\bar T];\mathcal D(A^\beta))\cap  \mathcal C^1((0,\bar T];E),$$
$$\frac{dX}{dt}, AX\in \mathcal F^{\beta, \sigma}((0,\bar T];E),$$
where $\bar T$ depends only on  norms  $\|F_1\|_{\mathcal F^{\beta, \sigma}}$ and $\|A^\beta \xi\|$.

In this paper, we want to extend the techniques of \cite{yagi} to the stochastic equation \eqref{E2}. Although we only treat additive noise, we  construct both strict and mild solutions possessing very strong regularities. In some cases, these regularities are maximal. 

Our contribution is twofold. 
First, we  handle  the linear case of \eqref{E2}. Existence of a unique strict solution as well as its   maximal space-time regularity are proved. Maximal space-time regularity of mild solutions is also shown.

Second, we  treat  \eqref{E2} in general.  
By using the fix-point method, we prove  existence and uniqueness of a local mild solution, which is defined on an interval $[0,\bar T]\subset [0,T],$ where $\bar T>0$ is a  nonrandom constant. (Note that another common way to prove  existence of mild solutions is to use the cut-off method. This method however results in that local mild solutions are defined on 
$[0,\tau),$ where $\tau$ is a stopping time (see  \cite{Brzezniak1,van2})). 
Based on solution formula, we  give a sufficient condition for existence of a unique global strict solution to \eqref{E2}. In addition,  regularity and regular dependence on initial data of  local mild solutions are also investigated.

For this current research, we  use semigroup methods, which takes idea from ones for (deterministic) parabolic evolution equations. More precisely, the methods have been  initiated by the invention of the analytic semigroups in the middle of the last century (see Hille \cite{Hille} and Yosida \cite{Yosida}).
 The semigroup methods are characterized  by precise formulas representing the solutions of the Cauchy problem for stochastic  evolution equations.

The analytical semigroup $S(t)=e^{-tA}$ generated by the linear operator $(-A)$ directly provides  a mild solution to the Cauchy problem for a stochastic linear evolution equation
\begin{equation*}
\begin{cases}
dX+AXdt=F(t)dt + G(t)dW(t), \hspace{1cm} 0<t\leq T,\\
X(0)=\xi.
\end{cases}
\end{equation*}
The solution is given by the formula 
$$X(t)=S(t)\xi+\int_0^t S(t-s)F(s)ds+ \int_0^t S(t-s)G(s)dW(s).$$
 Similarly, a solution to \eqref{E2}
can be obtained as a solution of an integral equation 
$$X(t)=S(t)\xi+\int_0^t S(t-s)[F_1(s)+F_2(X(s))]ds+ \int_0^t S(t-s)G(s)dW(s).$$

These  solutions formulas provide us important information on solutions such as uniqueness, regularity, smoothing effect and so forth. Especially, for nonlinear problems one can derive Lipschitz continuity of solutions with respect to  initial values, even their Fr\'{e}chet differentiability. This powerful approach has been used for the study of stochastic evolution equations in Hilbert spaces.
Some early work  has been proposed by Curtain-Falb \cite{CurtainFalb}, Dawson  \cite{Dawson0}. After that,  there are many  papers and monographs following it.  We refer for instance  to  \cite{Brzezniak}-\cite{Brzezniak3}, \cite{prato-Flandoli}-\cite{prato}, \cite{Gawarecki},  \cite{PrevotRockner}-\cite{van6},  
and references therein.

The organization of the  paper is as follows. Section \ref{section2} is preliminary. We introduce notions such as weighted H\"{o}lder continuous function spaces,  analytical semigroups, cylindrical Wiener processes, strict and mild solutions to \eqref{E2}.

Section \ref{section4} deals with  the linear case of \eqref{E2}, i.e. $F_2\equiv 0$ in $E$. We prove existence of unique strict solutions (Theorems \ref{Th1} and  \ref{Th2}), provided that $F_1$ and $G$ belong to weighted H\"{o}lder continuous function spaces. Maximal regularity of strict solutions  is then presented (Theorem \ref{Th3}). Maximal regularity of mild solutions is also shown (Theorem \ref{Th4}).

Section \ref{section5} studies the general case of \eqref{E2}. Existence, uniqueness and regularity of strict and mild solutions are proved (Theorems \ref{Th8}, \ref{Th9} and \ref{Th6}). Regular dependence of mild solutions on initial data is also shown (Theorem \ref{Th10}).

In Section \ref{section6}, we present some examples of stochastic partial differential equations (PDEs) to illustrate our abstract results.

 \section{Preliminary}  \label{section2}
\subsection{Function spaces and analytical semigroups}
First, let us  review some notions of function spaces. For  $0<\sigma<\beta< 1,$ denote by   $\mathcal F^{\beta, \sigma}((0,T];E)$ the space  of all $E$-valued continuous functions $F$ on $(0,T]$    with the following three properties:
\begin{itemize}
  \item [\rm (i)]  
  \begin{equation}  \label{H11.2}
  t^{1-\beta} F(t)  \text{  has a limit as   } t\to 0.
  \end{equation}
  \item [\rm (ii)]  $F$ is H\"{o}lder continuous with  exponent $\sigma$ and   weight $s^{1-\beta+\sigma}$, i.e.
  \begin{equation} \label{H11.3}
\begin{aligned}
\sup_{0\leq s<t\leq T} & \frac{s^{1-\beta+\sigma}\|F(t)-F(s)\|}{(t-s)^\sigma}\\
&=\sup_{0\leq t\leq T}\sup_{0\leq s<t}\frac{s^{1-\beta+\sigma}\|F(t)-F(s)\|}{(t-s)^\sigma}<\infty.
\end{aligned}
\end{equation}
  \item [\rm (iii)] 
  \begin{equation} \label{Fbetasigma3}
  \lim_{t\to 0} w_{F}(t)=0, 
  \end{equation}
  where $w_{F}(t)=\sup_{0\leq s  <t}\frac{s^{1-\beta+\sigma}\|F(t)-F(s)\|}{(t-s)^\sigma}$.
\end{itemize}  
It is easy to see that   $\mathcal F^{\beta, \sigma}((0,T];E)$ is a Banach space with   norm
$$\|F\|_{\mathcal F^{\beta, \sigma}(E)}=\sup_{0\leq t\leq T} t^{1-\beta} \|F(t)\|+ \sup_{0\leq s<t\leq T} \frac{s^{1-\beta+\sigma}\|F(t)-F(s)\|}{(t-s)^\sigma}.$$
 
The space $\mathcal F^{\beta, \sigma}((0,T];E)$ is called a {\it weighted H\"{o}lder continuous function space} that is introduced by Yagi  (see \cite{yagi}). Clearly,  for  $F\in \mathcal F^{\beta, \sigma}((0,T];E),$
\begin{equation} \label{H12}  
\|F(t)\|\leq \|F\|_{\mathcal F^{\beta, \sigma}(E)} t^{\beta-1},  \hspace{1cm} 0<t\leq T,  
\end{equation}
and 
\begin{align}
 \|F(t)-F(s)\|
 & \leq w_{F}(t) (t-s)^{\sigma} s^{\beta-\sigma-1}    \label{H12.5} \\
& \leq \|F\|_{\mathcal F^{\beta, \sigma}(E)} (t-s)^{\sigma} s^{\beta-\sigma-1}, \hspace{1cm} 0<s<t\leq T.  \notag
\end{align}
In addition, it is not hard to see that
\begin{equation} \label{FbetaFgammasigmaSpaceProperty}
\mathcal F^{\gamma,\sigma} ((0,T];E)\subset \mathcal F^{\beta,\sigma} ((0,T];E), \hspace{1cm} 0<\sigma<\beta<\gamma< 1.
\end{equation}
\begin{remark}  \label{rm1}
\begin{itemize}
  \item [{\rm (a)}]
The space $\mathcal F^{\beta, \sigma}((0,T];E)$ is not  a trivial space.  The function $F$ defined by  $F(t) =t^{\beta-1} f(t), 0<t\leq T,$ belongs to this space, where  $f$ is any   function  
in  $ \mathcal C^\sigma([0,T];E)$ such that  $f(0)=0,$ where $ \mathcal C^\sigma$ is the space of $\sigma$-H\"{o}lder continuous functions.
\item [{\rm (b)}]  The space $\mathcal F^{\beta, \sigma}((a,b];E)$, $0\leq a<b<\infty$, is defined in a similar way. For more details, see \cite{yagi}.
\end{itemize}
\end{remark}

Next, let us  recall the notion of sectorial operators and some properties of  analytical semigroups generated by them.

A densely defined, closed linear operator $A$ is said to be sectorial if it satisfies two conditions:
\begin{itemize}
  \item [(\rm{Aa})] The spectrum $\sigma(A)$ of $A$ is  contained in an open sectorial domain $\Sigma_{\varpi}$: 
\begin{equation*} \label{H1} 
\sigma(A) \subset  \Sigma_{\varpi}=\{\lambda \in \mathbb C: |\arg \lambda|<\varpi\}, \quad \quad 0<\varpi<\frac{\pi}{2}.
       \end{equation*}
  \item [(\rm{Ab})] The resolvent of $A$ satisfies the estimate  
\begin{equation*} \label{H2}
          \|(\lambda-A)^{-1}\| \leq \frac{M_{\varpi}}{|\lambda|}, \quad\quad\quad \quad   \lambda \notin \Sigma_{\varpi}
     \end{equation*}
     with some   $M_{\varpi}>0$ depending only on the angle $\varpi$.
     \end{itemize}

When $A$ is a sectorial operator, it generates an analytical semigroup. 
\begin{theorem}
Let $A$ be a sectorial operator. Then, $(-A)$ generates an analytical   semigroup $S(t)=e^{-tA}, 0\leq t<\infty,$ having the following properties:
\begin{itemize}
\item  [\rm (i)] For  any $\theta\geq 0$, there exists $\iota_\theta>0$ such that 
\begin{equation} \label{H10}
\|A^\theta S(t)\| \leq \iota_\theta t^{-\theta}, \hspace{1cm}  0< t< \infty.   
\end{equation}
In particular, there exists $\nu>0$ such that
\begin{equation} \label{H11}
\|S(t)\| \leq \iota_0 e^{-\nu t}, \hspace{1cm}  0< t< \infty.
\end{equation}
 \item  [\rm (ii)]  For any $0<\theta\leq 1,$ 
\begin{equation} \label{H11.5}
\|[S(t)-I]A^{-\theta}\|\leq \frac{\iota_{1-\theta}}{\theta} t^\theta, \hspace{1cm}  0< t< \infty.
\end{equation}
\item  [\rm (iii)]  For  $0<\sigma<\beta< 1$ and $U\in \mathcal D(A^\beta)$, 
\begin{equation} \label{Eq53}
AS(\cdot)U \in \mathcal F^{\beta,\sigma}((0,T];E).
\end{equation}
\end{itemize}
\end{theorem}
For the proof, see  \cite{yagi}.

\subsection{Cylindrical Wiener processes}
Let us review a central notion to the theory of stochastic evolution equations, say  cylindrical Wiener processes on a separable Hilbert space $H$. The following  definition is given by Curtain-Falb \cite{CurtainFalb}. 
\begin{definition} \label{T6}
Let $Q$ be a  symmetric nonnegative nuclear operator in $ L(H).$
An $H$-valued stochastic process $W$ defined on a filtered  probability space  $(\Omega, \mathcal F,\mathcal F_t,\mathbb P)$ is called a $Q$-Wiener process if it satisfies the conditions:
\begin{enumerate}
  \item [(i)] $W(0)=0$  a.s.
  \item [(ii)] $W$ has continuous sample paths.
  \item [(iii)] $W$ has independent increments.
  \item  [(iv)] The law of $W(t)-W(s), 0<s\leq t,$ is a Gaussian measure on $H$ with mean $0$ and covariance $Q.$  
\end{enumerate}
\end{definition}
Note that a nuclear operator  is a bounded linear operator with finite trace.  Here, the trace of  $Q$ is defined by 
$$Tr (Q)=\sum_{i=1}^\infty \langle Qe_i,e_i\rangle,$$
where $\{e_i\}_{i=1}^\infty$ is a complete orthonormal basis in $H$. 

Let  fix a larger Hilbert space $H_1$ such that $H$ is embedded continuously into $H_1$ and the embedding 
$J\colon H\to H_1$ is Hilbert-Schmidt (i.e. $\sum_{i=1}^\infty \|Je_i\|_{H_1}^2<\infty).$ For example (see Hairer \cite{Hairer}), one takes $H_1$ to be the closure of $H$ under the norm
$$\|h\|_{H_1}=\left[\sum_{n=1}^\infty \frac{\langle h,e_n\rangle_H^2}{n^2}\right]^{\frac{1}{2}}.$$
For any $h_1\in H_1$, 
\begin{align*}
\langle JJ^*e_m, h_1\rangle_{H_1}
&=\langle J^*e_m, J^*h_1\rangle_H\\
&=\langle\sum_{k=1}^\infty e_k \langle J^*e_m,e_k\rangle_H, \sum_{k=1}^\infty e_k \langle J^*h_1,e_k\rangle_H\rangle_H\\
&= \sum_{k=1}^\infty \langle J^*e_m,e_k\rangle_H  \langle J^*h_1,e_k\rangle_H\\
&=\sum_{k=1}^\infty \langle e_m,Je_k\rangle_{H_1}  \langle h_1,Je_k\rangle_{H_1}\\
&=\sum_{k=1}^\infty \langle e_m,e_k\rangle_{H_1}  \langle h_1,e_k\rangle_{H_1}\\
&=\|e_m\|_{H_1}^2  \langle h_1,e_m\rangle_{H_1}\\
&= \frac{1}{m^2} \langle h_1,e_m\rangle_{H_1}, \hspace{1.5cm}  m=1,2,\dots
\end{align*}
Thereby, $JJ^*e_m=\frac{1}{m^2} e_m $ for  $m=1,2,\dots$ Therefore,
$$Tr (JJ^*)=\sum_{m=1}^\infty \langle JJ^*e_m,e_m\rangle=\sum_{m=1}^\infty \frac{1}{m^2}<\infty.$$
Thus, $JJ^*$ is a nuclear operator. The following definition is taken from Da Prato-Zabczyk \cite{prato} (see also Hairer \cite{Hairer}).
\begin{definition} \label{T7}
An $H_1$-valued $JJ^*$-Wiener process   is called a cylindrical Wiener process on $H$.
\end{definition}
\begin{remark}
There is another way to define cylindrical Wiener processes. One can define  a cylindrical Wiener process as the canonical process for any Gaussian measure with Cameron-Martin space. For more details, see Hairer  \cite{Hairer}.
\end{remark}

\subsection{Stochastic integrals in Banach spaces of M-type 2}
\begin{definition}[Pisier \cite{Pisier}]  \label{MType2BanachSpace}
A Banach space $E$ is said to be of martingale type 2 (or M-type $2$ for abbreviation),  if there is a constant $c(E)$ such that for all $E$-valued martingales $\{M_n\}_n,$ it holds true that 
$$\sup_n \mathbb E\|M_n\|^2 \leq c(E) \sum_{n\geq 0}\mathbb E \|M_n-M_{n-1}\|^2,$$
where $M_{-1}=0$.
\end{definition}
It is known that  Hilbert spaces are of M-type $2$ and that, when $2\leq p<\infty,$
the $L^p$ space is the same.

When $E$ is of M-type 2, stochastic integrals  with respect to  cylindrical Wiener processes  can be constructed in a quite similar way as for the  usual  It\^{o}  integrals.  There are several literature dealing with this subject (see Brze\'{z}niak \cite{Brzezniak,Brzezniak1}, Dettweiler \cite{Dettweiler,Dettweiler3}, 
 van Neerven et al. \cite{Haak,van2}). Let us explain the construction.

Let $W$ be a cylindrical Wiener process on a separable Hilbert space $H$, which is defined on a  filtered  probability space  $(\Omega, \mathcal F,\mathcal F_t,\mathbb P)$.
\begin{definition}  
Let $(S,\Sigma)$ be a measurable space. 
\begin{itemize}
  \item [(i)] A function $\varphi\colon S\to E$ is said to be strongly measurable if it is the pointwise limit of a sequence of simple functions. 
  \item [(ii)] A  function $\phi\colon S\to L(H;E)$  is said to be $H$-strongly measurable if $\phi(\cdot)h  \colon S\to E$ is strongly measurable for any $h\in H$.
\end{itemize}  
\end{definition}
\begin{definition}[$\gamma$-radonifying operators] \label{Def10}
\begin{itemize}
  \item [(i)]
Let $\{\gamma_n\}_{n=1}^\infty$ be a sequence of independent standard Gaussian random variables on a probability space $(\Omega', \mathcal F',\mathbb P')$. Let $\{e_n\}_{n=1}^\infty$ be an orthonormal basis of $H$. A bounded linear operator $\phi$ from $H$ to $E$ is called a $\gamma$-radonifying operator if the Gaussian series $\sum_{n=1}^\infty \gamma_n \phi e_n$ converges in $L^2(\Omega',E)$.
 \item [(ii)] 
The set of all $\gamma$-radonifying operators is denoted by $\gamma(H;E)$. 
\end{itemize}
\end{definition}

According to Definition \ref{Def10}, a  natural norm in $\gamma(H;E)$ is defined by 
$$\|\phi\|_{\gamma(H;E)}=\Big[\mathbb E'\Big\|\sum_{n=1}^\infty \gamma_n \phi e_n\Big\|^2\Big]^{\frac{1}{2}}.$$
It is easily seen that the norm is independent of the orthonormal basis $\{e_n\}_{n=1}^\infty$ and the Gaussian sequence $\{\gamma_n\}_{n=1}^\infty$. Furthermore, the normed space $(\gamma(H;E), \|\cdot\|_{\gamma(H;E)})$ is complete (see \cite{van1,van2}). The following lemma is  used very often.
\begin{lemma}  \label{T2.12}
Let $\phi_1\in L(E)$ and $\phi_2 \in \gamma(H;E)$. Then, $\phi_1\phi_2 \in \gamma(H;E)$ and
$$\|\phi_1\phi_2\|_{\gamma(H;E)} \leq \|\phi_1\|_{L(E)} \|\phi_2\|_{\gamma(H;E)}.$$
\end{lemma}
\begin{definition}  \label{Def11}
Denote by $\mathcal I^2([0,T])$ the class of   all adapted processes $\phi\colon [0,T]\times \Omega \to \gamma (H;E)$ in $L^2((0,T)\times \Omega; \gamma (H;E)),$ which are $H$-strongly measurable. 
\end{definition}

For every $\phi \in \mathcal I^2([0,T]),$ the stochastic integral $\int_0^T \phi(t)dW(t)$  is defined as a limit of integrals of adapted step processes. By a localization argument one can extend stochastic integrals  to the class $\mathcal I([0,T])$ of all $H$-strongly measurable and adapted processes $\phi\colon [0,T]\times \Omega \to \gamma (H;E)$ which are  in $L^2((0,T); \gamma (H;E))$ a.s. (see \cite{van1,van2}).
 \begin{theorem}\label{T9}
Let $E$ be a Banach space of M-type 2. Let $W$ be a cylindrical Wiener process on a separable Hilbert  space $H$. Then, there exists  $c(E)>0$ depending only on $E$ such that
$$\mathbb E\Big\|\int_0^T \phi(t)dW(t)\Big\|^2\leq c(E) \|\phi\|_{L^2((0,T)\times \Omega; \gamma (H;E))}^2, \hspace{1cm} \phi \in \mathcal I^2([0,T]),$$
here $\|\phi\|_{L^2((0,T)\times \Omega; \gamma (H;E))}^2=\mathbb E \int_0^T \|\phi(s)\|_{\gamma (H;E)}^2ds.$
In addition, for every $\phi\in \mathcal I([0,T]),$
\begin{itemize}
  \item [{\rm (i)}]  $\{\int_0^t \phi(s)dW(s), 0\leq t\leq T\}$ is an $E$-valued  continuous local martingale and  a Gaussian process.
  \item [{\rm (ii)}] (Burkholder-Davis-Gundy inequality) For any $p>1, $ 
$$\mathbb E \sup_{t\in [0,T]} \Big\|\int_0^t \phi(s) dW(s)\Big\|^p \leq  c_p(E) \mathbb E\Big[\int_0^T \|\phi(s)\|_{\gamma (H;E)}^2 ds\Big]^{\frac{p}{2}},$$
where $c_p(E)$ is some constant depending only on  $p$ and  $ E$. 
\end{itemize}
\end{theorem}
For the proof, see, e.g.,  \cite{Brzezniak1,Brzezniak1.5,Martin}.
\begin{lemma}  \label{T10}
Let $B$ be a closed linear operator in $E$ and $\phi\colon [0,T]\subset \mathbb R\to \gamma (H;E).$
  If $\phi$ and $B\phi$  belong to $  \mathcal I^2([0,T]), $   then
$$B\int_0^T \phi(t)dW(t)=\int_0^T B\phi(t)dW(t)  \hspace{1cm}  \text{  a.s.}$$
\end{lemma}
The proof of Lemma \ref{T10} is very similar to one in Da Prato-Zabczyk  \cite{prato}. So, we omit it.

The Kolmogorov continuity theorem gives a sufficient condition for a stochastic process to be H\"{o}lder continuous.
\begin{theorem} \label{Th-1}
Let $\zeta$ be  an $E$-valued stochastic process on $[0,T].$ Assume that  for some  $c>0$ and $ \epsilon_i>0 \, (i=1,2),$ 
\begin{equation} \label{E11}
\mathbb E\|\zeta(t)-\zeta(s)\|^{\epsilon_1}\leq c |t-s|^{1+\epsilon_2}, \hspace{1cm}  0\leq s,t\leq T.
\end{equation}
Then, $\zeta$ has  a version whose $\mathbb P$-almost all trajectories are H\"{o}lder continuous functions with an  arbitrarily smaller  exponent than $\frac{\epsilon_2}{\epsilon_1}$.
\end{theorem}

When  $\zeta$ is a Gaussian process, one can weaken the condition \eqref{E11}.
\begin{theorem}  \label{Th0}
Let $\zeta$ be an $E$-valued   Gaussian process on $[0,T]$ such that $\mathbb E \zeta(t)=0$ for  $ t\geq 0$.  Assume that  for some  $c>0$ and $ 0<\epsilon\leq 1,$
$$
\mathbb E\|\zeta(t)-\zeta(s)\|^2\leq c (t-s)^\epsilon, \hspace{1cm}  0\leq s\leq t\leq T.
$$
Then, there exists a modification of $\zeta$ whose $\mathbb P$-almost all trajectories are H\"{o}lder continuous functions with an  arbitrarily smaller  exponent than $\frac{\epsilon}{2}$.
\end{theorem}
For the proofs of Theorems \ref{Th-1} and \ref{Th0}, see, e.g.,  \cite{prato}.  

\subsection{Strict and mild   solutions}
Let us restate the problem  \eqref{E2}. Throughout this paper, we  consider \eqref{E2} in a complex separable Banach space $E$ of M-type 2, where
\begin{itemize}
  \item [\rm{(i)}] $A$ is a densely defined, closed linear operator  in $E$.
  \item [\rm{(ii)}]  $W$ is a  cylindrical Wiener process  on a separable Hilbert space $H,$ and is defined on a complete filtered  probability space  $(\Omega, \mathcal F,\mathcal F_t,\mathbb P).$
   \item [\rm{(iii)}] $F_1$ is  a measurable process  from $(\Omega_T, \mathcal P_T)$ to $(E,\mathcal B(E)).$
  \item  [\rm{(iv)}] $F_2$ is a measurable function from $(\Omega\times E, \mathcal F\times \mathcal B(E))$ to $(E,\mathcal B(E))$.
   \item  [\rm{(v)}] $G\in \mathcal I^2([0,T]),$ where $\mathcal I^2([0,T])$ is given by Definition \ref{Def11}.
  \item [\rm{(vi)}] $\xi$ is an $E$-valued $\mathcal F_0$-measurable random variable.
\end{itemize}  

\begin{definition}\label{Def2}
A adapted $E$-valued continuous process $X$ on $ [0,T]$ is called a strict solution of \eqref{E2} if  
$$\int_0^T \|F_2(X(s))\|ds <\infty  \hspace{1cm} \text{a.s.,} $$
$$X(t)\in \mathcal D(A) \quad \text{  and } \quad \Big \|\int_0^t AX(s) ds \Big \|<\infty  \hspace{1cm} \text{a.s., } \, 0<t\leq  T,$$
and  
\begin{align*}
X(t)=&\xi -\int_0^t AX(s)ds+\int_0^t [F_1(s)+F_2(X(s))]ds\\
&+ \int_0^t  G(s)dW(s) \hspace{1cm} \text{a.s., } \, 0<t\leq  T.
\end{align*}
\end{definition}

\begin{definition}\label{Def1}
A adapted $E$-valued continuous process $X$ on $ [0,T]$ is called a  mild solution  of \eqref{E2} if 
$$\int_0^T \|S(t-s)F_2(X(s))\| ds<\infty  \hspace{1cm} \text{a.s.,}$$
and  
\begin{align*}
 X(t)=&S(t)\xi +\int_0^tS(t-s) [F_1(s)+F_2(X(s))]ds\\
&+ \int_0^t S(t-s) G(s)dW(s) \hspace{1cm} \text{a.s., } \, 0<t\leq  T.\notag
\end{align*}
\end{definition}

 A strict (mild) solution $X$ on $[0,T]$ is said to be unique if any other strict (mild) solution $ \bar X$ on $[0,T]$ is indistinguishable from it, i.e. 
$$\mathbb P\{X(t)=\bar X(t) \text { for every } t\in [0,T]\}=1.$$

\begin{remark}
A strict solution is a mild solution (\cite{prato}). The inverse is however not true in general.
\end{remark}

\section{Linear evolution equations} \label{section4}
In this section, we  handle the linear case of the problem  \eqref{E2}, i.e.  $F_2\equiv 0$ in $E$.
Let us rewrite  \eqref{E2}  as 
\begin{equation} \label{linear}
\begin{cases}
dX+AXdt=F_1(t)dt+ G(t)dW(t), \hspace{1cm} 0<t\leq T, \\
X(0)=\xi.
\end{cases}
\end{equation}
First, we  study strict solutions to  \eqref{linear}, where $F_1$ and $G$ are taken from weighted H\"{o}lder continuous function spaces. This kind of  solutions is  constructed in Theorem \ref{Th2}  on the basis of solution formula. Maximal regularity of strict solutions  are shown in Theorem \ref{Th3}. Second, we give  maximal regularity of mild solutions to  \eqref{linear}  in Theorem \ref{Th4}.

Suppose that $F_1$ belongs to a weighted H\"{o}lder continuous function space:
\begin{itemize}
   \item [(F1)]   \hspace{0.3cm} For some $0<\sigma<\beta< \frac{1}{2},$  
    \begin{equation*} 
        F_1\in \mathcal F^{\beta, \sigma}((0,T];E)  \hspace{1cm} \text{ a.s. and }  \hspace{0.3cm} \mathbb E\|F_1\|_{\mathcal F^{\beta, \sigma}(E)}^2<\infty.
                                   \end{equation*}
                          \end{itemize}
And $G$ satisfies one of the   following conditions:
\begin{itemize}
   \item [(Ga)]    \hspace{0.3cm}  For some $1-\beta<\delta\leq 1$, 
       $$  
      A^\delta G\in \mathcal F^{\beta+\frac{1}{2}, \sigma} ((0,T];\gamma(H;E))  \hspace{1cm} \text{  a.s.} $$
and 
  $$ \mathbb E\|A^\delta G\|_{\mathcal F^{\beta+\frac{1}{2}, \sigma}(\gamma(H;E))}^2<\infty.  $$
                           \end{itemize}
\begin{itemize}
   \item [(Gb)]   \hspace{1cm}  
    $ 
       G\in \mathcal F^{\beta+\frac{1}{2}, \sigma} ((0,T];\gamma(H;E)) $  \hspace{1cm}  a.s. and   
    $$\mathbb E\|G\|_{\mathcal F^{\beta+\frac{1}{2}, \sigma}(\gamma(H;E))}^2<\infty.$$
              \end{itemize}
\begin{remark} 
The spatial regularity of $G$ in {\rm (Ga)} is essential for our construction of strict solutions. For existence of mild solutions, it is not necessary. In fact, for mild solutions, only temporal regularity, say {\rm (Gb)}, is required.
\end{remark}

\subsection{Strict solutions and  maximal regularity}
First, let us prove  uniqueness of strict solutions.
\begin{theorem}[Uniqueness]  \label{Th1}
Let $A$ satisfy  {\rm (Aa)} and  {\rm (Ab)}. If there exists a strict solution to  \eqref{linear},  then it is unique.
 \end{theorem}
\begin{proof}
Let $X$ and $\bar X$  be two strict solutions of  \eqref{linear} on $[0,T]$. Put $Y(t)=X(t)-\bar X(t), 0\leq t\leq T$. From the definition of strict solutions, we observe that 
$$
\begin{cases}
Y(t)=-\int_0^t AY(s)ds, \hspace{1cm} 0<t\leq T,\\
Y(0)=0.
\end{cases}
$$
Since $\frac{dY}{dr}=-AY(r)$ and $\frac{dS(t-r)}{dr}=AS(t-r)$ for  $0<r<t\leq T$, 
it is easily seen that 
$$
 \frac{d S(t-r)Y(r) }{dr}=0, \hspace{1cm} 0<r<t\leq T.
$$
This implies that $S(t-\cdot)Y(\cdot)$ is a constant on $(0,t)$. Hence, 
$$S(t-r)Y(r)=\lim_{r\to t} S(t-r)Y(r)=0,\hspace{1cm}  0<r<t\leq T.$$

Since $Y$ is continuous at $t=0$ and $Y(0)=0$, 
$$Y(r)=\lim_{t\to r}S(t-r)Y(r)=Y(r),  \hspace{1cm}  0<r\leq T.$$
Thus, $X=\bar X$ a.s.  on $[0,T]$. By the continuity, we conclude that $X$ and $\bar X$ are indistinguishable.
\end{proof}

Second, let us construct strict solutions on the basis of solution formula.
\begin{theorem}[Existence] \label{Th2}
Let {\rm (Aa)}, {\rm (Ab)},  {\rm (F1)} and {\rm (Ga)} be satisfied.
  Then, there exists a unique strict solution $X$ of \eqref{linear}.  Furthermore, $X$ possesses  the regularity
 $$AX\in \mathcal C((0,T];E) \hspace{1cm} \text{a.s.}$$
and satisfies the estimate
\begin{align}
\mathbb E & \|X(t)\|^2  + t^2 \mathbb E \|AX(t)\|^2 \label{H12.7}\\
 \leq &  C[\mathbb E \|\xi\|^2 +      \mathbb E\|F_1\|_{\mathcal F^{\beta,\sigma}(E)}^2   t^{2\beta}+\mathbb E\|A^\delta G\|_{\mathcal F^{\beta+\frac{1}{2},\sigma}(\gamma(H;E))}^2    \notag\\
& \times  \{      t^{2\beta}  +t^{2(\beta+\delta)}   \}    ], \hspace{1cm} 0\leq t\leq T,     \notag  
\end{align}
where $C>0$ is some constant depending only on the exponents.
\end{theorem}

\begin{proof}
Uniqueness of strict solutions has already been verified by Theorem \ref{Th1}. What we need is to construct a strict solution to   \eqref{linear} on the basis of solution formula. 
For this purpose, we divide the proof into several steps.

{\bf Step 1}. Put 
\begin{equation}\label{Pt9}
I_1(t)=S(t)\xi +\int_0^tS(t-s) F_1(s)ds.
\end{equation}
 Let us verify that 
$$I_1\in  \mathcal C((0,T];\mathcal D(A))  \cap  \mathcal C([0,T];E) $$ 
and $I_1$ satisfies the integral equation
\begin{equation}  \label{I1equation}
I_1(t)+\int_0^t AI_1(s)ds=\xi+ \int_0^t F_1(s)ds, \hspace{1cm} 0<t\leq T.  
\end{equation}

Let $A_n=A(1+\frac{A}{n})^{-1}, n=1,2,3,\dots,$  be the Yosida approximation of $A$. Then, $A_n$ satisfies {\rm (Aa)} and {\rm (Ab)} uniformly and generates an analytic semigroup $S_n(t)$ (see  \cite{yagi}). Furthermore, for any $0\leq \nu <\infty$, 
\begin{equation} \label{LimitOfAnnuSn}
\lim_{n\to\infty} A_n^\nu S_n(t)=A^\nu S(t)  \hspace{1cm}\text{  (limit in    }  L(E)).
\end{equation}
In addition,  there exists $\varsigma_\nu>0$ independent of $n$ such that 
\begin{equation} \label{EstimateIofAnnu}
\begin{aligned}
\begin{cases}
\|A_n^\nu S_n(t)\| \leq \varsigma_\nu t^{-\nu} &\hspace{1cm}\text { if } \nu>0, 0<t\leq T\\
\|A_n^\nu S_n(t)\| \leq \varsigma_\nu e^{-\varsigma_\nu t} &\hspace{1cm} \text { if } \nu=0, 0<t\leq T.
\end{cases}
\end{aligned}
\end{equation}

Consider the function
$$I_{1n}(t)=S_n(t)\xi +\int_0^tS_n(t-s) F_1(s)ds, \hspace{1cm} 0\leq t \leq T.$$
Due to \eqref{LimitOfAnnuSn} and \eqref{EstimateIofAnnu}, 
$$\lim_{n\to \infty} I_{1n}(t)=I_1(t), \hspace{1cm} 0\leq t\leq T. $$
  Since $A_n$ is bounded, we observe that
$$dI_{1n}=[-A_n I_{1n}+F_1(t)]dt, \hspace{1cm} 0<t\leq T.$$
From this equation,  for any $0<\epsilon \leq T,$
\begin{equation}\label{EquationOfI1n}
I_{1n}(t)=I_{1n}(\epsilon)+\int_\epsilon^t[F_1(s)-A_nI_{1n}(s)]ds, \hspace{1cm}   0<\epsilon\leq t\leq T.
\end{equation}

We   observe  convergence of $A_nI_{1n}$. We have
\begin{align} 
A_nI_{1n}(t) 
=&A_nS_n(t)\xi+\int_0^t A_nS_n(t-s)[F_1(s)-F_1(t)]ds \notag\\
&+ \int_0^t A_nS_n(t-s)ds F_1(t)\notag\\
=&A_nS_n(t)\xi+\int_0^t A_nS_n(t-s)[F_1(s)-F_1(t)]ds + [I-S_n(t)]F_1(t). \label{Eq36}
\end{align}
Thanks to  \eqref{H12} and \eqref{EstimateIofAnnu}, 
\begin{align} 
\|A_nI_{1n}(t)\|  
\leq &\varsigma_1 t^{-1} \|\xi\|+\int_0^t  \varsigma_1  \|F_1\|_{\mathcal F^{\beta,\sigma}(E)} (t-s)^{\sigma-1} s^{\beta-\sigma-1} ds\notag\\
&+(1+\varsigma_0 e^{-\varsigma_0 t}) \|F_1\|_{\mathcal F^{\beta,\sigma}(E)} t^{\beta-1}\notag\\
=&\varsigma_1 \|\xi\| t^{-1} + [1+\varsigma_1  B(\beta-\sigma, \sigma)   +\varsigma_0 e^{-\varsigma_0 t}         ]       \label{H12.6}\\
&\hspace{1cm}  \times \|F_1\|_{\mathcal F^{\beta,\sigma}(E)} t^{\beta-1}, \hspace{1cm} 0<t\leq T, \notag
\end{align}
where $B (\cdot,\cdot)$ is the beta function. 
Furthermore, due to \eqref{LimitOfAnnuSn} and \eqref{Eq36}, we have 
\begin{equation}  \label{H13.1}
\lim_{n\to\infty}A_nI_{1n}(t)=Y(t),
\end{equation}
where
$$Y(t)=AS(t)\xi+\int_0^t AS(t-s)[F_1(s)-F_1(t)]ds + [I-S(t)]F_1(t).$$

Let us verify  that the function $Y$ is continuous  on $(0,T]$. Take $0<t_0\leq T$. By using \eqref{H12} and \eqref{H10},  for  $t\geq t_0$, we have
\begin{align}
&\|Y(t)-Y(t_0)\|\notag\\
\leq & \|AS(t_0)[S(t-t_0)-I]\xi\| + \|[I-S(t)]F_1(t)-[I-S(t_0)]F_1(t_0)\|\notag \\
&+\Big\| \int_{t_0}^t AS(t-s)[F_1(s)-F_1(t)]ds\notag \\
&+\int_0^{t_0} AS(t-s)ds[F_1(t_0)-F_1(t)]\notag\\
&+\int_0^{t_0} S(t-t_0)AS(t_0-s)[F_1(s)-F_1(t_0)]ds\notag\\
&-\int_0^{t_0} AS(t_0-s)[F_1(s)-F_1(t_0)]ds\Big\|\notag\\
\leq & \iota_1 t_0^{-1}\|S(t-t_0)\xi-\xi\| + \|[I-S(t)]F_1(t)-[I-S(t_0)]F_1(t_0)\|\notag \\
&+ \int_{t_0}^t \|AS(t-s)\| \|F_1(s)-F_1(t)\|ds\notag \\
&+\|[S(t-t_0)-S(t)][F_1(t_0)-F_1(t)]\|\notag\\
&+\int_0^{t_0} \|[S(t-t_0)-I]AS(t_0-s)[F_1(s)-F_1(t_0)]\|ds\notag\\
\leq & \iota_1 t_0^{-1}\|S(t-t_0)\xi-\xi\| + \|[I-S(t)]F_1(t)-[I-S(t_0)]F_1(t_0)\| \notag\\
&+ \int_{t_0}^t \iota_1  \|F_1\|_{\mathcal F^{\beta, \sigma}(E)} (t-s)^{\sigma-1} s^{\beta-\sigma-1}ds\notag \\
&+\|S(t-t_0)-S(t)\| \|F_1(t_0)-F_1(t)\|\notag\\
&+\|S(t-t_0)-I\|\int_0^{t_0}  \iota_1\|F_1\|_{\mathcal F^{\beta, \sigma}(E)} (t_0-s)^{\sigma-1} s^{\beta-\sigma-1}ds\notag\\
\leq  & \iota_1 t_0^{-1}\|S(t-t_0)\xi-\xi\| + \|[I-S(t)]F_1(t)-[I-S(t_0)]F_1(t_0)\|\notag \\
&+ \varsigma_1  \|F_1\|_{\mathcal F^{\beta, \sigma}(E)}  t_0^{\beta-\sigma-1} \int_{t_0}^t(t-s)^{\sigma-1} ds\notag \\
&+\|S(t-t_0)-S(t)\| \|F_1(t_0)-F_1(t)\|\notag\\
&+  \iota_1\|F_1\|_{\mathcal F^{\beta, \sigma}(E)} B(\beta-\sigma,\sigma) t_0^{\beta-1} \|S(t-t_0)-I\|\notag\\
\leq  & \iota_1 t_0^{-1}\|S(t-t_0)\xi-\xi\| + \|[I-S(t)]F_1(t)-[I-S(t_0)]F_1(t_0)\| \label{Eq26}\\
&+ \frac{\iota_1  \|F_1\|_{\mathcal F^{\beta, \sigma}(E)}}{\sigma}  t_0^{\beta-\sigma-1} (t-t_0)^\sigma\notag \\
&+\|S(t-t_0)-S(t)\| \|F_1(t_0)-F_1(t)\|\notag\\
&+  \iota_1\|F_1\|_{\mathcal F^{\beta, \sigma}(E)} B(\beta-\sigma,\sigma) t_0^{\beta-1} \|S(t-t_0)-I\|.\notag
\end{align}
Thus, 
$$\lim_{t\searrow  t_0}Y(t)=Y(t_0).$$ 
Similarly, we obtain that 
$$\lim_{t\nearrow  t_0}Y(t)=Y(t_0).$$
 Hence, $Y$ is continuous at $t=t_0$ and therefore at every point in $(0,T]$.

By the above arguments, we conclude that
\begin{align*}
I_1(t)&=\lim_{n\to\infty} I_{1n}(t)=\lim_{n\to\infty} A_n^{-1} A_n I_{1n}(t) \notag\\
&=A^{-1}Y(t), \hspace{1cm} 0<t\leq T.     
\end{align*}
This shows that $I_1(t) \in \mathcal D(A)$ for  $0<t\leq T$ and 
\begin{align}
AI_1=Y\in \mathcal C((0,T];E).    \label{H13.2}
\end{align}

 Furthermore, since $Y$ is continuous on $(0,T],$ $I_1=A^{-1} Y$ is the same. The continuity of $I_1$ at $t=0$ can be seen as follows. By \eqref{H12} and  \eqref{H11}, 
\begin{align}
\int_0^t\|S(t-s) F_1(s)\|ds& \leq \int_0^t\|S(t-s)\| \|F_1(s)\|ds  \notag\\
&\leq  \iota_0 \|F_1\|_{\mathcal F^{\beta,\sigma}(E)} \int_0^t  s^{\beta-1}ds \notag\\
&\leq \frac{\iota_0 \|F_1\|_{\mathcal F^{\beta,\sigma}(E)} t^\beta}{\beta}.    \label{Eq0}
\end{align}
So, 
$$\lim_{t\to 0} \int_0^tS(t-s) F_1(s)ds=0.$$
This and \eqref{Pt9} imply  that 
$$\lim_{t\to 0} I_1(t)=\lim_{t\to 0} S(t)\xi=\xi=I_1(0).$$
Hence, 
$$I_1 \in \mathcal C([0,T];E).$$

It now remains to prove  \eqref{I1equation}. It is seen that  the integral $\int_0^t AI_1(s)ds$ exists. Indeed, by virtue of  \eqref{H12},  \eqref{H10} and \eqref{H11},  
\begin{align*}
\int_0^t \|[I-S(r)]F_1(r)\|dr &\leq \int_0^t \|I-S(r)\| \|F_1(r)\|dr \\
&\leq (1+\iota_0) \|F_1\|_{\mathcal F^{\beta,\sigma}(E)}\int_0^t r^{\beta-1}dr\\
&=\frac{(1+\iota_0) \|F_1\|_{\mathcal F^{\beta,\sigma}(E)} r^{\beta}}{\beta}, \hspace{1cm} 0\leq t\leq T,
\end{align*}
and
\begin{align*}
&\int_0^t \Big\|\int_0^s AS(s-r)[F_1(r)-F_1(s)]dr\Big\| ds\\
&\leq \int_0^t \int_0^s \|AS(s-r)\| \|F_1(r)-F_1(s)\|dr ds\\
&\leq \iota_1 \|F_1\|_{\mathcal F^{\beta, \sigma}(E)} \int_0^t \int_0^s   (s-r)^{\sigma-1} r^{\beta-\sigma-1}dr ds\\
&=\frac{\iota_1 \|F_1\|_{\mathcal F^{\beta, \sigma}(E)} B(\beta-\sigma, \sigma)  t^\beta}{\beta}, \hspace{1cm} 0\leq t\leq T.
\end{align*}
These  estimates show that the integrals $\int_0^t [I-S(r)]F_1(r)dr$ and $\int_0^t \int_0^s AS(s-r)[F_1(r)-F_1(s)]dr ds $ are finite  for  $0\leq t\leq T.$ The finiteness of the integral $\int_0^t AI_1(s)ds$,  $0\leq t\leq T,$ hence follows from the equality:
\begin{align*}
\int_0^t AI_1(s)ds=&\int_0^t Y(s)ds\\
=&\int_0^t AS(r)\xi dr+\int_0^t \int_0^s AS(s-r)[F_1(r)-F_1(s)]dr ds \\
&+\int_0^t [I-S(r)]F_1(r)dr\\
=&[I-S(t)]\xi+\int_0^t \int_0^s AS(s-r)[F_1(r)-F_1(s)]dr ds\\
& +\int_0^t [I-S(r)]F_1(r)dr.
\end{align*}

Thanks to \eqref{H12.6},   the Lebesgue dominate convergence theorem applied to \eqref{EquationOfI1n} provides that  
 \begin{equation} \label{Eq27}
I_1(t)=I_1(\epsilon)+\int_\epsilon^t[F_1(s)-AI_1(s)]ds, \hspace{1cm}  0<\epsilon\leq  t\leq T.
\end{equation}
On the other hand, \eqref{Pt9} and  \eqref{Eq0} give that
$$\lim_{\epsilon\to 0} I_1(\epsilon) =\xi.$$
 Letting $\epsilon \to 0$ in \eqref{Eq27},  we arrive at the equation \eqref{I1equation}.

{\bf Step 2}. Put 
\begin{equation}\label{Pt10}
I_2(t)= \int_0^t S(t-s) G(s)dW(s).
\end{equation}
Let us manifest  that  the stochastic convolution  $I_2$ is well-defined  on $[0,T]$. Indeed, 
by using \eqref{H12},  \eqref{H11}  and (Ga),  
\begin{align}
&\int_0^t  \mathbb E \|S(t-s)G(s)\|_{\gamma(H;E)}^2 ds   \notag\\
& \leq 
\int_0^t\|A^{-\delta}S(t-s)\|\mathbb E \|A^\delta G(s)\|_{\gamma(H;E)}^2 ds  \notag\\
&\leq  \int_0^t  \iota_0^2 \|A^{-\delta}\| \mathbb E\|A^\delta G\|_{\mathcal F^{\beta+\frac{1}{2},\sigma}(\gamma(H;E))}^2 s^{2\beta-1}ds    \notag\\
&=\frac{\iota_0^2 \|A^{-\delta}\| \mathbb E\|A^\delta G\|_{\mathcal F^{\beta+\frac{1}{2},\sigma}(\gamma(H;E))}^2 t^{2\beta}}{2\beta}<\infty, \hspace{1cm} 0\leq t\leq T.    \label{H13.3}
\end{align}
Hence,  $I_2$ is well-defined on $[0,T]$.

{\bf Step 3}. 
Let us observe that 
$I_2$  satisfies the equation
\begin{equation} \label{I2}
I_2(t)+ \int_0^t AI_2(s)ds=\int_0^t G(s)dW(s), \hspace{1.5cm}   0\leq t\leq T,
\end{equation}
and for any  $0<\gamma<\beta+\delta-1$,
$$AI_2\in \mathcal C^\gamma([0,T];E) \hspace{1cm} \text{ a.s.  } $$ 
In particular, $AI_2$ and $I_2$ belong to $\mathcal C([0,T];E)$ a.s.

First, we  prove \eqref{I2}.
It follows from \eqref{H10}  and (Ga) that
\begin{align}
&\int_0^t\mathbb E \|AS(t-s)G(s)\|_{\gamma(H;E)}^2 ds    \notag\\
&\leq \int_0^t\|A^{1-\delta}S(t-s)\|^2 \mathbb E \|A^\delta G(s)\|_{\gamma(H;E)}^2 ds   \notag\\
& \leq   \int_0^t  \iota_{1-\delta}^2 (t-s)^{2(\delta-1)} \mathbb E \|A^\delta G\|_{\mathcal F^{\beta+\frac{1}{2}, \sigma}(\gamma(H;E))}^2 s^{2\beta-1}ds   \notag\\
&=\iota_{1-\delta}^2 \mathbb E \|A^\delta G\|_{\mathcal F^{\beta+\frac{1}{2}, \sigma}(\gamma(H;E))}^2  B(2\beta,2\delta-1)  t^{2(\beta+\delta-1)}     \label{H13.4}   \\
&<\infty, \hspace{4cm} 0\leq t\leq T,      \notag
\end{align}
and  
\begin{align*}
&\int_0^t\mathbb E \|G(s)\|_{\gamma(H;E)}^2 ds    \notag\\
&\leq \int_0^t\|A^{-\delta}\|^2 \mathbb E \|A^\delta G(s)\|_{\gamma(H;E)}^2 ds    \notag\\
& \leq  \|A^{-\delta}\|^2 \int_0^t  \mathbb E \|A^\delta G\|_{\mathcal F^{\beta+\frac{1}{2}, \sigma}(\gamma(H;E))}^2 s^{2\beta-1}ds    \notag\\
&=\frac{\|A^{-\delta}\|^2 \mathbb E \|A^\delta G\|_{\mathcal F^{\beta+\frac{1}{2}, \sigma}(\gamma(H;E))}^2    t^{2\beta}}{2\beta}<\infty, \hspace{1cm} 0\leq t\leq T.   
\end{align*}
The stochastic integrals $\int_0^t AS(t-s)G(s)dW(s)$ and $\int_0^t G(s)dW(s)$ are therefore  well-defined 
on $[0,T]$. 
 Lemma \ref{T10} then provides that 
$$AI_2(t)=\int_0^t AS(t-s)G(s)dW(s), \hspace{1cm} 0\leq t\leq T.$$
Using the Fubini theorem, we have
\begin{equation*}
\begin{aligned}
A \int_0^t I_2(s)ds&=\int_0^t \int_0^s AS(s-u) G(u)dW(u)ds\\
&=\int_0^t \int_u^t AS(s-u) G(u)dsdW(u)\\
&=\int_0^t [G(u)-S(t-u)G(u)]dW(u)\\
&=\int_0^t G(u)dW(u)-\int_0^t S(t-u)G(u)dW(u)\\
&=\int_0^t G(u)dW(u)-I_2(t),  \hspace{1cm}  0\leq t \leq T.
\end{aligned}
\end{equation*}
This means that  $I_2$ satisfies \eqref{I2}. 

Next, we prove the H\"{o}lder continuity of $AI_2$ on $[0,T]$
 by using 
 the expression: 
\begin{align*}
AI_2(t)&-AI_2(s)=\int_s^t AS(t-r)G(r)dW(r) \\
&+\int_0^s A[S(t-r)-S(s-r)]G(r)dW(r), \hspace{1cm}   0\leq s<t\leq T.
\end{align*}
Since the integrals in the right-hand side of the latter equality are independent and have  zero expectation,  
\begin{align*}
\mathbb E& \|AI_2(t)-AI_2(s)\|^2  \notag\\
=&\mathbb E \Big\|\int_s^t A^{1-\delta}S(t-r)A^\delta G(r)dW(r)\Big\|^2 \notag \\
     &+\mathbb E\Big\|\int_0^s A^{1-\delta}[S(t-r)-S(s-r)]A^\delta G(r)dW(r)\Big\|^2  \notag\\
\leq & c(E) \int_s^t \mathbb E\|A^{1-\delta} S(t-r)A^\delta G(r)\|_{\gamma(H;E)}^2dr \notag \\
    &+c(E) \int_0^s \mathbb E\|A^{1-\delta} [S(t-r)-S(s-r)]A^\delta G(r)\|_{\gamma(H;E)}^2dr  \notag\\
\leq & c(E) \int_s^t  \|A^{1-\delta} S(t-r)\|^2  \mathbb E\|A^\delta G(r)\|_{\gamma(H;E)}^2 dr\notag\\
&+ c(E) \int_0^s  \|A^{1-\delta}[S(t-r)-S(s-r)]\|^2 \\
& \hspace{1.5cm} \times \mathbb E\|A^\delta G(r)\|_{\gamma(H;E)}^2dr, \hspace{1cm}   0\leq s<t\leq T.  \notag
\end{align*}
Then,   \eqref {H12},  \eqref{H10} and  \eqref{H11} give 
\begin{align*}
\mathbb E& \|AI_2(t)-AI_2(s)\|^2  \notag\\
      \leq & \iota_{1-\delta}^2 \mathbb E\|A^\delta G\|_{\mathcal F^{\beta+\frac{1}{2},\sigma}(\gamma(H;E))}^2  \int_s^t (t-r)^{2(\delta-1)}r^{2\beta-1}dr\notag\\
        &+ \mathbb E\|A^\delta G\|_{\mathcal F^{\beta+\frac{1}{2},\sigma}(\gamma(H;E))}^2 \int_0^s\Big\|\int_{s-r}^{t-r} A^{2-\delta} S(u) du\Big\|^2 r^{2\beta-1}dr   \notag\\
\leq &\iota_{1-\delta}^2 \mathbb E\|A^\delta G\|_{\mathcal F^{\beta+\frac{1}{2},\sigma}(\gamma(H;E))}^2  \int_s^t (t-r)^{2(\delta-1)}(r-s)^{2\beta-1}dr\notag\\
     &  +\iota_{2-\delta}^2 \mathbb E\|A^\delta G\|_{\mathcal F^{\beta+\frac{1}{2},\sigma}(\gamma(H;E))}^2 \int_0^s \Big(\int_{s-r}^{t-r} u^{\delta-2}du\Big)^2 r^{2\beta-1}dr  \notag \\
= &\iota_{1-\delta}^2 \mathbb E\|A^\delta G\|_{\mathcal F^{\beta+\frac{1}{2},\sigma}(\gamma(H;E))}^2  \int_s^t (t-r)^{2(\delta-1)}(r-s)^{2\beta-1}dr\notag\\
     &  +\iota_{2-\delta}^2 \mathbb E\|A^\delta G\|_{\mathcal F^{\beta+\frac{1}{2},\sigma}(\gamma(H;E))}^2 \int_0^s \Big(\int_{s-r}^{t-r} u^{\delta-2}du\Big)^2 r^{2\beta-1}dr  \notag  \\
     = &\iota_{1-\delta}^2 \mathbb E\|A^\delta G\|_{\mathcal F^{\beta+\frac{1}{2},\sigma}(\gamma(H;E))}^2 B(2\beta,2\delta-1)  (t-s)^{2(\beta+\delta-1)}\notag\\
     &  +\iota_{2-\delta}^2 \mathbb E\|A^\delta G\|_{\mathcal F^{\beta+\frac{1}{2},\sigma}(\gamma(H;E))}^2 \int_0^s \Big(\int_{s-r}^{t-r} u^{\delta-2}du\Big)^2\\
& \hspace{1.5cm} \times r^{2\beta-1}dr, \hspace{1cm}   0\leq s<t\leq T.  \notag
\end{align*}
By dividing $\delta-2$ as $\delta-2=-\beta+\beta+\delta-2$, 
\begin{align*}
\Big(\int_{s-r}^{t-r} u^{\delta-2}du\Big)^2&\leq    \Big[\int_{s-r}^{t-r} (s-r)^{-\beta}u^{\beta+\delta-2}du\Big]^2\notag\\
&= (s-r)^{-2\beta}\frac{[(t-r)^{\beta+\delta-1}-(s-r)^{\beta+\delta-1}]^2}{(\beta+\delta-1)^2}\notag\\
&\leq (s-r)^{-2\beta}\frac{(t-s)^{2(\beta+\delta-1)}}{(\beta+\delta-1)^2}. 
\end{align*}
Hence, 
\begin{align}
\mathbb E & \|AI_2(t)-AI_2(s)\|^2  \notag\\
    \leq &\iota_{1-\delta}^2 \mathbb E\|A^\delta G\|_{\mathcal F^{\beta+\frac{1}{2},\sigma}(\gamma(H;E))}^2 B(2\beta,2\delta-1)  (t-s)^{2(\beta+\delta-1)}\notag\\
     &  +\frac{\iota_{2-\delta}^2 \mathbb E\|A^\delta G\|_{\mathcal F^{\beta+\frac{1}{2},\sigma}(\gamma(H;E))}^2 }{(\beta+\delta-1)^2}     \int_0^s (s-r)^{-2\beta} r^{2\beta-1}dr   \notag  \\
& \times  (t-s)^{2(\beta+\delta-1)}  \notag \\
    = &\Big[\iota_{1-\delta}^2 \mathbb E\|A^\delta G\|_{\mathcal F^{\beta+\frac{1}{2},\sigma}(\gamma(H;E))}^2 B(2\beta,2\delta-1)    \label{E100} \\
     &  +\frac{\iota_{2-\delta}^2 \mathbb E\|A^\delta G\|_{\mathcal F^{\beta+\frac{1}{2},\sigma}(\gamma(H;E))}^2 B(2\beta, 1-2\beta)}{(\beta+\delta-1)^2}\Big]  \notag\\
& \times (t-s)^{2(\beta+\delta-1)}, \hspace{1cm} 0\leq s<t\leq T.   \notag
\end{align}

On the other hand,  Theorem \ref{T9} provides that  $ AI_2$ is a Gaussian process on $[0,T]$.   Thanks to \eqref{E100}, Theorem \ref{Th0} applied to the Gaussian process  $ AI_2$ gives  that for any $0<\gamma< \beta+\delta-1$,
$$AI_2\in \mathcal C^\gamma([0,T];E) \hspace{1cm} \text{ a.s.  } $$ 
As a consequence,
$AI_2$ and then $I_2=A^{-1} AI_2$ belong to $\mathcal C([0,T];E)$  a.s.

{\bf Step 4}. Let us prove that
\begin{align}
\mathbb E&\|I_1(t)\|^2 +  \mathbb E\|I_2(t)\|^2    \label{H13.5} \\
 \leq  &  2\iota_0^2 e^{-2\nu t} \mathbb E \|\xi\|^2 + \frac{\iota_0^2 \mathbb E\|F_1\|_{\mathcal F^{\beta,\sigma}(E)}^2 t^{2\beta}}{\beta^2}       \notag  \\
&  +\frac{c(E) \iota_0^2 \|A^{-\delta}\|^2 \mathbb E\|A^\delta G\|_{\mathcal F^{\beta+\frac{1}{2},\sigma}(\gamma(H;E))}^2 t^{2\beta}}{2\beta},    \hspace{1cm} 0\leq t\leq T.  \notag
\end{align}
Thanks to \eqref{H11} and  \eqref{Eq0}, 
\begin{align*}
\|I_1(t)\| \leq & \|S(t)\xi\| + \int_0^t\|S(t-s) F_1(s)\|ds   \\
\leq &  \iota_0 e^{-\nu t} \|\xi\| + \frac{\iota_0 \|F_1\|_{\mathcal F^{\beta,\sigma}(E)} t^\beta}{\beta}, \hspace{1cm} 0\leq t\leq T.
\end{align*}
Hence,
\begin{align*}
\mathbb E \|I_1(t)\|^2 \leq 
  2\iota_0^2 e^{-2\nu t} \mathbb E \|\xi\|^2 + \frac{2\iota_0^2 \mathbb E\|F_1\|_{\mathcal F^{\beta,\sigma}(E)}^2 t^{2\beta}}{\beta^2}, \hspace{1.5cm} 0\leq t\leq T.
\end{align*}

On the other hand, by \eqref{H13.3}, 
\begin{align*}
\mathbb E \|I_2(t)\|^2 & \leq  c(E) \int_0^t\|S(t-s)G(s)\|_{\gamma(H;E)}^2 ds \\
& \leq  \frac{c(E) \iota_0^2 \|A^{-\delta}\|^2 \mathbb E\|A^\delta G\|_{\mathcal F^{\beta+\frac{1}{2},\sigma}(\gamma(H;E))}^2 t^{2\beta}}{2\beta},
   \hspace{1.5cm} 0\leq t\leq T.
\end{align*}
Combining the above inequalities yields \eqref{H13.5}.

{\bf Step 5}. Let us prove that
\begin{align}
\mathbb E&\|AI_1(t)\|^2 +  \mathbb E\|AI_2(t)\|^2    \label{H13.6} \\
 \leq  &  2\varsigma_1^2 \mathbb E \|\xi\|^2 t^{-2}   \notag\\
& + 2[1+\varsigma_1  B(\beta-\sigma, \sigma)   +\varsigma_0 e^{-\varsigma_0 t} ]^2   \mathbb E \|F_1\|_{\mathcal F^{\beta,\sigma}(E)}^2 t^{2(\beta-1)}       \notag  \\
&  +\iota_{1-\delta}^2 \mathbb E \|A^\delta G\|_{\mathcal F^{\beta+\frac{1}{2}, \sigma}(\gamma(H;E))}^2  B(2\beta,2\delta-1)  t^{2(\beta+\delta-1)},    \hspace{0.5cm} 0< t\leq T.  \notag
\end{align}
On account of    \eqref{H12.6},  \eqref{H13.1} and   \eqref{H13.2}, it is easily seen that
\begin{align*}
\|AI_1(t)\| \leq & \varsigma_1 \|\xi\| t^{-1} + [1+\varsigma_1  B(\beta-\sigma, \sigma)   +\varsigma_0 e^{-\varsigma_0 t} ]  \\
& \hspace{1cm} \times \|F_1\|_{\mathcal F^{\beta,\sigma}(E)} t^{\beta-1}, \hspace{1cm} 0< t\leq T.
\end{align*}
Therefore,
\begin{align*}
\mathbb E \|AI_1(t)\|^2 \leq & 2\varsigma_1^2 \mathbb E \|\xi\|^2 t^{-2} + 2[1+\varsigma_1  B(\beta-\sigma, \sigma)   +\varsigma_0 e^{-\varsigma_0 t} ]^2  \\
& \hspace{1cm} \times \mathbb E \|F_1\|_{\mathcal F^{\beta,\sigma}(E)}^2 t^{2(\beta-1)}, \hspace{1cm} 0< t\leq T.
\end{align*}

On the other hand,  \eqref{H13.4} gives
\begin{align*}
&\mathbb E \|AI_2(t)\|^2 \\
& \leq  c(E) \int_0^t\mathbb E \|AS(t-s)G(s)\|_{\gamma(H;E)}^2 ds  \\
& \leq  \iota_{1-\delta}^2 \mathbb E \|A^\delta G\|_{\mathcal F^{\beta+\frac{1}{2}, \sigma}(\gamma(H;E))}^2  B(2\beta,2\delta-1)  t^{2(\beta+\delta-1)},
   \hspace{1cm} 0\leq t\leq T.
\end{align*}
Combining the above inequalities, it derives \eqref{H13.6}.

Based on  Steps 1 and  3, we conclude that  $X=I_1+I_2$ is a strict solution of \eqref{linear} possessing the regularity 
$$AX\in \mathcal C((0,T];E) \hspace{1cm} \text{a.s.}$$
 In addition,  it is easily seen that \eqref{H12.7} follows from Step 4 and  Step 5. This completes the proof.
\end{proof}

Finally, let us show the maximal regularity  of strict solutions for more regular initial value, say $\xi \in \mathcal D(A^\beta)$.
\begin{theorem}   \label{Th3} 
Let {\rm (Aa)}, {\rm (Ab)},  {\rm (F1)} and {\rm (Ga)} be satisfied.
 Assume that   $\xi \in \mathcal D(A^\beta)$ a.s.  Then,  the strict solution $X$ of \eqref{linear} has the space-time regularity
$$
 X \in \mathcal C([0,T];\mathcal D(A^\beta)) \cap \mathcal C^{\gamma_1}([0,T];E)
\hspace{1cm} \text{a.s.,}
$$
$$AX \in \mathcal C^{\gamma_2}([\epsilon,T];E)    \hspace{1cm} \text{a.s.}$$
for any  $ 0<\gamma_1<\beta$, $ 0<\gamma_2<\beta+\delta-1, 0<\gamma_2\leq \sigma $ and $0<\epsilon \leq T.$ 
In addition,  $X$ satisfies the estimate
\begin{align}
\mathbb E \|A^\beta X(t)\|^2 
 \leq  &C[e^{-2\nu t}\mathbb E \|A^\beta \xi\|^2 +  \mathbb E \|F_1\|_{\mathcal F^{\beta,\sigma}(E)}^2     \label{H10.5}\\
& \quad +   \mathbb E \|A^\delta G\|_{\mathcal F^{\beta+\frac{1}{2},\sigma}(\gamma(H;E))}^2  t^{2\beta}], \hspace{1.5cm} 0\leq t\leq T  \notag
\end{align}
with some $C>0$ depending only on the exponents.
\end{theorem}
\begin{proof}
We divide the proof into several steps.

{\bf Step 1}. Let us show that  $A^\beta X=A^\beta I_1+A^\beta I_2 \in \mathcal C([0,T];E) $  a.s., where  $I_1$ and $ I_2$ are defined by \eqref{Pt9} and \eqref{Pt10}.
Indeed, Step 1 and  Step 3 of the proof of Theorem \ref{Th2} give that
$$A^\beta I_1=A^{\beta-1}  AI_1 \in  \mathcal C((0,T];E)   \hspace{1cm} \text{a.s.,}$$
and 
$$A^\beta I_2=A^{\beta-1}  AI_2 \in  \mathcal C([0,T];E)   \hspace{1cm} \text{a.s.}$$
It suffices to verify the continuity of $A^\beta I_1$ at $t=0$.

 Since $S(\cdot)$ is strongly continuous,  
$$\lim_{t\to 0}\|A^\beta S(t)\xi-A^\beta \xi\|=\lim_{t\to 0}\|[S(t)-I]A^\beta \xi\|=0.$$
Therefore, $A^\beta S(\cdot)\xi$ is continuous at $t=0$.

 In the meantime, by the property of the space $\mathcal F^{\beta,\sigma}((0,T];E),$  we may put $z=\lim_{t\to 0} t^{1-\beta}F_1(t)$.  Then, 
\begin{align*}
\Big\|A^\beta & \int_0^t S(t-s)F_1(s)ds\Big\|\notag\\
=&\Big\|\int_0^t A^\beta S(t-s)[F_1(s)-F_1(t)]ds\Big\|+\Big\|\int_0^t A^\beta S(t-s)F_1(t)ds\Big\|\notag\\
=&\Big\|\int_0^t A^\beta S(t-s)[F_1(s)-F_1(t)]ds\Big\|+\Big\|[I-S(t)]A^{\beta-1}F_1(t)\Big\|\notag\\
\leq &\int_0^t \|A^\beta S(t-s)\| \|F_1(t)-F_1(s)\|ds\notag\\
&+\| t^{\beta-1}[I-S(t)]A^{\beta-1} [t^{1-\beta}F_1(t)-z]\|\notag\\
&+\| t^{\beta-1}[I-S(t)]A^{\beta-1} z\|.\notag
\end{align*}
Thereby,  \eqref{Fbetasigma3},  \eqref{H10} and  \eqref{H11.5} imply that 
\begin{align*}
\limsup_{t\to 0}&\Big\|A^\beta\int_0^t S(t-s)F_1(s)ds\Big\|\notag\\
\leq & \iota_\beta \limsup_{t\to 0}\int_0^t  (t-s)^{-\beta}\|F_1(t)-F_1(s)\|ds\notag\\
&+\frac{\iota_\beta}{1-\beta}\limsup_{t\to 0}\|t^{1-\beta}F_1(t)-z\|\notag\\
&+\limsup_{t\to 0}\| t^{\beta-1}[I-S(t)]A^{\beta-1} z\|\notag\\
= & \iota_\beta \limsup_{t\to 0}\int_0^t (t-s)^{\sigma-\beta}s^{-1+\beta-\sigma}  \frac{ s^{1-\beta+\sigma}\|F_1(t)-F_1(s)\|}{(t-s)^\sigma}ds\notag\\
&+\frac{\iota_\beta}{1-\beta}\limsup_{t\to 0}\|t^{1-\beta}F_1(t)-z\|\notag\\
&+\limsup_{t\to 0}\| t^{\beta-1}[I-S(t)]A^{\beta-1} z\|\notag\\
\leq &\iota_\beta B(\beta-\sigma, 1-\beta+\sigma)\limsup_{t\to 0} \sup_{s\in[0,t)}\frac{ s^{1-\beta+\sigma}\|F_1(t)-F_1(s)\|}{(t-s)^\sigma}\notag\\
&+\limsup_{t\to 0}\| t^{\beta-1}[I-S(t)]A^{\beta-1} z\|\notag\\
=&\limsup_{t\to 0}\| t^{\beta-1}[I-S(t)]A^{\beta-1} z\|. \notag
\end{align*}
Since $\mathcal D(A^\beta)$ is dense in $H$, there exists a sequence $\{z_n\}_{n=1}^\infty$  in $\mathcal D(A^\beta)$ that converges to $z$ as $n\to \infty.$ Hence, by using \eqref{H11.5}, 
\begin{align*}
\limsup_{t\to 0}&\Big\|A^\beta\int_0^t S(t-s)F_1(s)ds\Big\|\\
\leq&\limsup_{t\to 0}  \|t^{\beta-1}[I-S(t)]A^{\beta-1} (z-z_n)\|\\
&+\limsup_{t\to 0}  \|t^{\beta-1}[I-S(t)]A^{-1}A^{\beta} z_n\|\\
\leq &\frac{\iota_\beta}{1-\beta}\|z-z_n\|+\iota_0 \limsup_{t\to 0}  t^\beta \|A^{\beta} z_n\|\\
= & \frac{\iota_\beta}{1-\beta}\|z-z_n\|, \hspace{1cm} n=1,2,\dots
\end{align*}
Letting $n$ to $\infty$, we arrive at
 $$\lim_{t\to 0}A^\beta\int_0^t S(t-s)F_1(s)ds=0.$$
 This means that  $A^\beta\int_0^\cdot S(\cdot-s)F_1(s)ds$ is (right) continuous at $t=0$.
 Because   
$$A^\beta I_1(t)=A^\beta S(t)\xi+A^\beta\int_0^t S(t-s)F_1(s)ds,$$
we conclude  that $A^\beta I_1$  is continuous at $t=0$.

{\bf Step 2}. Let us show that 
$$I_1\in \mathcal C^\beta([0,T];E) \hspace{1cm} \text{ a.s.}$$
Let $0\leq s<t\leq T$.
From \eqref{H10}, \eqref{H11}, \eqref{Pt9} and \eqref{I1equation}, 
\begin{align}
&\|I_1(t)-I_1(s)\|\notag\\
=&\Big\|\int_s^t F_1(u)du-\int_s^t AI_1(u) du\Big\|  \notag \\ 
=&\Big\|\int_s^t F_1(u)du-\int_s^t AS(u)\xi  du-\int_s^t \int_0^uAS(u-r)F_1(r)drdu\Big\|  \notag \\ 
\leq& \Big\|\int_s^t [F_1(u)-AS(u)\xi]  du\Big\|+\int_s^t  \Big\|\int_0^uAS(u-r)F_1(u)dr\Big\|du  \notag  \\
&+\int_s^t  \Big\|\int_0^uAS(u-r)[F_1(u)-F_1(r)]dr\Big\|du   \notag\\ 
\leq & \Big\|\int_s^t [F_1(u)-AS(u)\xi]  du\Big\|+\int_s^t  \|[I-S(u)]F_1(u)\|du  \label{Eq2} \\
&+\int_s^t  \int_0^u\|AS(u-r)\| \|F_1(u)-F_1(r)\|drdu  \notag \\ 
\leq & \int_s^t [\|F_1(u)-AS(u)\xi  \|+(1+\iota_0)\|F_1(u)\|]du  \notag \\
&+\iota_1\int_s^t  \int_0^u(u-r)^{-1} \|F_1(u)-F_1(r)\|drdu  \notag \\ 
=&I_{11}(s,t)+I_{12}(s,t).  \label{Eq28} 
\end{align}

We estimate   $I_{11}$ and $I_{12}$ as follows. 
Since $\xi\in \mathcal D(A^\beta)$ a.s.,  
$$AS(\cdot)\xi \in \mathcal F^{\beta,\sigma}((0,T];E) \hspace{1cm} \text{ a.s.} \hspace{1cm}  (\text{see } \eqref{Eq53}).$$
So,  
$$F_1-AS\xi \in \mathcal F^{\beta,\sigma}((0,T];E) \hspace{1cm} \text{ a.s.}$$ 
By \eqref {H12},  
\begin{align}
I_{11}(s,t)&\leq \int_s^t [\|F_1-AS\xi\|_ {\mathcal F^{\beta,\sigma}(E)}   u^{\beta-1}+\|F_1\|_{\mathcal F^{\beta,\sigma}(E)}(1+\iota_0) u^{\beta-1}]du    \notag\\
&=\frac{\|F_1-AS\xi\|_{\mathcal F^{\beta,\sigma}(E)} + \|F_1\|_ {\mathcal F^{\beta,\sigma}(E)}(1+\iota_0)}{\beta}    (t^\beta-s^\beta)    \notag\\
&\leq \frac{\|F_1-AS\xi\|_ {\mathcal F^{\beta,\sigma}(E)} + \|F_1\|_ {\mathcal F^{\beta,\sigma}(E)}(1+\iota_0)}{\beta} (t-s)^\beta \hspace{1cm} \text{ a.s.}    \label{H16.5}
\end{align}

In the meantime,  \eqref {H12} gives
\begin{align}
I_{12}(s,t)=&\iota_1\int_s^t  \int_0^u(u-r)^{\sigma-1}  r^{\beta-1-\sigma}\frac{r^{1-\beta+\sigma}\|F_1(u)-F_1(r)\|}{(u-r)^\sigma}drdu \notag\\
\leq & \iota_1\|F_1\|_{\mathcal F^{\beta,\sigma}(E)}\int_s^t  \int_0^u(u-r)^{\sigma-1}  r^{\beta-\sigma-1}drdu\notag\\
= & \iota_1\|F_1\|_{\mathcal F^{\beta,\sigma}(E)}\int_s^t  u^{\beta-1}\int_0^1(1-v)^{\sigma-1}  v^{\beta-\sigma-1}dvdu\notag\\
= & \iota_1\|F_1\|_{\mathcal F^{\beta,\sigma}(E)} B(\beta-\sigma,\sigma)\int_s^t  u^{\beta-1}du\notag\\
= & \frac{\iota_1\|F_1\|_{\mathcal F^{\beta,\sigma}(E)} B(\beta-\sigma,\sigma)}{\beta} (t^\beta-s^\beta)\notag\\
 \leq& \frac{\iota_1\|F_1\|_{\mathcal F^{\beta,\sigma}(E)} B(\beta-\sigma,\sigma)}{\beta} (t-s)^\beta \hspace{1cm} \text{ a.s.}  \label{Eq29}
\end{align}
In view of  \eqref{Eq28}, \eqref{H16.5} and \eqref{Eq29}, we conclude that $I_1$ is $\beta$\,{-}\,H\"older continuous on $[0,T]$.

{\bf Step 3}. Let us verify that for any $0<\gamma_1<\beta$, 
$$I_2\in \mathcal C^{\gamma_1}([0,T];E) \hspace{1cm} \text{ a.s.  } $$
We notice the expression 
\begin{align*}I_2(t)-I_2(s)=&\int_s^t S(t-r)G(r)dW(r) \\
& +\int_0^s [S(t-r)-S(s-r)]G(r)dW(r), \hspace{1cm} 0\leq s<t\leq T.
\end{align*}
Since the integrals in the right-hand side of the latter equality are independent and have  zero expectation,  
\begin{align*}
\mathbb E& \|I_2(t)-I_2(s)\|^2  \notag\\
%
=&\mathbb E \Big\|\int_s^t S(t-r)G(r)dW(r)\Big\|^2 \notag\\
&   +\mathbb E\Big\|\int_0^s [S(t-r)-S(s-r)]G(r)dW(r)\Big\|^2  \notag\\
\leq & c(E) \int_s^t \mathbb E\|S(t-r)A^{-\delta} A^\delta G(r)\|_{\gamma(H;E)}^2dr \notag\\
&   +c(E) \int_0^s \mathbb E\|A^{-\delta} [S(t-r)-S(s-r)] A^\delta G(r)\|_{\gamma(H;E)}^2dr  \notag\\
%
\leq & c(E) \int_s^t  \|S(t-r)\|^2 \|A^{-\delta}\|^2 \mathbb E \|A^\delta G(r)\|_{\gamma(H;E)}^2 dr\notag\\
&+ c(E) \int_0^s  \|A^{-\delta} [S(t-r)-S(s-r)]\|^2 \mathbb E \|A^\delta G(r)\|_{\gamma(H;E)}^2dr.  \notag
\end{align*}
By \eqref {H12}, \eqref{H10},  \eqref{H11} and {\rm (Ga)},  
\begin{align*}
\mathbb E& \|I_2(t)-I_2(s)\|^2  \notag\\
%
\leq & \iota_0^2 c(E)    \|A^{-\delta}\|^2 \mathbb E \|A^\delta G\|_{\mathcal F^{\beta+\frac{1}{2},\sigma}(\gamma(H;E))}^2 
 \int_s^t r^{2\beta-1}dr  \\
& +c(E) \mathbb E \|A^\delta G\|_{\mathcal F^{\beta+\frac{1}{2},\sigma}(\gamma(H;E))}^2 \int_0^s\Big\|\int_{s-r}^{t-r} A^{1-\delta}S(u) du\Big\|^2 r^{2\beta-1}dr   \notag\\
%
\leq &  \frac {\iota_0^2 c(E)    \|A^{-\delta}\|^2 \mathbb E \|A^\delta G\|_{\mathcal F^{\beta+\frac{1}{2},\sigma}(\gamma(H;E))}^2     
 (t^{2\beta}-s^{2\beta})}{2\beta}\\
& +\iota_{1-\delta}^2 c(E) \mathbb E \|A^\delta G\|_{\mathcal F^{\beta+\frac{1}{2},\sigma}(\gamma(H;E))}^2    \int_0^s \Big(\int_{s-r}^{t-r} u^{\delta-1}du\Big)^2 r^{2\beta-1}dr.  \notag
\end{align*}
Because 
\begin{align*}
\Big(\int_{s-r}^{t-r} u^{\delta-1}du\Big)^2&= \frac{[(t-r)^\delta-(s-r)^\delta]^2}{\delta^2}\notag\\
&\leq \frac{(t-s)^{2\delta}}{\delta^2},  
\end{align*}
we obtain that for $0\leq s<t\leq T$,
\begin{align}
\mathbb E& \|I_2(t)-I_2(s)\|^2  \notag\\
\leq & \frac {\iota_0^2 c(E)    \|A^{-\delta}\|^2 \mathbb E \|A^\delta G\|_{\mathcal F^{\beta+\frac{1}{2},\sigma}(\gamma(H;E))}^2     
 (t^{2\beta}-s^{2\beta})}{2\beta}      \notag\\
& +   \frac{\iota_{1-\delta}^2 c(E) \mathbb E \|A^\delta G\|_{\mathcal F^{\beta+\frac{1}{2},\sigma}(\gamma(H;E))}^2     (t-s)^{2\delta}}{\delta^2}    \int_0^s r^{2\beta-1}dr  \Big ]  \notag\\
%
= &  \frac {\iota_0^2 c(E)    \|A^{-\delta}\|^2 \mathbb E \|A^\delta G\|_{\mathcal F^{\beta+\frac{1}{2},\sigma}(\gamma(H;E))}^2     
 (t^{2\beta}-s^{2\beta})}{2\beta}    
                      \label{ExpextationOfI2tI2sSquare}\\
& +   \frac{\iota_{1-\delta}^2 c(E) \mathbb E \|A^\delta G\|_{\mathcal F^{\beta+\frac{1}{2},\sigma}(\gamma(H;E))}^2  s^{2\beta}   (t-s)^{2\delta}}{2\beta \delta^2}.\notag
\end{align}

On the other hand,  Theorem \ref{T9} gives that  $  I_2$ is a Gaussian process on $[0,T]$. Thanks to  \eqref{ExpextationOfI2tI2sSquare} and a fact that $\delta>\beta$, Theorem \ref{Th0}  applied to the Gaussian process $I_2$ provides  that for any $0<\gamma_1<\beta$,
$$I_2\in \mathcal C^{\gamma_1}([0,T];E)  \hspace{1cm} \text{ a.s.  } $$

{\bf Step 4}. Let us verify that for any $ 0<\epsilon<T,$ $ 0<\gamma_2<\beta+\delta-1$ and $ 0<\gamma_2\leq \sigma, $
$$AX=AI_1+AI_2 \in \mathcal C^{\gamma_2}([\epsilon,T];E) \hspace{1cm} \text{a.s.}$$
We already proved in  Step 3 of the proof of Theorem \ref{Th2} that  for any   $  0<\gamma<\beta+\delta-1$, 
$$AI_2\in \mathcal C^\gamma([0,T];E) \hspace{1cm} \text{
a.s.}$$ 
It suffices to show that  for any $0<\epsilon<T,$
$$AI_1\in \mathcal C^\sigma([\epsilon,T];E) \hspace{1cm} \text{
a.s.}  $$

We use the techniques presented in  \cite{yagi}. From the expression
\begin{align*}
AI_1(t)&=AS(t)\xi+\int_0^t AS(t-r)F_1(t)dt+ \int_0^t AS(t-r)[F_1(r)-F_1(t)]dr \\
&=AS(t)\xi+[I-S(t)]F_1(t)+ \int_0^t AS(t-r)[F_1(r)-F_1(t)]dr,
\end{align*}
it follows that  
\begin{align*}
A&I_1(t)-AI_1(s)\\
=& A[S(t)-S(s)]  \xi \\
     &+ [I-S(t)][F_1(t)-F_1(s)]-[S(t-s)-I]S(s)F_1(s)\\
     &+ \int_s^t AS(t-r)[F_1(r)-F_1(t)]dr\\
     &+ [S(t-s)-I]  \int_0^s AS(s-r)[F_1(r)-F_1(s)]dr\\
     &+ \int_s^t AS(t-r)[F_1(s)-F_1(t)]dr\\
=& A[S(t)-S(s)]  \xi \\
     &+ [S(t-s)-S(t)][F_1(t)-F_1(s)]-[S(t-s)-I]S(s)F_1(s)\\
     &+ \int_s^t AS(t-r)[F_1(r)-F_1(t)]dr\\
     &+ [S(t-s)-I]  \int_0^s AS(s-r)[F_1(r)-F_1(s)]dr, \hspace{1cm} 0<s<t\leq T,
\end{align*}
here we used the equality 
$$\int_s^t AS(t-r)[F_1(s)-F_1(t)]dr=[I-S(t-s)][F_1(s)-F_1(t)].$$
Hence,
\begin{align*}
\|A&I_1(t)-AI_1(s)\|\\
\leq & \|AS(t)\xi-AS(s)  \xi\| \\
     &+ \|[S(t-s)-S(t)][F_1(t)-F_1(s)]\|  + \|[S(t-s)-I]S(s)F_1(s)\|\\
     &+ \int_s^t \|AS(t-r)\|  \|F_1(r)-F_1(t)\|dr\\
     &+ \Big\| [S(t-s)-I]  \int_0^s AS(s-r)[F_1(r)-F_1(s)]dr  \Big\|\\
=& J_1+J_2+J_3+J_4+J_5, \hspace{1cm} 0<s<t\leq T.
\end{align*}
We  give estimates for   $J_i$ $(i=1,\dots,5)$ as follows.

{\it For $J_1$.} By \eqref{Eq53}, $AS(\cdot) \xi \in \mathcal F^{\beta,\sigma}((0,T];E).$ Due to \eqref{H12.5},   
\begin{align*}
 J_1 & \leq \|AS(\cdot) \xi \|_{\mathcal F^{\beta, \sigma}(E)} (t-s)^{\sigma} s^{\beta-\sigma-1}. 
\end{align*}

{\it For $J_2$.}  It follows from   \eqref{H12.5} and \eqref{H11} that 
\begin{align*}
 J_2 & \leq 2\iota_0 \|F_1\|_{\mathcal F^{\beta, \sigma}(E)} (t-s)^{\sigma} s^{\beta-\sigma-1}. 
\end{align*}

{\it For $J_3$.}  Thanks to \eqref{H12}, \eqref{H10} and \eqref{H11.5},
\begin{align*}
 J_3 & \leq \|[S(t-s)-I]A^{-\sigma}\| \|A^\sigma  S(s)\| \|F_1(s)\|  \\
  & \leq   \frac{\iota_{1-\sigma}\iota_\sigma}  {\sigma}        \|F_1\|_{\mathcal F^{\beta, \sigma}(E)}    (t-s)^\sigma   s^{\beta-\sigma-1}. 
\end{align*}

{\it For $J_4$.} By \eqref{H12.5} and  \eqref{H10}, 
\begin{align*}
 J_4 & \leq \iota_1       \|F_1\|_{\mathcal F^{\beta, \sigma}(E)}  \int_s^t (t-r)^{\sigma-1} r^{\beta-\sigma-1} dr\\
& \leq \iota_1       \|F_1\|_{\mathcal F^{\beta, \sigma}(E)}  s^{\beta-\sigma-1}  \int_s^t (t-r)^{\sigma-1} dr\\
&=  \frac{\iota_1       \|F_1\|_{\mathcal F^{\beta, \sigma}(E)}}  {\sigma}   (t-s)^\sigma   s^{\beta-\sigma-1}. 
\end{align*}

{\it For $J_5$.} Due to \eqref{H12.5} and  \eqref{H10}, 
\begin{align*}
 J_5 = & \Big\|  \int_0^{t-s} AS(u)du \int_0^s AS(s-r)[F_1(r)-F_1(s)]dr  \Big\| \\
 = & \Big\|  \int_0^{t-s}  \int_0^s A^2 S(s+u-r)[F_1(r)-F_1(s)]dr  du \Big\|  \\
 \leq  &  \int_0^{t-s}  \int_0^s \|A^2 S(s+u-r)\| \|F_1(r)-F_1(s)\|dr  du   \\
 \leq  &  \iota_2 \|F_1\|_{\mathcal F^{\beta, \sigma}(E)} \int_0^{t-s}  \int_0^s   (s+u-r)^{-2}    (s-r)^\sigma   r^{\beta-\sigma-1} dr  du   \\
 = &   \iota_2 \|F_1\|_{\mathcal F^{\beta, \sigma}(E)}    (t-s)   \int_0^s  (t-r)^{-1} (s-r)^{\sigma-1}      r^{\beta-\sigma-1} dr     \\
 = &   \iota_2 \|F_1\|_{\mathcal F^{\beta, \sigma}(E)}    (t-s)   \int_0^s  (t-s+r)^{-1} r^{\sigma-1}      (s-r)^{\beta-\sigma-1} dr\\
= & \iota_2 \|F_1\|_{\mathcal F^{\beta, \sigma}(E)}    (t-s)   \int_{\frac{s}{2}}^s  (t-s+r)^{-1} r^{\sigma-1}      (s-r)^{\beta-\sigma-1} dr \\
&+ \iota_2 \|F_1\|_{\mathcal F^{\beta, \sigma}(E)}    (t-s)   \int_0^{\frac{s}{2}}  (t-s+r)^{-1} r^{\sigma-1}      (s-r)^{\beta-\sigma-1} dr.
\end{align*} 
The last two integrals can be estimated as 
\begin{align*}
& (t-s)   \int_{\frac{s}{2}}^s  (t-s+r)^{-1} r^{\sigma-1}    (s-r)^{\beta-\sigma-1} dr\\
&\leq 2(t-s)^\sigma s^{-1} \int_{\frac{s}{2}}^s [(t-s)^{1-\sigma} r^\sigma (t-s+r)^{-1}] (s-r)^{\beta-\sigma-1} dr\\
&\leq 2(t-s)^\sigma s^{-1} \int_0^s (s-r)^{\beta-\sigma-1} dr\\
&= \frac{2}{\beta-\sigma}  (t-s)^\sigma  s^{\beta-\sigma-1},
\end{align*} 
and 
\begin{align*}
& (t-s)   \int_0^{\frac{s}{2}}   (t-s+r)^{-1} r^{\sigma-1}    (s-r)^{\beta-\sigma-1} dr\\
&\leq  (t-s) 2^{1-\beta+\sigma} s^{\beta-\sigma-1} \int_0^\infty (t-s+r)^{-1} r^{\sigma-1} dr\\
&= \frac{2^{1-\beta+\sigma}}{\sin(\sigma\pi)}  (t-s)^\sigma  s^{\beta-\sigma-1}.
\end{align*} 
Thereby,
$$J_5\leq \iota_2 \|F_1\|_{\mathcal F^{\beta, \sigma}(E)}  \Big[\frac{2}{\beta-\sigma}+\frac{2^{1-\beta+\sigma}}{\sin(\sigma\pi)} \Big](t-s)^\sigma  s^{\beta-\sigma-1}.$$

In this way, we have
\begin{align*}
\|A&I_1(t)-AI_1(s)\|\\
\leq & \Big[ \|AS(\cdot) \xi \|_{\mathcal F^{\beta, \sigma}(E)} + \Big\{   2\iota_0  + \frac{\iota_{1-\sigma}\iota_\sigma}  {\sigma}         +  \frac{\iota_1      }  {\sigma}  +  \frac{2\iota_2}{\beta-\sigma}+\frac{2^{1-\beta+\sigma} \iota_2}{\sin(\sigma\pi)} \Big\} \\
     & \quad \times \|F_1\|_{\mathcal F^{\beta, \sigma}(E)}    \Big ] (t-s)^\sigma  s^{\beta-\sigma-1}, 
 \hspace{1cm} 0<s<t\leq T.
\end{align*}
Thus, for any $ 0<\epsilon<T$, 
$$AI_1\in \mathcal C^\sigma([\epsilon,T];E). $$

{\bf Step 5}.  Let us prove the estimate   \eqref{H10.5}.   By  \eqref {H12},  \eqref{H10} and  \eqref{H11},
\begin{align*}
&\mathbb E \|A^\beta X(t)\|^2\\
=& \mathbb E\|A^\beta [S(t)\xi+\int_0^t  S(t-s)  F_1(s) ds+\int_0^t  S(t-s)  G(s) dW(s)]\|^2 \\
 \leq &3 \|S(t)\|^2 \mathbb E \|A^\beta \xi\|^2 +3 \mathbb E  \Big\|\int_0^t A^\beta S(t-s)  F_1(s) ds\Big\|^2 \\
& +3 \mathbb E  \Big\|\int_0^t A^\beta S(t-s)  G(s) dW(s)\Big\|^2 \\
\leq &3 \iota_0^2 e^{-2\nu t} \mathbb E \|A^\beta \xi\|^2 +3 \iota_\beta^2 \mathbb E \|F_1\|_{\mathcal F^{\beta,\sigma}(E)}^2  \Big\|\int_0^t (t-s)^{-\beta}   s^{\beta-1} ds\Big\|^2 \\
& +3 c(E) \int_0^t  \mathbb E  \|A^\beta S(t-s)  G(s)\|^2 ds \\
\leq &3 \iota_0^2 e^{-2\nu t} \mathbb E \|A^\beta \xi\|^2 +3 \iota_\beta^2   B(\beta,1-\beta)^2   \mathbb E \|F_1\|_{\mathcal F^{\beta,\sigma}(E)}^2   \\
& +3 c(E) \int_0^t   \|A^{\beta-\delta}\|^2 \| S(t-s)\|^2  \mathbb E \|A^\delta G(s)\|^2 ds \\
\leq &3 \iota_0^2 e^{-2\nu t} \mathbb E \|A^\beta \xi\|^2 +3 \iota_\beta^2   B(\beta,1-\beta)^2   \mathbb E \|F_1\|_{\mathcal F^{\beta,\sigma}(E)}^2   \\
& +3 c(E)\|A^{\beta-\delta}\|^2 \iota_0^2  \int_0^t   s^{2\beta-1}  \mathbb E \|A^\delta G\|_{\mathcal F^{\beta+\frac{1}{2},\sigma}(\gamma(H;E))}^2 ds\\
= &3 \iota_0^2 e^{-2\nu t} \mathbb E \|A^\beta \xi\|^2 +3 \iota_\beta^2   B(\beta,1-\beta)^2   \mathbb E \|F_1\|_{\mathcal F^{\beta,\sigma}(E)}^2   \\
& +\frac{3 c(E)\|A^{\beta-\delta}\|^2 \iota_0^2 }{ 2\beta}    \mathbb E \|A^\delta G\|_{\mathcal F^{\beta+\frac{1}{2},\sigma}(\gamma(H;E))}^2  t^{2\beta}, \hspace{1cm} 0\leq t\leq T.
\end{align*}
Therefore, \eqref{H10.5} holds true.
\end{proof}


\subsection{Maximal regularity of mild solutions}
Let us   present maximal   regularity of mild solutions for  \eqref{linear}. For this study, the spatial regularity of $G$ in (Ga) is not necessary. In stead of that, we  assume that    (Gb) takes place. 
\begin{theorem} \label{Th4} 
Let {\rm (Aa)}, {\rm (Ab)},   {\rm (F1)} and   {\rm (Gb)}  be satisfied. 
 Suppose further that   $\xi \in \mathcal D(A^\beta)$ a.s.  Then,  there exists a unique mild solution $X$ of \eqref{linear} possessing the regularity 
\begin{equation*} \label{regularity theorem autonomous linear evolution equationXSpace}
 X \in   \mathcal C^\alpha([0,T];E)  \hspace{1cm} \text{ a.s.,     } 0\leq \alpha <\beta,
 \end{equation*}
and 
\begin{itemize}
  \item [{\rm (i)}]  When $\beta\geq \frac{1}{4}$, for any $0<\epsilon<T$, $\frac{1}{4}\leq \theta \leq \beta,$  $ 0\leq \gamma<\frac{1}{2}-\theta$ and $\gamma \leq \sigma$,
                   $$A^\theta X \in \mathcal C^\gamma ([\epsilon,T];E) \hspace{1cm} \text{a.s. }$$ 
 \item [{\rm (ii)}]  When $\beta< \frac{1}{4}$, for any $0<\epsilon<T$, $0 \leq \theta \leq \beta$, $ 0\leq \gamma<\theta$ and $\gamma \leq \sigma$,
                   $$A^\theta X \in \mathcal C^\gamma ([\epsilon,T];E) \hspace{1cm} \text{a.s. }$$ 
\end{itemize}
In addition,  $X$ satisfies the estimate
\begin{align}
\mathbb E \|A^\beta X(t)\|^2   
\leq &C[e^{-2\nu t} \mathbb E \|A^\beta \xi\|^2 +  \mathbb E \|F_1\|_{\mathcal F^{\beta,\sigma}(E)}^2     \label{H17.6}\\
&\quad  + \mathbb E\|G\|_{\mathcal F^{\beta+\frac{1}{2},\sigma}(\gamma(H;E))}^2],  \hspace{1cm} 0\leq t\leq T,      \notag   
 \end{align}
where $C$ is some positive constant depending only on $\iota_\theta \, (\theta\geq 0)$ in \eqref{H10} and the exponents, and $ \nu$ is defined by   \eqref{H11}.
\end{theorem}

For the proof, we  use the following lemma.
\begin{lemma}   \label{Th5}
Let $G$ satisfies {\rm (Gb)}. Put
$$W_\theta(\cdot)=\int_0^\cdot A^\theta S(\cdot-s) G(s)dW(s), \hspace{1cm} 0\leq \theta\leq \beta.$$ 
Then, the  stochastic convolution $W_\theta$ is well-defined on $[0,T]$. Furthermore,  $W_\theta$ possesses the space-time regularity:
\begin{itemize}
  \item [{\rm (i)}] For any $0\leq \theta< \beta$ and $ 0\leq \gamma<\beta-\theta$, 
                   $$W_\theta \in \mathcal C^\gamma ([0,T];E) \hspace{1cm} \text{a.s. }$$ 
 \item [{\rm (ii)}]  When $\beta\geq \frac{1}{4}$, for any $0<\epsilon<T$, $\frac{1}{4}\leq \theta \leq \beta$ and $ 0\leq \gamma<\frac{1}{2}-\theta$,
                   $$W_\theta \in \mathcal C^\gamma ([\epsilon,T];E) \hspace{1cm} \text{a.s. }$$ 
 \item [{\rm (iii)}]  When $\beta< \frac{1}{4}$, for any $0<\epsilon<T$, $0 \leq \theta \leq \beta$ and $ 0\leq \gamma<\theta$,
                   $$W_\theta \in \mathcal C^\gamma ([\epsilon,T];E) \hspace{1cm} \text{a.s. }$$ 
\end{itemize}
\end{lemma}

\begin{proof}
First, let us show that $W_\theta$ 
is well-defined on $[0,T].$ 
By \eqref{H12} and \eqref {H10},  for any $0\leq \theta\leq \beta$,
\begin{align}
&\int_0^t \mathbb E\|A^\theta  S(t-s) G(s)\|_{\gamma(H;E)}^2 ds \notag\\
&\leq \int_0^t \|A^\theta  S(t-s)\|^2 \mathbb E\|G(s)\|_{\gamma(H;E)}^2 ds \notag \\
&\leq \iota_\theta^2  \mathbb E\|G\|_{\mathcal F^{\beta+\frac{1}{2},\sigma}(\gamma(H;E))}^2 \int_0^t (t-s)^{-2\theta} s^{2\beta-1}ds \notag \\
&=\iota_\theta^2  \mathbb E\|G\|_{\mathcal F^{\beta+\frac{1}{2},\sigma}(\gamma(H;E))}^2 B(2\beta, 1-2\theta) t^{2(\beta-\theta)}< \infty, \hspace{1cm} 0\leq t\leq T.       \label{Eq5}
\end{align}
Thereby, $W_\theta$ is a well-defined  Gaussian process on $[0,T]$.

Since $A^\theta$ is closed, Lemma \ref{T10} provides that for any $0\leq \theta\leq \beta$, 
$$ A^\theta W_0(t) =W_\theta(t), \hspace{1cm} 0\leq t\leq T. $$ 

Second, let us   verify {\rm (i)} by using 
 the same  argument as in  Step 3 in the proof of Theorem \ref{Th2} and the expression: 
\begin{align*}
A^\theta W_0(t)&-A^\theta W_0(s)=\int_s^t A^\theta S(t-r)G(r)dW(r) \\
&+\int_0^s A^\theta [S(t-r)-S(s-r)]G(r)dW(r), \hspace{1.5cm}   0\leq s<t\leq T.
\end{align*}

Since the integrals in the right-hand side of the latter equality are independent and have  zero expectation,  
\begin{align*}
\mathbb E& \|A^\theta W_0(t)-A^\theta W_0(s)\|^2  \notag\\
=&\mathbb E \Big\|\int_s^t A^\theta S(t-r) G(r)dW(r)\Big\|^2 \notag \\
     &+\mathbb E\Big\|\int_0^s A^\theta [S(t-r)-S(s-r)] G(r)dW(r)\Big\|^2  \notag\\
\leq & c(E) \int_s^t \mathbb E\|A^\theta  S(t-r) G(r)\|_{\gamma(H;E)}^2dr \notag \\
    &+c(E) \int_0^s \mathbb E\|A^\theta  [S(t-r)-S(s-r)] G(r)\|_{\gamma(H;E)}^2dr  \notag\\
\leq & c(E) \int_s^t  \|A^\theta  S(t-r)\|^2  \mathbb E\| G(r)\|_{\gamma(H;E)}^2 dr\notag\\
&+ c(E) \int_0^s  \|A^\theta [S(t-r)-S(s-r)]\|^2 \mathbb E\| G(r)\|_{\gamma(H;E)}^2dr \notag\\
= & c(E) \int_s^t  \|A^\theta  S(t-r)\|^2  \mathbb E\| G(r)\|_{\gamma(H;E)}^2 dr\notag\\
&+ c(E) \int_0^s  \Big\|\int_{s-r}^{t-r} A^{1+\theta}S(u)du\Big\|^2 \mathbb E\| G(r)\|_{\gamma(H;E)}^2dr,\hspace{1 cm}   0\leq s<t\leq T.  \notag
\end{align*}
Therefore,   \eqref {H12}, \eqref{H10} and \eqref{H11} give   
\begin{align}
\mathbb E& \|A^\theta W_0(t)-A^\theta W_0(s)\|^2        \notag\\
      \leq & c(E) \iota_\theta^2 \mathbb E \| G\|_{\mathcal F^{\beta+\frac{1}{2},\sigma}(\gamma(H;E))}^2  \int_s^t (t-r)^{-2\theta}r^{2\beta-1}dr       \label{H16.6}\\
        &+ c(E) \mathbb E \| G\|_{\mathcal F^{\beta+\frac{1}{2},\sigma}(\gamma(H;E))}^2 \int_0^s\Big(\int_{s-r}^{t-r} \|A^{1+\theta} S(u)\|^2 du\Big)^2 r^{2\beta-1}dr       \notag\\
     \leq &c(E)\iota_\theta^2 \mathbb E \| G\|_{\mathcal F^{\beta+\frac{1}{2},\sigma}(\gamma(H;E))}^2  \int_s^t (t-r)^{-2\theta}(r-s)^{2\beta-1}dr         \notag\\
     &  +c(E)\iota_{1+\theta}^2 \mathbb E \| G\|_{\mathcal F^{\beta+\frac{1}{2},\sigma}(\gamma(H;E))}^2 \int_0^s \Big(\int_{s-r}^{t-r} u^{-\theta-1}du\Big)^2 r^{2\beta-1}dr  \notag \\
      = &c(E)\iota_\theta^2 \mathbb E \| G\|_{\mathcal F^{\beta+\frac{1}{2},\sigma}(\gamma(H;E))}^2  \int_s^t (t-r)^{-2\theta}(r-s)^{2\beta-1}dr\notag\\
     &  +c(E)\iota_{1+\theta}^2 \mathbb E \| G\|_{\mathcal F^{\beta+\frac{1}{2},\sigma}(\gamma(H;E))}^2 \int_0^s \Big(\int_{s-r}^{t-r} u^{-\theta-1}du\Big)^2 r^{2\beta-1}dr  \notag  \\
     = &c(E)\iota_\theta^2 \mathbb E \| G\|_{\mathcal F^{\beta+\frac{1}{2},\sigma}(\gamma(H;E))}^2 B(2\beta,1-2\theta)  (t-s)^{2(\beta-\theta)}\notag\\
     &  +c(E)\iota_{1+\theta}^2 \mathbb E \| G\|_{\mathcal F^{\beta+\frac{1}{2},\sigma}(\gamma(H;E))}^2 \int_0^s \Big(\int_{s-r}^{t-r} u^{-\theta-1}du\Big)^2   \notag\\
&\hspace{1cm} \times r^{2\beta-1}dr, \hspace{1cm}   0\leq s<t\leq T.  \notag
\end{align}

The last integral is evaluated by dividing $-\theta-1$ as $-\theta-1=-\beta+\beta-\theta-1$: 
\begin{align*}
\Big(\int_{s-r}^{t-r} u^{-\theta-1}du\Big)^2&\leq    \Big[\int_{s-r}^{t-r} (s-r)^{-\beta}u^{\beta-\theta-1}du\Big]^2\notag\\
&= (s-r)^{-2\beta}\frac{[(t-r)^{\beta-\theta}-(s-r)^{\beta-\theta}]^2}{(\beta-\theta)^2}\notag\\
&\leq (s-r)^{-2\beta}\frac{(t-s)^{2(\beta-\theta)}}{(\beta-\theta)^2}. 
\end{align*}
Hence, 
\begin{align}
&\mathbb E \|A^\theta W_0(t)-A^\theta W_0(s)\|^2  \notag\\
    \leq &c(E) \iota_\theta^2 \mathbb E \| G\|_{\mathcal F^{\beta+\frac{1}{2},\sigma}(\gamma(H;E))}^2 B(2\beta,1-2\theta)  (t-s)^{2(\beta-\theta)}\notag\\
     &  +\frac{c(E)\iota_{1+\theta}^2 \mathbb E \| G\|_{\mathcal F^{\beta+\frac{1}{2},\sigma}(\gamma(H;E))}^2 }{(\beta-\theta)^2}   (t-s)^{2(\beta-\theta)}    \int_0^s (s-r)^{-2\beta} r^{2\beta-1}dr   \notag  \\
    \leq &\Big[c(E) \iota_\theta^2 \mathbb E  \| G\|_{\mathcal F^{\beta+\frac{1}{2},\sigma}(\gamma(H;E))}^2 B(2\beta,1-2\theta)    \label{H9.5} \\
     &  +\frac{c(E) \iota_{1+\theta}^2 \mathbb E \| G\|_{\mathcal F^{\beta+\frac{1}{2},\sigma}(\gamma(H;E))}^2 B(2\beta, 1-2\beta)}{(\beta-\theta)^2}\Big]  \notag\\
& \times (t-s)^{2(\beta-\theta)}, \hspace{1cm}   0\leq s<t\leq T.   \notag
\end{align}

On the other hand, by Theorem \ref{T9}, $  A^\theta W_0$ is a Gaussian process on $[0,T]$.  Combining this with \eqref{H9.5}, we apply Theorem \ref{Th0} to  $A^\theta W_0$. So,  for any $0\leq \theta< \beta$ and $ 0\leq \gamma<\beta-\theta$,  
                   $$W_\theta \in \mathcal C^\gamma ([0,T];E) \hspace{1cm} \text{a.s. }$$

Third, let us   prove  {\rm (ii)}. Let $0<\epsilon<T$ and $\frac{1}{4}\leq \theta \leq \beta$.
We use again  \eqref{H16.6}. For $\epsilon\leq s\leq t\leq T$, 
\begin{align}
\mathbb E& \|A^\theta W_0(t)-A^\theta W_0(s)\|^2        \notag\\
      \leq & c(E) \iota_\theta^2 \mathbb E \| G\|_{\mathcal F^{\beta+\frac{1}{2},\sigma}(\gamma(H;E))}^2  \int_s^t (t-r)^{-2\theta}r^{2\beta-1}dr\notag\\
        &+ c(E) \mathbb E \| G\|_{\mathcal F^{\beta+\frac{1}{2},\sigma}(\gamma(H;E))}^2 \int_0^s\Big(\int_{s-r}^{t-r} \|A^{1+\theta} S(u)\| du\Big)^2 r^{2\beta-1}dr       \notag\\
\leq &c(E)\iota_\theta^2 \mathbb E \| G\|_{\mathcal F^{\beta+\frac{1}{2},\sigma}(\gamma(H;E))}^2  \epsilon^{2\beta-1} \int_s^t (t-r)^{-2\theta}dr          \notag\\
     &  +c(E)\iota_{1+\theta}^2 \mathbb E \| G\|_{\mathcal F^{\beta+\frac{1}{2},\sigma}(\gamma(H;E))}^2 \int_0^s \Big(\int_{s-r}^{t-r} u^{-\theta-1}du\Big)^2 r^{2\beta-1}dr  \notag \\
= &\frac{    c(E)\iota_\theta^2 \mathbb E \| G\|_{\mathcal F^{\beta+\frac{1}{2},\sigma}(\gamma(H;E))}^2  \epsilon^{2\beta-1}(t-s)^{1-2\theta}       } {   1-2\theta}  \notag\\
     &  +\frac{   c(E)\iota_{1+\theta}^2 \mathbb E \| G\|_{\mathcal F^{\beta+\frac{1}{2},\sigma}(\gamma(H;E))}^2      }{\theta^2  }   \notag\\
&\times \int_0^s [(t-r)^\theta-(s-r)^\theta]^2 (t-r)^{-2\theta}  (s-r)^{-2\theta}  r^{2\beta-1}dr  \notag  \\
     \leq &\frac{    c(E)\iota_\theta^2 \mathbb E \| G\|_{\mathcal F^{\beta+\frac{1}{2},\sigma}(\gamma(H;E))}^2  \epsilon^{2\beta-1}(t-s)^{1-2\theta}       } {   1-2\theta}  \label{H19.6}\\
     &  +\frac{   c(E)\iota_{1+\theta}^2 \mathbb E \| G\|_{\mathcal F^{\beta+\frac{1}{2},\sigma}(\gamma(H;E))}^2   (t-s)^{2\theta}}   {  \theta^2 }  \notag\\
& \times    \int_0^s (t-r)^{-2\theta} (s-r)^{-2\theta} r^{2\beta-1}dr.  \notag
\end{align}

Let $0<\epsilon_1<\min\{1-2\theta, 1-2\beta\}. $ Since $1-4\theta-\epsilon_1<0$ and $\epsilon_1+2\theta-1<0$, 
\begin{align*}
&\int_0^s (t-r)^{-2\theta} (s-r)^{-2\theta} r^{2\beta-1}dr        \notag\\
  &  = \int_0^s (t-r)^{1-4\theta-\epsilon_1}  (t-r)^{\epsilon_1+2\theta-1} (s-r)^{-2\theta} r^{2\beta-1}dr \notag\\
  &   \leq   (t-s)^{1-4\theta-\epsilon_1}\int_0^s   (s-r)^{\epsilon_1-1}  r^{2\beta-1}dr \notag\\
  &  =   (t-s)^{1-4\theta-\epsilon_1}  B(2\beta,\epsilon_1) s^{2\beta+\epsilon_1-1} \notag\\
  &  \leq      B(2\beta,\epsilon_1) \epsilon^{2\beta+\epsilon_1-1}(t-s)^{1-4\theta-\epsilon_1}, \hspace{1cm} \epsilon\leq s\leq t\leq T. \notag
    \end{align*}
Therefore,  
\begin{align*}
\mathbb E& \|A^\theta W_0(t)-A^\theta W_0(s)\|^2        \notag\\
     \leq &\frac{    c(E)\iota_\theta^2 \mathbb E \| G\|_{\mathcal F^{\beta+\frac{1}{2},\sigma}(\gamma(H;E))}^2  \epsilon^{2\beta-1}(t-s)^{1-2\theta}       } {   1-2\theta}  \notag\\
     &  +\frac{   c(E)\iota_{1+\theta}^2 B(2\beta,\epsilon_1) \epsilon^{2\beta+\epsilon_1-1}\mathbb E \| G\|_{\mathcal F^{\beta+\frac{1}{2},\sigma}(\gamma(H;E))}^2   }   {  \theta^2 }\notag \\
& \times (t-s)^{1-2\theta-\epsilon_1}, \hspace{1cm} \epsilon\leq s\leq t\leq T.  \notag
\end{align*}
 Theorem \ref{Th0} applied to the Gaussian process $A^\theta W_0$ yields that
for any  $ 0\leq \gamma<\frac{1}{2}-\theta-\frac{\epsilon_1}{2}$,
                   $$W_\theta=A^\theta W_0 \in \mathcal C^\gamma ([\epsilon,T];E) \hspace{1cm} \text{a.s. }$$ 
Since $\epsilon_1$ is arbitrary in $(0,\min\{1-2\theta, 1-2\beta\})$, we conclude that 
for any  $ 0\leq \gamma<\frac{1}{2}-\theta$,
$$W_\theta=A^\theta W_0 \in \mathcal C^\gamma ([\epsilon,T];E) \hspace{1cm} \text{a.s. }$$

Finally, let us prove {\rm (iii)}. Let $0<\epsilon<T$ and $0\leq \theta\leq \beta<\frac{1}{4}$. The integral in the right-hand side of \eqref{H19.6} can be estimated as follows. For $\epsilon \leq s<t\leq T$, 
\begin{align*}
\int_0^s (t-r)^{-2\theta} (s-r)^{-2\theta} r^{2\beta-1}dr   &\leq \int_0^s (s-r)^{-4\theta} r^{2\beta-1}dr     \notag \\
&= s^{2\beta-4\theta}  B(2\beta,1-4\theta)       \notag \\
&\leq \max\{\epsilon^{2\beta-4\theta}, T^{2\beta-4\theta}\}  B(2\beta,1-4\theta).
\end{align*}
Combining this with \eqref{H19.6} yields that
\begin{align*}
\mathbb E& \|A^\theta W_0(t)-A^\theta W_0(s)\|^2        \notag\\
  \leq &\frac{    c(E)\iota_\theta^2 \mathbb E \| G\|_{\mathcal F^{\beta+\frac{1}{2},\sigma}(\gamma(H;E))}^2  \epsilon^{2\beta-1}(t-s)^{1-2\theta}       } {   1-2\theta}  \label{H19.6}\\
     &  +\frac{   c(E)\iota_{1+\theta}^2 \max\{\epsilon^{2\beta-4\theta}, T^{2\beta-4\theta}\}  B(2\beta,1-4\beta)\mathbb E \| G\|_{\mathcal F^{\beta+\frac{1}{2},\sigma}(\gamma(H;E))}^2   }   {  \theta^2 }         \notag\\
& \times  (t-s)^{2\theta},   \hspace{4cm} \epsilon \leq s<t\leq T.
\end{align*}
Again, we apply Theorem \ref{Th0} to the Gaussian process $A^\theta W_0$. Then, {\rm (iii)} follows. 
It completes the proof of the theorem.
\end{proof}

\begin{proof}[Proof for Theorem \ref{Th4}]
 In Step 2 of the proof of Theorem \ref{Th3}, we already proved that 
$$I_1 \in \mathcal C^\beta([0,T];E),$$
where $I_1$ is defined by \eqref{Pt9}.
By  choosing $\theta=0$ in Lemma \ref{Th5},  for any $0<\alpha< \beta$, 
$$W_0\in \mathcal C^\alpha([0,T];E) \hspace{1cm} \text{ a.s.  } $$ 
The process $X=I_1+W_0$   is hence a  mild solution of \eqref{linear} in the space
$$X\in \mathcal C^\alpha([0,T];E) \hspace{1cm} \text{ a.s.  } $$
for any $0<\alpha< \beta$. Note that the uniqueness of solution is obvious.


On the other hand, in Step 4 of the proof of Theorem \ref{Th3}, we verified that  for any $ 0<\epsilon<T$, 
$$AI_1 \in \mathcal C^\sigma([\epsilon,T];E).$$
As a consequence,  for any $ 0<\epsilon<T$, 
$$A^\theta I_1=A^{\theta-1} AI_1 \in \mathcal C^\sigma([\epsilon,T];E).$$
Combining this with Lemma \ref{Th5}, it yields {\rm (i)} and {\rm (ii)}.

Let us finally verify the estimate     
 \eqref{H17.6}. 
It is easily seen from Step 5 of the proof of Theorem \ref{Th3} that 
 \begin{align*}
\mathbb E \|A^\beta I_1(t)\|^2
\leq  &3 \iota_0^2 e^{-2\nu t} \mathbb E \|A^\beta \xi\|^2 \\
&+3 \iota_\beta^2   B(\beta,1-\beta)^2   \mathbb E \|F_1\|_{\mathcal F^{\beta,\sigma}(E)}^2, \hspace{1cm} 0\leq t\leq T.
\end{align*}
Combining this and \eqref{Eq5} with $\theta=\beta,$ we obtain that
\begin{align*}
\mathbb E & \|A^\beta X(t)\|^2 \\
\leq & 2\mathbb E \|A^\beta I_1(t)\|^2 + 2\mathbb E \|W_\beta (t)\|^2 \\
\leq &6 \iota_0^2 e^{-2\nu t} \mathbb E \|A^\beta \xi\|^2 +6 \iota_\beta^2   B(\beta,1-\beta)^2   \mathbb E \|F_1\|_{\mathcal F^{\beta,\sigma}(E)}^2\\
&+2c(E)\int_0^t \mathbb E\|A^\beta  S(t-s) G(s)\|_{\gamma(H;E)}^2 ds \notag\\
\leq &6 \iota_0^2 e^{-2\nu t} \mathbb E \|A^\beta \xi\|^2 +6 \iota_\beta^2   B(\beta,1-\beta)^2   \mathbb E \|F_1\|_{\mathcal F^{\beta,\sigma}(E)}^2\\
&+2c(E)\iota_\beta^2  \mathbb E\|G\|_{\mathcal F^{\beta+\frac{1}{2},\sigma}(\gamma(H;E))}^2 B(2\beta, 1-2\beta), \hspace{1cm} 0\leq t\leq T.      
 \end{align*}
Thus, \eqref{H17.6} holds true.
The proof is complete.
\end{proof}

\section{Semilinear evolution equations} \label{section5}
In this section, we handle the semilinear case of the problem \eqref{E2}, i.e. $F_2 \not\equiv 0$ in $E$. For the convenience, let us recall \eqref{E2}:
\begin{equation} \label{E101}
\begin{cases}
dX+AXdt=[F_1(t)+F_2(X)]dt+ G(t)dW(t), \hspace{1cm} 0<t\leq T,\\
X(0)=\xi.
\end{cases}
\end{equation}
We first investigate existence, uniqueness and regularity of  local mild solutions (Theorems \ref{Th8} and \ref{Th9}), and of global strict solutions (Theorem \ref{Th6}) on the basis of solution formula. We then show regular dependence of local mild solutions on initial data (Theorem \ref{Th10}).

It is known that there are two common methods for studying existence of local mild solutions to deterministic parabolic evolution equations. The first one is to use the fixed point theorem for contractions. By constructing  suitable underlying spaces and suitable mappings, one can show that these mappings have fixed points, which are the desired solutions (see, e.g., Yagi \cite{yagi}). This method is available for stochastic parabolic evolution equations. When applied to this kind of equations, local solutions are defined on a non-random interval (see Da Prato-Zabczyk \cite{prato}). 
The second one is the cut-off function method. A systematic study on local mild solutions for \eqref{E101} using this method can be found in  Brze\'{z}niak \cite{Brzezniak1}.

We  use the first method to prove existence of local mild solutions to  \eqref{E101}. For this purpose, we  construct a contraction mapping based on the solution formula and prove that the fixed point of the mapping is a unique local mild solution to \eqref{E101}. In particular, it is shown that the local mild solution becomes a strict solution under certain conditions.

\subsection{Mild solutions}
This subsection proves  existence   of local mild solutions to \eqref{E101} and show their regularity, provided that either  {\rm (Ga)} or {\rm (Gb)} takes place. First, we consider the case where {\rm (Ga)} holds true. 
Let fix  $\eta, \beta, \sigma $ such that
\begin{equation*}
\begin{cases}
 \max\{0, 2\eta-\frac{1}{2}\}<\beta<\eta, \\
 0<\sigma <\beta<\frac{1}{2}.
\end{cases}
\end{equation*}
Assume further that    
\begin{itemize}
\item [{\rm (F2a)}] $F_2\colon\mathcal D(A^\eta)\to E$  and satisfies a Lipschitz condition of the form
     \begin{equation*} 
        \|F_2(x)-F_2(y)\|\leq c_{F_2}  \|A^\eta(x-y)\| \hspace{1cm}  \text{ a.s., }  x,y\in \mathcal D(A^\eta)
      \end{equation*}
where $c_{F_2}>0$ is some  constant.
\end{itemize}
\begin{remark}
In  \cite{Brzezniak1}, Brze\'{z}niak  presented a sufficient condition, say a local Lipschitz condition \eqref{E0.6} on $F_2$, for existence of local mild solutions to \eqref{E101}. The author used the cut-off function method to show that local mild solutions are defined on an interval $[0,\tau_{loc})$, where $\tau_{loc}$ is a stopping time. 
\end{remark}
\begin{theorem}\label{Th8}
Let {\rm (Aa)}, {\rm (Ab)}, {\rm (F1)}, {\rm (F2a)} and {\rm (Ga)}   be satisfied.
Assume that  $\xi\in \mathcal D(A^\beta)$ a.s. such that $\mathbb E\|A^\beta\xi\|^2<\infty$.  Then, \eqref{E2} possesses a unique local mild solution $X$ in the function space:
\begin{equation*}\label{Ph1}
X\in  \mathcal C^\gamma ([\epsilon,T_{loc}];\mathcal D(A^\eta))\cap \mathcal C([0,T_{loc}];\mathcal D(A^\beta))  \hspace{1cm} \text{a.s.}
\end{equation*}
for any $0<\epsilon<T_{loc}$, $0<\gamma<\min\{\beta+\delta-1,\frac{1}{2}-\eta,\frac{1+2\beta}{4}-\eta\},$
where $T_{loc}$ is some positive constant in $[0,T]$ depending only on the exponents, $\mathbb E \|F_1\|_{\mathcal F^{\beta,\sigma}(E)}^2$, $\mathbb E \|F_2(0)\|^2, $ $\mathbb E \|A^\beta\xi\|^2,$  and  $ \mathbb E \|A^\delta G\|_{\mathcal F^{\beta+\frac{1}{2},\sigma}(\gamma(H;E))}^2.$  
 Furthermore, $X$ satisfies the estimates
\begin{align}
\mathbb E\| A^\beta X(t)\|^2 
\leq & C[\mathbb E\|A^\beta \xi\|^2 + \mathbb E \|F_1\|_{\mathcal F^{\beta,\sigma}(E)}^2] [1+t^{2(1-\eta)}]          \label{Ph2}\\
&+  C\mathbb E \|A^\delta G\|_{\mathcal F^{\beta+\frac{1}{2},\sigma}(\gamma(H;E))}^2 [t^{2\beta}+t^{2(1-\eta)}]    \notag\\
& +C  \mathbb E\|F_2(0)\|^2   t^{2(1-\beta)}, \hspace{1cm} 0\leq t\leq T_{loc},     \notag
\end{align}
and 
\begin{align}
\mathbb E\| A^\eta X(t)\|^2                    
\leq & C\mathbb E\|F_2(0)\|^2 t^{2(1-\eta)}    \label{semilinear evolution equationExpectationAetaXSquare}\\
&+C [ \mathbb E \|A^\beta \xi\|^2  +  \mathbb E\|F_1\|_{\mathcal F^{\beta,\sigma}(E)}^2] [t^{2(\beta-\eta)}+ t^{2(1+\beta-2\eta)} ]    \notag\\
&+ C\Big[\sup_{0\leq t<\infty}\int_0^t e^{-2\nu (t-s)} s^{2\beta-1} ds+t^{2(1+\beta-2\eta)}\Big]    \notag\\
&\times \mathbb E \|A^\delta G\|_{\mathcal F^{\beta+\frac{1}{2},\sigma}(\gamma(H;E))}^2, \hspace{1cm} 0< t\leq T_{loc}  \notag
\end{align}
with some  $C>0$ depending only on the exponents. 
\end{theorem}
\begin{proof}
We  use the fixed point theorem for contractions to prove existence and uniqueness of a local mild solution.

For each $S\in (0,T]$,  set an underlying space as follows. Denote by $\Xi (S)$ the set of all $E$-valued  processes $Y$ on $[0,S]$ such that $Y(0)\in \mathcal D(A^\beta)$ a.s. and 
$$\sup_{0<t\leq S} t^{2(\eta-\beta)} \mathbb E\|A^\eta Y(t)\|^2+ \sup_{0\leq t\leq S}\mathbb E\|A^\beta Y(t)\|^2 <\infty.$$
Up to indistinguishability, $\Xi (S)$  is then a Banach space with  norm
\begin{equation} \label{Ph3}
\|Y\|_{\Xi (S)}=\Big[\sup_{0<t\leq S} t^{2(\eta-\beta)} \mathbb E\|A^\eta Y(t)\|^2+ \sup_{0\leq t\leq S}\mathbb E\|A^\beta Y(t)\|^2\Big]^{\frac{1}{2}}. 
\end{equation}

Let fix  $\kappa>0$  such that 
\begin{equation}  \label{Ph4}
\frac{\kappa^2}{2}> C_1\vee C_2,
\end{equation}
 where  $C_1$ and $C_2$ will be fixed below.
Consider a subset $\Upsilon(S) $ of $\Xi (S)$ which consists of all  processes $Y\in \Xi (S)$  such that
\begin{equation}  \label{Ph5}
\max\left\{\sup_{0<t\leq S} t^{2(\eta-\beta)} \mathbb E\|A^\eta Y(t)\|^2,
 \sup_{0\leq t\leq S}\mathbb E\|A^\beta Y(t)\|^2\right\} \leq \kappa^2.
\end{equation}
Obviously, $\Upsilon(S) $  is a nonempty closed subset of $\Xi (S)$.

For $Y\in \Upsilon(S)$,  define a function $\Phi Y$ on $[0,S]$  by 
\begin{align}  
\Phi Y(t)=&S(t) \xi+\int_0^t S(t-s)[F_1(s)+F_2(Y(s))]ds \label{Ph6} \\
&+ \int_0^t S(t-s)G(s) dW(s).   \notag
\end{align}
Our goal is then to verify that $\Phi$ is a contraction mapping from $\Upsilon(S)$ into itself, provided that $S$ is sufficiently small, and that the fixed point of $\Phi$ is the desired solution of \eqref{E2}. For this purpose, we divide the proof into six steps.

{\bf Step 1}. Let $Y\in \Upsilon(S)$. Let us verify that 
$$ \Phi Y \in \Upsilon(S). $$

 For $\beta\leq \theta < \frac{1}{2}$,   \eqref{Ph5} and   \eqref{Ph6} give
\begin{align*}
&t^{2(\theta-\beta)}\mathbb E\|A^\theta\Phi Y(t)\|^2   \notag\\   
 \leq & 3t^{2(\theta-\beta)}\mathbb E \Big[ \|A^\theta S(t) \xi\|^2+\Big|\Big|\int_0^tA^\theta S(t-s)[F_1(s)+F_2(Y(s))]ds\Big|\Big|^2 \notag\\
&+\Big|\Big|\int_0^tA^\theta S(t-s) G(s) dW(s)\Big|\Big|^2 \Big]  \notag\\
\leq & 3t^{2(\theta-\beta)}  \|A^{\theta-\beta} S(t)\|^2\mathbb E\|A^\beta \xi\|^2\notag\\
&+6t^{2(\theta-\beta)}\mathbb E\Big|\Big|\int_0^t A^\theta S(t-s)F_1(s)ds\Big|\Big|^2 \notag\\
&+6t^{2(\theta-\beta)}\mathbb E\Big|\Big|\int_0^tA^\theta S(t-s)F_2(Y(s))ds\Big|\Big|^2 \notag\\
& +3c(E) t^{2(\theta-\beta)} \int_0^t \mathbb E\|A^\theta S(t-s) G(s)\|_{\gamma(H;E)}^2 ds. \notag
\end{align*}
On  account of \eqref{H12},  \eqref{H10} and \eqref{H11}, 
\begin{align*}
&t^{2(\theta-\beta)}\mathbb E\|A^\theta\Phi Y(t)\|^2   \notag\\
\leq  &  3\iota_{\theta-\beta}^2 \mathbb E\|A^\beta \xi\|^2+6\iota_\theta^2 t^{2(\theta-\beta)} \mathbb E \|F_1\|_{\mathcal F^{\beta,\sigma}(E)}^2 \Big[\int_0^t(t-s)^{-\theta}s^{\beta-1}ds\Big]^2 \notag\\
&+6\iota_\theta^2 t^{2(\theta-\beta)} \mathbb E \Big[\int_0^t (t-s)^{-\theta}\|F_2(Y(s))\|ds\Big]^2\notag\\
&+3c(E) t^{2(\theta-\beta)} \int_0^t\|A^{\theta-\delta}\|^2\|S(t-s)\|^2 \mathbb E\|A^\delta G(s)\|_{\gamma(H;E)}^2 ds \notag\\
\leq  &  3\iota_{\theta-\beta}^2 \mathbb E\|A^\beta \xi\|^2+ 6\iota_\theta^2 t^{2(\theta-\beta)} \mathbb E \|F_1\|_{\mathcal F^{\beta,\sigma}(E)}^2 \Big[\int_0^t(t-s)^{-\theta}s^{\beta-1}ds\Big]^2\notag\\
&+ 6\iota_\theta^2 t^{1+2(\theta-\beta)}   \int_0^t (t-s)^{-2\theta}\mathbb E \|F_2(Y(s))\|^2ds \notag\\
&+3 c(E)\iota_0^2 \|A^{\theta-\delta}\|^2 \mathbb E \|A^\delta G\|_{\mathcal F^{\beta+\frac{1}{2},\sigma}(\gamma(H;E))}^2  t^{2(\theta-\beta)}\int_0^te^{-2\nu(t-s)} s^{2\beta-1} ds  \notag\\
= &  3\iota_{\theta-\beta}^2 \mathbb E\|A^\beta \xi\|^2+6 \iota_\theta^2 \mathbb E \|F_1\|_{\mathcal F^{\beta,\sigma}(E)}^2  B (\beta,1-\theta)^2\notag\\
& +6\iota_\theta^2 t^{1+2(\theta-\beta)}   \int_0^t (t-s)^{-2\theta}\mathbb E \|F_2(Y(s))\|^2ds \notag\\
&+3 c(E)\iota_0^2 \chi(\theta,\beta,\nu) \|A^{\theta-\delta}\|^2  \mathbb E \|A^\delta G\|_{\mathcal F^{\beta+\frac{1}{2},\sigma}(\gamma(H;E))}^2,  \notag
\end{align*}
where 
$$\chi(\theta,\beta,\nu)=\sup_{0\leq t<\infty}  t^{2(\theta-\beta)} \int_0^te^{-2\nu(t-s)} s^{2\beta-1} ds<\infty. $$

On the other hand, due to {\rm (F2a)}  and \eqref{Ph5}, 
\begin{align}
\mathbb E\|F_2(Y(t))\|^2 &\leq \mathbb E[c_{F_2}\|A^\eta Y(t)\|+\|F_2(0)\|]^2\notag \\
&\leq 2[c_{F_2}^2 \mathbb E \|A^\eta Y(t)\|^2 +\mathbb E\|F_2(0)\|^2]\label{Ph7}\\
&\leq 2[c_{F_2}^2 \kappa^2 t^{2(\beta-\eta)}   +\mathbb E\|F_2(0)\|^2], \hspace{1cm} 0<t\leq S.   \label{Ph8}
\end{align}
Thereby,
\begin{align}
t^{2(\theta-\beta)}&\mathbb E\|A^\theta\Phi Y(t)\|^2   \notag\\
\leq &  3\iota_{\theta-\beta}^2 \mathbb E\|A^\beta \xi\|^2+6 \iota_\theta^2 \mathbb E \|F_1\|_{\mathcal F^{\beta,\sigma}(E)}^2  B (\beta, 1-\theta)^2\notag\\
&+12 \iota_\theta^2 t^{1+2(\theta-\beta)}   \int_0^t (t-s)^{-2\theta}[c_{F_2}^2 \kappa^2  s^{2(\beta-\eta)}   +\mathbb E\|F_2(0)\|^2]ds\notag\\
&+3 c(E)\iota_0^2 \chi(\theta,\beta,\nu) \|A^{\theta-\delta}\|^2  \mathbb E \|A^\delta G\|_{\mathcal F^{\beta+\frac{1}{2},\sigma}(\gamma(H;E))}^2  \notag\\
= &  3\iota_{\theta-\beta}^2 \mathbb E\|A^\beta \xi\|^2+6 \iota_\theta^2 \mathbb E \|F_1\|_{\mathcal F^{\beta,\sigma}(E)}^2  B (\beta, 1-\theta)^2 \notag\\
&+12 \iota_\theta^2 c_{F_2}^2 \kappa^2  t^{1+2(\theta-\eta)}\int_0^t (t-s)^{-2\theta} s^{2(\beta-\eta)} ds\notag\\
&+\frac{12 \iota_\theta^2 \mathbb E\|F_2(0)\|^2}{1-2\theta} t^{2(1-\beta)}\notag\\
&+3 c(E)\iota_0^2 \chi(\theta,\beta,\nu) \|A^{\theta-\delta}\|^2  \mathbb E \|A^\delta G\|_{\mathcal F^{\beta+\frac{1}{2},\sigma}(\gamma(H;E))}^2     \notag\\
= &  3\iota_{\theta-\beta}^2 \mathbb E\|A^\beta \xi\|^2+6 \iota_\theta^2 B (\beta, 1-\theta)^2 \mathbb E \|F_1\|_{\mathcal F^{\beta,\sigma}(E)}^2      \label{Ph9} \\
&+3 c(E)\iota_0^2 \chi(\theta,\beta,\nu) \|A^{\theta-\delta}\|^2  \mathbb E \|A^\delta G\|_{\mathcal F^{\beta+\frac{1}{2},\sigma}(\gamma(H;E))}^2      \notag\\
&+ 12\iota_\theta^2 c_{F_2}^2 \kappa^2   B ( 1+2\beta-2\eta, 1-2\theta) t^{2(1+\beta-2\eta)} \notag\\
&+\frac{12 \iota_\theta^2 \mathbb E\|F_2(0)\|^2}{1-2\theta} t^{2(1-\beta)}.   \notag
\end{align}

We apply these estimates with  $\theta=\eta$ and $\theta=\beta$. Put 
\begin{equation} \label{Ph10}
\begin{aligned}
C_1=&3\iota_{\eta-\beta}^2 \mathbb E\|A^\beta \xi\|^2+6 \iota_\eta^2 B (\beta, 1-\eta)^2 \mathbb E \|F_1\|_{\mathcal F^{\beta,\sigma}(E)}^2  \\
&+3 c(E)\iota_0^2 \chi(\eta,\beta,\nu) \|A^{\eta-\delta}\|^2  \mathbb E \|A^\delta G\|_{\mathcal F^{\beta+\frac{1}{2},\sigma}(\gamma(H;E))}^2,\\
C_2=&3\iota_0^2 \mathbb E\|A^\beta \xi\|^2+6 \iota_\beta^2 B (\beta, 1-\beta)^2 \mathbb E \|F_1\|_{\mathcal F^{\beta,\sigma}(E)}^2  \\
&+3 c(E)\iota_0^2 \chi(\beta,\beta,\nu) \|A^{\beta-\delta}\|^2  \mathbb E \|A^\delta G\|_{\mathcal F^{\beta+\frac{1}{2},\sigma}(\gamma(H;E))}^2,
\end{aligned}
\end{equation}
 and take $S$ to be  small enough. By using \eqref{Ph4},  
\begin{align}
t^{2(\eta-\beta)}& \mathbb E\|A^\eta \Phi Y(t)\|^2 \label{Ph11}\\
  \leq & C_1+ 12\iota_\eta^2 c_{F_2}^2 \kappa^2   B ( 1+2\beta-2\eta, 1-2\eta) t^{2(1+\beta-2\eta)} \notag\\
&+\frac{12 \iota_\eta^2 \mathbb E\|F_2(0)\|^2}{1-2\eta} t^{2(1-\beta)}   \notag \\
\leq & \kappa^2,  \hspace{1cm} 0<t\leq S,  \notag
\end{align}
and 
\begin{align}
\mathbb E & \|A^\beta\Phi Y(t)\|^2    \label{Ph12} \\
  \leq& C_2+ 12\iota_\beta^2 c_{F_2}^2 \kappa^2   B ( 1+2\beta-2\eta, 1-2\beta) t^{2(1+\beta-2\eta)}\notag\\
 &+\frac{12 \iota_\beta^2 \mathbb E\|F_2(0)\|^2}{1-2\beta} t^{2(1-\beta)}\notag \\
\leq & \kappa^2,  \hspace{1cm} 0<t\leq S.  \notag
\end{align}
 We have thus shown that
 \begin{equation*} 
\max\left\{\sup_{0<t\leq S} t^{2(\eta-\beta)} \mathbb E\|A^\eta \Phi Y(t)\|^2,
 \sup_{0\leq t\leq S}\mathbb E\|A^\beta \Phi Y(t)\|^2\right\} \leq \kappa^2.
\end{equation*}
  In addition, it is clear that $\Phi Y(0)=\xi\in \mathcal D(A^\beta)$ a.s. Hence,
$$ \Phi Y \in \Upsilon(S).$$

{\bf Step 2}. Let us show that $\Phi$ is a contraction mapping, provided $S>0$ is sufficiently small. Let $Y_1,Y_2\in \Upsilon (S)$ and $0\leq \theta < \frac{1}{2}.$  It follows from  \eqref{H10} and  \eqref{Ph6}  that
\begin{align*}
t^{2(\theta-\beta)}&\mathbb E\|A^\theta[\Phi Y_1(t)-\Phi Y_2(t)]\|^2  \notag \\
=&t^{2(\theta-\beta)} \mathbb E\Big|\Big|\int_0^t A^\theta S(t-s)[F_2(Y_1(s))-F_2(Y_2(s))]ds\Big|\Big|^2 \notag\\
\leq&t^{2(\theta-\beta)} \mathbb E\Big[\int_0^t \|A^\theta S(t-s)\| \|F_2(Y_1(s))-F_2(Y_2(s))\|ds\Big]^2 \notag \\
\leq& \iota_\theta^2 t^{2(\theta-\beta)} \mathbb E\Big[\int_0^t (t-s)^{-\theta} \|F_2(Y_1(s))-F_2(Y_2(s))\|ds\Big]^2. \notag
\end{align*}
Thanks to  {\rm (F2a)} and  \eqref{Ph3}, 
\begin{align*}
t^{2(\theta-\beta)}& \mathbb E\|A^\theta[\Phi Y_1(t)-\Phi Y_2(t)]\|^2  \notag \\
\leq& c_{F_2}^2\iota_\theta^2 t^{2(\theta-\beta)} \mathbb E\Big[\int_0^t (t-s)^{-\theta}   \|A^\eta(Y_1(s)-Y_2(s))\|ds\Big]^2 \notag\\
\leq& c_{F_2}^2\iota_\theta^2 t^{1+2(\theta-\beta)} \mathbb E\int_0^t (t-s)^{-2\theta}   \|A^\eta(Y_1(s)-Y_2(s))\|^2ds \notag\\
=& c_{F_2}^2\iota_\theta^2 t^{1+2(\theta-\beta)} \int_0^t (t-s)^{-2\theta}  \mathbb E \|A^\eta(Y_1(s)-Y_2(s))\|^2ds \notag\\
\leq&c_{F_2}^2\iota_\theta^2 t^{1+2(\theta-\beta)} \int_0^t  (t-s)^{-2\theta} s^{2(\beta-\eta)} \|Y_1-Y_2\|_{{\Xi (S)}}^2ds \notag\\
=&c_{F_2}^2\iota_\theta^2   B (1+2\beta-2\eta,1-2\theta) t^{2(1-\eta)}   \|Y_1-Y_2\|_{{\Xi (S)}}^2 \notag\\
\leq &c_{F_2}^2  \iota_\theta^2 B (1+2\beta-2\eta,1-2\theta)   S^{2(1-\eta)} \|Y_1-Y_2\|_{{\Xi (S)}}^2. \notag
\end{align*}
Applying these estimates with $\theta=\eta$ and $\theta=\beta$, we conclude that
\begin{align}
\|\Phi & Y_1-\Phi Y_2\|_{{\Xi (S)}}^2 \notag\\
=&\sup_{0<t\leq S} t^{2(\eta-\beta)} \mathbb E\|A^\eta [\Phi Y_1(t)-\Phi Y_2(t)]\|^2\notag\\
&+ \sup_{0\leq t\leq S}\mathbb E\|A^\beta  [\Phi Y_1(t)-\Phi Y_2(t)]\|^2\notag\\
\leq &c_{F_2}^2  [\iota_\eta^2 B (1+2\beta-2\eta,1-2\eta)+\iota_\beta^2 B (1+2\beta-2\eta,1-2\beta)]                  \label{Ph13}\\
& \times S^{2(1-\eta)} \|Y_1-Y_2\|_{{\Xi (S)}}^2.   \notag
\end{align}
Therefore, $\Phi$ is contraction on $\Upsilon (S),$ provided $S>0$ is sufficiently small.

{\bf Step 3}.  Let  
\begin{equation} \label{E102}
 0<\gamma<\min\left\{\beta+\delta-1,\frac{1}{2}-\eta,\frac{1+2\beta}{4}-\eta\right\}.
\end{equation} 
Let us prove existence of a  mild solution $X$ to \eqref{E2} on $[0,S]$ in the function space
\begin{equation}  \label{H25.5}
X\in  \mathcal C^\gamma([\epsilon,S];\mathcal D(A^\eta))\cap \mathcal C([0,S];\mathcal D(A^\beta)) \hspace{1cm} \text{ a.s.}
\end{equation}
for any $0<\epsilon<S$, provided that $S$ is sufficiently small.

Let $S>0$ be sufficiently small in such a way that $\Phi$ maps $\Upsilon(S)$ into itself and is contraction with respect to the norm of $\Xi (S).$ 
Due to Step 1 and  Step 2, $S=T_{loc}$ can be determined by the exponents, $\mathbb E \|A^\beta\xi\|^2,$ $\mathbb E\|F_2(0)\|^2, \mathbb E \|F_1\|_{\mathcal F^{\beta,\sigma}(E)}^2$ and $ \mathbb E \|A^\delta G\|_{\mathcal F^{\beta+\frac{1}{2},\sigma}(\gamma(H;E))}^2$.
 Thanks to the fixed point theorem, there exists  $X\in \Upsilon(T_{loc})$ such that $X=\Phi X$.

It therefore suffices to prove that $X$ satisfies  \eqref{H25.5}. For this purpose, we divide $ X$ into two parts: $ X(t)=X_1(t)+ I_2(t),$
where 
\begin{equation}  \label{H26.5}
{X_1}(t)=S(t) \xi+\int_0^t S(t-s)[F_1(s)+F_2(X(s))]ds,
\end{equation}
and $I_2$ is defined by \eqref{Pt10}.

As seen by Step 3 of the proof of Theorem \ref{Th2}, 
$$
I_2(t)+ \int_0^t AI_2(s)ds=\int_0^t G(s)dW(s), \hspace{1cm}  0\leq t\leq T_{loc} $$
and $
I_2 \in \mathcal C^{\gamma}([0,T_{loc}];\mathcal D(A))$ a.s.

Let $0<\epsilon<S.$ 
Since 
$$\mathcal C^{\gamma}([0,T_{loc}];\mathcal D(A)) \subset \mathcal C^\gamma([\epsilon,T_{loc}];\mathcal D(A^\eta))\cap \mathcal C([0,T_{loc}];\mathcal D(A^\beta)), $$  
  what we need is  to prove that  
\begin{equation}   \label{H27.4}
X_1\in  \mathcal C^\gamma([\epsilon,T_{loc}];\mathcal D(A^\eta))\cap \mathcal C([0,T_{loc}];\mathcal D(A^\beta))     \hspace{1cm} \text{ a.s.  }
\end{equation}

To this end, we  use the Kolmogorov continuity theorem. 
For $0<s<t\leq T_{loc}$, by the semigroup property, 
\begin{align}
X_1& (t)-X_1(s)     \label{H27.5}\\
=&S(t-s)S(s)\xi+ S(t-s)\int_0^s S(s-r)[F_1(r)+F_2(X(r))]dr   \notag\\
& -X_1(s)+\int_s^t S(t-r)[F_1(r)+F_2(X(r))]dr     \notag\\
=&[S(t-s)-I]X_1(s)+\int_s^t S(t-r)[F_1(r)+F_2(X(r))]dr.       \notag
\end{align}
Let $\frac{1}{2}< \rho <  1-\eta$. Due to  \eqref{H12},  \eqref{H10},  \eqref{H11.5}  and  \eqref{H27.5}, 
\begin{align*}
\|A^\eta&[X_1(t)-X_1(s)]\|  \notag\\
\leq &\|[S(t-s)-I]A^{-\rho}\|  \|A^{\eta+\rho}X_1(s)\|\notag\\
&+\int_s^t \|A^\eta S(t-r)\| [\|F_1(r)\|+\|F_2(X(r))\|]dr \notag\\
\leq &\frac{\iota_{1-\rho}(t-s)^\rho}{\rho} \Big \|A^{\eta+\rho}\Big[S(s) \xi+\int_0^s S(s-r)[F_1(r)+F_2(X(r))]dr\Big]\Big \| \notag \\
&+\iota_\eta\int_s^t (t-r)^{-\eta} [\|F_1(r)\|+\|F_2(X(r))\|]dr \notag\\
\leq &\frac{\iota_{1-\rho}(t-s)^\rho}{\rho}  \|A^{\eta+\rho-\beta}S(s)\| \|A^\beta  \xi\|\notag\\
&+\frac{\iota_{1-\rho}(t-s)^\rho}{\rho} \int_0^s \|A^{\eta+\rho} S(s-r)\| \|F_1(r)\|dr\notag\\
&+\frac{\iota_{1-\rho}(t-s)^\rho}{\rho} \int_0^s \|A^{\eta+\rho} S(s-r)\| \|F_2(X(r))\|dr \notag \\
&+\iota_\eta\int_s^t (t-r)^{-\eta} \|F_1(r)\|dr  +\iota_\eta\int_s^t (t-r)^{-\eta} \|F_2(X(r))\|dr \notag\\
\leq &\frac{\iota_{1-\rho}\iota_{\eta+\rho-\beta}(t-s)^\rho}{\rho}  s^{-\eta-\rho+\beta} \|A^\beta  \xi\|\notag\\
&+\frac{\iota_{1-\rho}\iota_{\eta+\rho}  \|F_1\|_{\mathcal F^{\beta,\sigma}(E)}(t-s)^\rho}{\rho} \int_0^s (s-r)^{-\eta-\rho}   r^{\beta-1} dr \notag \\
&+\frac{\iota_{1-\rho}\iota_{\eta+\rho} (t-s)^\rho}{\rho}\int_0^s  (s-r)^{-\eta-\rho} \|F_2(X(r))\|dr \notag\\
&+\iota_\eta \|F_1\|_{\mathcal F^{\beta,\sigma}(E)}\int_s^t (t-r)^{-\eta} r^{\beta-1} dr \notag\\
& +\iota_\eta\int_s^t (t-r)^{-\eta} \|F_2(X(r))\|dr\notag \\
= &\frac{\iota_{1-\rho} \iota_{\eta+\rho-\beta}}{\rho}\|A^\beta  \xi\| s^{\beta-\eta-\rho}(t-s)^\rho \notag\\
&+\frac{\iota_{1-\rho} \iota_{\eta+\rho} \|F_1\|_{\mathcal F^{\beta,\sigma}(E)} B (\beta,1-\eta-\rho)}{\rho}  s^{\beta-\eta-\rho}(t-s)^\rho \notag \\
&+\iota_\eta \|F_1\|_{\mathcal F^{\beta,\sigma}(E)}\int_s^t (t-r)^{-\eta}r^{\beta-1}dr \notag\\
&+\frac{\iota_{1-\rho}\iota_{\eta+\rho} (t-s)^\rho}{\rho}\int_0^s  (s-r)^{-\eta-\rho} \|F_2(X(r))\|dr \notag\\
& +\iota_\eta\int_s^t (t-r)^{-\eta} \|F_2(X(r))\|dr.\notag 
\end{align*}
Dividing $\beta-1$ as $\beta-1=(\eta+\rho-1)+(\beta-\eta-\rho),$ it is seen that
\begin{align*}
\int_s^t (t-r)^{-\eta}  r^{\beta-1} dr &\leq \int_s^t (t-r)^{-\eta}  (r-s)^{\eta+\rho-1} s^{\beta-\eta-\rho} dr\\
&=  B (\eta+\rho,1-\eta) s^{\beta-\eta-\rho}(t-s)^\rho.
\end{align*}
Hence, 
\begin{align*}
\|A^\eta &[X_1(t)-X_1(s)]\|  \notag\\
\leq  &\frac{\iota_{1-\rho}\iota_{\eta+\rho-\beta}}{\rho} \|A^\beta  \xi\| s^{\beta-\eta-\rho}(t-s)^\rho \notag\\
&+\Big[\frac{\iota_{1-\rho} \iota_{\eta+\rho}  B (\beta,1-\eta-\rho)}{\rho} +\iota_\eta B (\eta+\rho,1-\eta)\Big]\notag\\
&  \times  \|F_1\|_{\mathcal F^{\beta,\sigma}(E)}s^{\beta-\eta-\rho}(t-s)^\rho \notag\\
&+\frac{\iota_{1-\rho}\iota_{\eta+\rho}}{\rho} (t-s)^\rho \int_0^s  (s-r)^{-\eta-\rho} \|F_2(X(r))\|dr  \notag\\
&+\iota_\eta\int_s^t (t-r)^{-\eta} \|F_2(X(r))\|dr, \hspace{1cm} 0<s<t\leq T_{loc}.\notag 
\end{align*}

Taking expectation of the square of the both hand sides of the above inequality, it follows that 
\begin{align*}
\mathbb E&\|A^\eta[X_1(t)-X_1(s)]\|^2 \notag\\
\leq &\frac{4\iota_{1-\rho}^2\iota_{\eta+\rho-\beta}^2}{\rho^2}  \mathbb E\|A^\beta  \xi\|^2 s^{2(\beta-\eta-\rho)}(t-s)^{2\rho} \notag\\
& +4\Big[\frac{\iota_{1-\rho}\iota_{\eta+\rho}  B (\beta,1-\eta-\rho)}{\rho} +\iota_\eta B (\eta+\rho,1-\eta)\Big]^2\notag\\
&\times \mathbb E \|F_1\|_{\mathcal F^{\beta,\sigma}(E)}^2s^{2(\beta-\eta-\rho)}(t-s)^{2\rho} \notag\\
&+\frac{4\iota_{1-\rho}^2\iota_{\eta+\rho}^2}{\rho^2} (t-s)^{2\rho}\mathbb E\Big[ \int_0^s  (s-r)^{-\eta-\rho} \|F_2(X(r))\|dr\Big]^2  \notag\\
&+4\iota_\eta^2\mathbb E\Big[\int_s^t (t-r)^{-\eta} \|F_2(X(r))\|dr\Big]^2, \hspace{1cm} 0<s<t\leq T_{loc}.\notag
\end{align*}

Since 
\begin{align*}
&\Big[ \int_0^s  (s-r)^{-\eta-\rho} \|F_2(X(r))\|dr\Big]^2  \\
&=\Big[ \int_0^s  (s-r)^{\frac{-\eta-\rho}{2}} (s-r)^{\frac{-\eta-\rho}{2}}\|F_2(X(r))\|dr\Big]^2  \\
&\leq \int_0^s  (s-r)^{-\eta-\rho}dr \int_0^s (s-r)^{-\eta-\rho}\|F_2(X(r))\|^2dr\\
&=\frac{s^{1-\eta-\rho}}{1-\eta-\rho} \int_0^s (s-r)^{-\eta-\rho}\|F_2(X(r))\|^2dr
\end{align*}
and
\begin{align*}
\Big[ \int_s^t  (t-r)^{-\eta} \|F_2(X(r))\|dr\Big]^2  \leq (t-s) \int_s^t  (t-r)^{-2\eta}\|F_2(X(r))\|^2 dr,
\end{align*}
we have
\begin{align}
\mathbb E&\|A^\eta[X_1(t)-X_1(s)]\|^2       \label{Ph14}\\
\leq &\frac{4\iota_{1-\rho}^2  \iota_{\eta+\rho-\beta}^2}{\rho^2} \mathbb E\|A^\beta  \xi\|^2 s^{2(\beta-\eta-\rho)}(t-s)^{2\rho} \notag\\
& +4\Big[\frac{\iota_{1-\rho}\iota_{\eta+\rho}  B (\beta,1-\eta-\rho)}{\rho}  +\iota_\eta B (\eta+\rho,1-\eta)\Big]^2\notag\\
&\times \mathbb E \|F_1\|_{\mathcal F^{\beta,\sigma}(E)}^2s^{2(\beta-\eta-\rho)}(t-s)^{2\rho} \notag\\
&+\frac{4\iota_{1-\rho}^2\iota_{\eta+\rho}^2}{\rho^2(1-\eta-\rho)} (t-s)^{2\rho} s^{1-\eta-\rho}     \int_0^s  (s-r)^{-\eta-\rho} \mathbb E\|F_2(X(r))\|^2dr  \notag\\
&+4\iota_\eta^2(t-s) \int_s^t (t-r)^{-2\eta}\mathbb E \|F_2(X(r))\|^2dr.  \notag
\end{align}

Two integrals in the right-hand side of  \eqref{Ph14} can be estimated as follows.
Thanks to  \eqref{Ph8},   
\begin{align}
 \int_0^s & (s-r)^{-\eta-\rho} \mathbb E\|F_2(X(r))\|^2dr \notag\\
 \leq &
2 \int_0^s  (s-r)^{-\eta-\rho} [c_{F_2}^2 \kappa^2 r^{2(\beta-\eta)} +\mathbb E\|F_2(0)\|^2] dr  \notag\\
 = &
2c_{F_2}^2 \kappa^2   B (1+2\beta-2\eta,1-\eta-\rho) s^{1+2\beta-3\eta-\rho}  \label{Ph15}\\
&+\frac{2\mathbb E\|F_2(0)\|^2 s^{1-\eta-\rho}}{1-\eta-\rho}, \notag
\end{align}
and
\begin{align}
&\int_s^t (t-r)^{-2\eta}\mathbb E \|F_2(X(r))\|^2dr\notag\\
&\leq 2\int_s^t (t-r)^{-2\eta}[c_{F_2}^2 \kappa^2  r^{2(\beta-\eta)} + \mathbb E\|F_2(0)\|^2]dr\notag\\
&=2 c_{F_2}^2 \kappa^2  \int_s^t (t-r)^{-2\eta}r^{2(\beta-\eta)} dr+\frac{2\mathbb E\|F_2(0)\|^2}{1-2\eta} (t-s)^{1-2\eta}.\label{Ph16}
\end{align}
The latter integral can be evaluated by dividing $2(\beta-\eta)$ as $2(\beta-\eta)=(\beta-\frac{1}{2})+(\frac{1}{2}+\beta-2\eta):$ 
\begin{align}
\int_s^t (t-r)^{-2\eta} r^{2(\beta-\eta)} dr&\leq \int_s^t (t-r)^{-2\eta} (r-s)^{\beta-\frac{1}{2}} t^{\frac{1}{2}+\beta-2\eta}dr\notag\\
&= B (\frac{1}{2}+\beta,1-2\eta) t^{\frac{1}{2}+\beta-2\eta} (t-s)^{\frac{1}{2}+\beta-2\eta}.  \label{Ph17}
\end{align}
By virtue of  \eqref{Ph14},  \eqref{Ph15},  \eqref{Ph16} and  \eqref{Ph17},  
\begin{align*}
\mathbb E&\|A^\eta[X_1(t)-X_1(s)]\|^2 \notag\\
\leq  &\frac{4\iota_{1-\rho}^2  \iota_{\eta+\rho-\beta}^2}{\rho^2} \mathbb E\|A^\beta  \xi\|^2 s^{2(\beta-\eta-\rho)}(t-s)^{2\rho} \notag\\
& +4\Big[\frac{\iota_{1-\rho}\iota_{\eta+\rho}  B (\beta,1-\eta-\rho)}{\rho}  +\iota_\eta B (\eta+\rho,1-\eta)\Big]^2\notag\\
&\times \mathbb E \|F_1\|_{\mathcal F^{\beta,\sigma}(E)}^2s^{2(\beta-\eta-\rho)}(t-s)^{2\rho}  \notag\\
&+\frac{8\iota_{1-\rho}^2\iota_{\eta+\rho}^2c_{F_2}^2 \kappa^2   B (1+2\beta-2\eta,1-\eta-\rho)}{\rho^2(1-\eta-\rho)} (t-s)^{2\rho} s^{2(1+\beta-2\eta-\rho)}  \notag\\
&+\frac{8\iota_{1-\rho}^2\iota_{\eta+\rho}^2\mathbb E\|F_2(0)\|^2}{\rho^2(1-\eta-\rho)^2} (t-s)^{2\rho} s^{2(1-\eta-\rho)}  \notag\\
&+8 \iota_\eta^2c_{F_2}^2 \kappa^2   B (\frac{1}{2}+\beta,1-2\eta) t^{\frac{1}{2}+\beta-2\eta} (t-s)^{\frac{3}{2}+\beta-2\eta}\notag\\
&+\frac{8\iota_\eta^2\mathbb E\|F_2(0)\|^2}{1-2\eta} (t-s)^{2(1-\eta)}, \hspace{1cm} 0<s<t\leq T_{loc}.
\end{align*}
In particular, 
\begin{align*}
\mathbb E&\|A^\eta[X_1(t)-X_1(s)]\|^2 \notag\\
\leq  &\frac{4\iota_{1-\rho}^2  \iota_{\eta+\rho-\beta}^2}{\rho^2} \mathbb E\|A^\beta  \xi\|^2 \epsilon^{2(\beta-\eta-\rho)}(t-s)^{2\rho} \notag\\
& +4\Big[\frac{\iota_{1-\rho}\iota_{\eta+\rho}  B (\beta,1-\eta-\rho)}{\rho}  +\iota_\eta B (\eta+\rho,1-\eta)\Big]^2\notag\\
&\times \mathbb E \|F_1\|_{\mathcal F^{\beta,\sigma}(E)}^2\epsilon^{2(\beta-\eta-\rho)}(t-s)^{2\rho}  \notag\\
&+\frac{8\iota_{1-\rho}^2\iota_{\eta+\rho}^2c_{F_2}^2 \kappa^2   B (1+2\beta-2\eta,1-\eta-\rho)}{\rho^2(1-\eta-\rho)} \notag\\
&\times\max\{\epsilon^{2(1+\beta-2\eta-\rho)}, T_{loc}^{2(1+\beta-2\eta-\rho)} \}(t-s)^{2\rho}   \notag\\
&+\frac{8\iota_{1-\rho}^2\iota_{\eta+\rho}^2\mathbb E\|F_2(0)\|^2}{\rho^2(1-\eta-\rho)^2} T_{loc}^{2(1-\eta-\rho)} (t-s)^{2\rho}   \notag\\
&+8 \iota_\eta^2c_{F_2}^2 \kappa^2   B (\frac{1}{2}+\beta,1-2\eta) T_{loc}^{\frac{1}{2}+\beta-2\eta} (t-s)^{\frac{3}{2}+\beta-2\eta}\notag\\
&+\frac{8\iota_\eta^2\mathbb E\|F_2(0)\|^2}{1-2\eta} (t-s)^{2(1-\eta)}, \hspace{1cm} \epsilon \leq s<t\leq T_{loc}.
\end{align*}

Since this estimate holds true for any $\frac{1}{2}<\rho < 1-\eta,$ and since  $1<\frac{3}{2}+\beta-2\eta<2(1-\eta)$, Theorem \ref{Th-1}  provides that 
for $0<\alpha<\min\{\frac{1}{2}-\eta,\frac{1+2\beta}{4}-\eta\}$
$$X_1\in \mathcal C^\alpha ([\epsilon,T_{loc}];\mathcal D(A^\eta)) \hspace{1cm} \text{a.s.}$$
As a consequence, by \eqref{E102},  
  \begin{equation}   \label{H28.5}
X_1\in \mathcal C^\gamma ([\epsilon,T_{loc}];\mathcal D(A^\eta)) \hspace{1cm} \text{a.s.}
\end{equation}
Since $A^\beta X_1=A^{\beta-\eta} A^\eta X_1$, $A^\beta X_1$ is continuous on $(0,T_{loc}]$.

It remains to prove  that $A^\beta X_1$ is continuous at $t=0$. We already know by Step 1 of the proof of Theorem \ref{Th3} that the process 
$$A^\beta S(\cdot) \xi+\int_0^\cdot S(\cdot-s)F_1(s)ds$$
is continuous at $t=0.$  Meanwhile,   \eqref{H10} and \eqref{Ph8} give
\begin{align} 
\mathbb E&\Big \|A^\beta \int_0^t S(t-s)F_2(X(s))ds\Big \|^2 \notag\\
 & \leq\mathbb E\Big[\int_0^t \|A^\beta S(t-s)\| \|F_2(X(s))\|ds\Big]^2\notag\\
 &\leq \iota_\beta^2 \mathbb E  \Big[\int_0^t (t-s)^{-\beta}\|F_2(X(s))\| ds\Big]^2\notag\\
 &\leq \iota_\beta^2 t   \int_0^t (t-s)^{-2\beta}\mathbb E\|F_2(X(s))\|^2 ds\notag\\
 &\leq 2\iota_\beta^2 t   \int_0^t (t-s)^{-2\beta}[c_{F_2}^2 \kappa^2  s^{2(\beta-\eta)} + \mathbb E\|F_2(0)\|^2] ds\notag\\
 &=  2\iota_\beta^2  c_{F_2}^2 \kappa^2   B (1+2\beta-2\eta,1-2\beta)t^{2(1-\eta)}      \label{Ph18} \\
& \quad  +\frac{2\iota_\beta^2 \mathbb E\|F_2(0)\|^2}{1-2\beta} t^{2(1-\beta)}   \to 0 \hspace{1cm} \text{as} \quad  t\to 0. \notag
\end{align}
Therefore, there exists a decreasing sequence $\{t_n, n=1,2,3,\dots\}$ converging to $0$ such that 
$$\lim_{n\to\infty} A^\beta \int_0^{t_n} S(t_n-s)F_2(X(s))ds=0 \hspace{1cm} \text{a.s.}$$
By the continuity of $A^\beta \int_0^\cdot S(\cdot-s)F_2(X(s))ds$ on $(0,T_{loc}]$, it implies that 
$$\lim_{t\to 0} A^\beta \int_0^t S(t-s)F_2(X(s))ds=0 \hspace{1cm} \text{a.s.,}$$
i.e. $A^\beta \int_0^\cdot S(\cdot-s)F_2(X(s))ds$ is continuous at $t=0$.
In this way,  we conclude that
\begin{align*}
A^\beta X_1=&A^\beta\Big[S(\cdot) \xi+\int_0^\cdot S(\cdot-s)F_1(s)ds \Big]\\
&+A^\beta \int_0^\cdot S(\cdot-s)F_2(X(s))ds
\end{align*}
is continuous at $t=0.$

{\bf Step 4}. Let us prove the estimate \eqref{Ph2}.
We have
\begin{align}
A^\beta X(t)=&A^\beta \Big[S(t) \xi+\int_0^t S(t-s)F_1(s)ds + \int_0^t S(t-s)G(s) dW(s) \Big] \notag\\
&+A^\beta \int_0^t S(t-s)F_2(X(s))ds        \notag \\
=&A^\beta X_3(t)+A^\beta X_4(t), \hspace{1cm} 0\leq t\leq T_{loc}.       \label{Ph19} 
\end{align}
On  account of  Theorem \ref{Th3}, it is easily seen that 
\begin{align*} 
\mathbb E \|A^\beta X_3(t)\|^2 \leq & \rho_1 [\mathbb E\|A^\beta \xi\|^2 + \mathbb E \|F_1\|_{\mathcal F^{\beta,\sigma}(E)}^2\\
&+  \mathbb E \|A^\delta G\|_{\mathcal F^{\beta+\frac{1}{2},\sigma}(\gamma(H;E))}^2 t^{2\beta}], \hspace{1cm} 0\leq t\leq T_{loc},
\end{align*}
where $\rho_1$ is a positive constant depending only on the exponents.
Meanwhile,     \eqref{Ph10} and  \eqref{Ph18} give that there exists $\rho_2>0$ depending only on the exponents such that 
\begin{align*}
\mathbb E\| A^\beta X_4(t)\|^2
\leq & \rho_2 [
\mathbb E\|A^\beta \xi\|^2+ \mathbb E \|F_1\|_{\mathcal F^{\beta,\sigma}(E)}^2  + \mathbb E \|A^\delta G\|_{\mathcal F^{\beta+\frac{1}{2},\sigma}(\gamma(H;E))}^2
]
t^{2(1-\eta)}\\
& +\rho_2  \mathbb E\|F_2(0)\|^2   t^{2(1-\beta)}, \hspace{1cm} 0\leq t\leq T_{loc}.
\end{align*}
Thus,
\begin{align*}
\mathbb E\| A^\beta X(t)\|^2 
\leq & 2\mathbb E\| A^\beta X_3(t)\|^2
+2 \mathbb E\| A^\beta X_4(t)\|^2\\
\leq & C[\mathbb E\|A^\beta \xi\|^2 + \mathbb E \|F_1\|_{\mathcal F^{\beta,\sigma}(E)}^2] [1+t^{2(1-\eta)}]\\
&+ C \mathbb E \|A^\delta G\|_{\mathcal F^{\beta+\frac{1}{2},\sigma}(\gamma(H;E))}^2 [t^{2\beta}+t^{2(1-\eta)}]\\
& +C  \mathbb E\|F_2(0)\|^2   t^{2(1-\beta)}, \hspace{1cm} 0\leq t\leq T_{loc},
\end{align*}
with some $C>0$ depending only on the exponents. 

{\bf Step 5}. Let us prove the estimate  \eqref{semilinear evolution equationExpectationAetaXSquare}. By using   \eqref{H12},  \eqref{H10},   \eqref{H11} and \eqref{Ph4}, 
\begin{align*}
\mathbb E&\| A^\eta X(t)\|^2\notag\\
=&\mathbb E  \Big \| A^\eta S(t) \xi+\int_0^t A^\eta S(t-s)F_1(s)ds + \int_0^t A^\eta S(t-s)G(s) dW(s) \notag\\
&+ \int_0^t A^\eta S(t-s)F_2(X(s))ds\Big \|^2        \notag \\
\leq & 4\mathbb E \|A^\eta S(t) \xi\|^2+4\mathbb E   \Big[\int_0^t \|A^\eta S(t-s)\| \|F_1(s)\|ds\Big]^2 \notag\\
&+4\mathbb E\Big[\int_0^t \|A^\eta S(t-s)\| \|F_2(X(s))\|ds\Big]^2\notag\\
&+ 4 c(E) \int_0^t \mathbb E  \|A^{\eta-\delta} S(t-s)\|^2 \|A^\delta G(s)\|_{\gamma(H;E)}^2 ds\notag\\
\leq & 4\mathbb E \|A^{\eta-\beta} S(t) A^\beta \xi\|^2  +4 \iota_\eta^2 \mathbb E\|F_1\|_{\mathcal F^{\beta,\sigma}(E)}^2 \Big[\int_0^t (t-s)^{-\eta} s^{\beta-1} ds\Big]^2  \\
&+ 4 \iota_\eta^2  \mathbb E \Big[\int_0^t (t-s)^{-\eta} \|F_2(X(s))\|ds\Big]^2  \\
&+ 4c(E) \iota_0^2 \|A^{\eta-\delta} \|^2 \mathbb E \|A^\delta G\|_{\mathcal F^{\beta+\frac{1}{2},\sigma}(\gamma(H;E))}^2  \int_0^t e^{-2\nu (t-s)} s^{2\beta-1} ds \\
\leq & 4 \iota_{\eta-\beta}^2 \mathbb E \|A^\beta \xi\|^2 t^{-2(\eta-\beta)} +4 \iota_\eta^2 B(\beta,1-\eta) \mathbb E\|F_1\|_{\mathcal F^{\beta,\sigma}(E)}^2 t^{2(\beta-\eta)}   \\
&+ 4 \iota_\eta^2  t  \int_0^t (t-s)^{-2\eta} \mathbb E\|F_2(X(s))\|^2ds  \\
&+ 4c(E) \iota_0^2 \|A^{\eta-\delta} \|^2 \mathbb E \|A^\delta G\|_{\mathcal F^{\beta+\frac{1}{2},\sigma}(\gamma(H;E))}^2 \\
& \times \sup_{0\leq t<\infty}\int_0^t e^{-2\nu (t-s)} s^{2\beta-1} ds, \hspace{1cm} 0<t\leq T_{loc}.
\end{align*}
On account of \eqref{Ph8}, 
\begin{align*}
4 \iota_\eta^2 & t  \int_0^t (t-s)^{-2\eta} \mathbb E\|F_2(X(s))\|^2ds \\
\leq & 8\iota_\eta^2  t  \int_0^t (t-s)^{-2\eta} [c_{F_2}^2 \kappa^2  s^{2(\beta-\eta)}   +\mathbb E\|F_2(0)\|^2]ds\\
= & 8\iota_\eta^2 \Big[c_{F_2}^2 \kappa^2   B (1+2\beta-2\eta,1-2\eta)t^{2(1+\beta-2\eta)} +\frac{\mathbb E\|F_2(0)\|^2}{1-2\eta} t^{2(1-\eta)}\Big].
\end{align*}
Therefore,
\begin{align*}
\mathbb E&\| A^\eta X(t)\|^2\notag\\
\leq & 4 [\iota_{\eta-\beta}^2 \mathbb E \|A^\beta \xi\|^2  + \iota_\eta^2 B(\beta,1-\eta) \mathbb E\|F_1\|_{\mathcal F^{\beta,\sigma}(E)}^2] t^{2(\beta-\eta)}   \\
&+ 8\iota_\eta^2 \Big[c_{F_2}^2 \kappa^2   B (1+2\beta-2\eta,1-2\eta)t^{2(1+\beta-2\eta)} +\frac{\mathbb E\|F_2(0)\|^2}{1-2\eta} t^{2(1-\eta)}\Big]  \\
&+ 4c(E) \iota_0^2 \|A^{\eta-\delta} \|^2 \mathbb E \|A^\delta G\|_{\mathcal F^{\beta+\frac{1}{2},\sigma}(\gamma(H;E))}^2 \\
& \times \sup_{0\leq t<\infty}\int_0^t e^{-2\nu (t-s)} s^{2\beta-1} ds, \hspace{1cm} 0<t\leq T_{loc}.
\end{align*}
In view of  \eqref{Ph4} and \eqref{Ph10}, it is easily seen that there exists $C>0$  depending only on the exponents such that 
\begin{align*}
\mathbb E&\| A^\eta X(t)\|^2\notag\\
\leq & C [ \mathbb E \|A^\beta \xi\|^2  +  \mathbb E\|F_1\|_{\mathcal F^{\beta,\sigma}(E)}^2] [t^{2(\beta-\eta)}+ t^{2(1+\beta-2\eta)} ]+ C\mathbb E\|F_2(0)\|^2 t^{2(1-\eta)}  \\
&+ C[\sup_{0\leq t<\infty}\int_0^t e^{-2\nu (t-s)} s^{2\beta-1} ds+t^{2(1+\beta-2\eta)}] \\
& \times \mathbb E \|A^\delta G\|_{\mathcal F^{\beta+\frac{1}{2},\sigma}(\gamma(H;E))}^2, \hspace{1cm} 0<t\leq T_{loc}.
\end{align*}
Thus, \eqref{semilinear evolution equationExpectationAetaXSquare} has been verified.

{\bf Step 6}.  Let us  prove uniqueness of   local mild solutions. Let $\bar X$ be any other local mild solution to  \eqref{E2} on the interval $[0,T_{loc}]$ which belongs to the space $\Xi (T_{loc})$.

The formula
$$\bar X(t)=S(t) \xi+\int_0^t S(t-s)[F_2(\bar X(s))+F_1(s)]ds + \int_0^t S(t-s)G(s) dW(s) $$
jointed with 
$$ X(t)=S(t) \xi+\int_0^t S(t-s)[F_2(X(s))+F_1(s)]ds + \int_0^t S(t-s)G(s) dW(s) $$
 yields that
$$X(t)-\bar X(t)=\int_0^t S(t-s)[F_2(X(s))-F_2(\bar X(s))] ds, \hspace{1cm} 0\leq t\leq T_{loc}.$$

We can then repeat the same arguments to Step 2 to deduce that  for  $0<\bar T\leq T_{loc}$, 
\begin{align}
\|X&-\bar X\|_{\Xi (\bar T)}^2\label{Eq9}\\
\leq &c_{F_2}^2  [\iota_\eta^2 B (1+2\beta-2\eta,1-2\eta)+\iota_\beta^2 B (1+2\beta-2\eta,1-2\beta)]\notag\\
&\times  {\bar T}^{2(1-\eta)} \|X-\bar X\|_{\Xi (\bar T)}^2.\notag
\end{align}
Let $\bar T$ be a positive constant such that
\begin{align*}
 &c_{F_2}^2  [\iota_\eta^2 B (1+2\beta-2\eta,1-2\eta)+\iota_\beta^2 B (1+2\beta-2\eta,1-2\beta)]\bar T^{2(1-\eta)}<1.
\end{align*}
Then,   \eqref{Eq9} implies that $X=\bar X$ a.s. on $ [0,\bar T].$

 Repeating the same procedure with  initial time $\bar T$ and  initial value $X(\bar T)=\bar X(\bar T)$, we derive that $X(\bar T +t)=\bar X(\bar T +t)$ a.s. for  $0\leq t\leq \bar T.$ This means that $X=\bar X$ a.s. on a larger interval $[0,2\bar T].$ 

We continue this procedure by finite times, the extended interval can cover the given interval $[0,T_{loc}].$ Therefore,   $X=\bar X$ a.s. on $[0,T_{loc}]$. Thanks to the continuity of $X$ and $\bar X$ on $[0,T_{loc}]$, they are indistinguishable.
\end{proof}

Next, we consider the case where  {\rm (Gb)} holds true.  Assume further that
\begin{itemize}
\item [(F2b)] $F_2\colon\mathcal D(A^\beta)\to E$  satisfies a Lipschitz condition of the form
     \begin{equation*} 
        \|F_2(x)-F_2(y)\|\leq c_{F_2}  \|A^\beta(x-y)\| \hspace{1cm}  \text{ a.s., }  x,y\in \mathcal D(A^\beta),
      \end{equation*}
where $c_{F_2}>0$ is some  constant.
\end{itemize}

\begin{theorem} \label{Th9} 
Let {\rm (Aa)}, {\rm (Ab)}, {\rm (F1)}, {\rm (F2b)} and {\rm (Gb)}   be satisfied.
Assume that  $\xi\in \mathcal D(A^\beta)$ a.s. such that $\mathbb E\|A^\beta\xi\|^2<\infty$.    Then, \eqref{E2} possesses a unique local mild solution $X$ in the function space:
\begin{equation*}
X\in  \mathcal C^\gamma([\epsilon,T_{loc}];\mathcal D(A^\beta))\cap \mathcal C([0,T_{loc}];\mathcal D(A^\theta)) \hspace{1cm} \text{a.s.}
\end{equation*}
for any  $ 0\leq \gamma<\min\{\frac{1}{2}-\beta,\beta\}$,  $0<\epsilon<T$ and $0\leq \theta <\beta$, where $T_{loc}$ is some positive constant in $[0,T]$ depending only on the exponents and $\mathbb E \|F_1\|_{\mathcal F^{\beta,\sigma}(E)}^2$, $\mathbb E \|F_2(0)\|^2, $ $\mathbb E \|A^\beta\xi\|^2,$   $ \mathbb E \|G\|_{\mathcal F^{\beta+\frac{1}{2},\sigma}(\gamma(H;E))}^2.$  
 Furthermore, $X$ satisfies the estimate
\begin{align}
\mathbb E  \|A^\beta X(t)\|^2      
\leq & C\mathbb E \|F_2(0)\|^2t^{2(1-\beta)}    \label{H18.2}\\
& +C[e^{-2\nu t} \mathbb E \|A^\beta\xi\|^2+\mathbb E\|F_1\|_{\mathcal F^{\beta,\sigma}(E)}^2     \notag \\
& +\mathbb E\|G\|_{\mathcal F^{\beta+\frac{1}{2},\sigma}(\gamma(H;E))}^2][1+t^{2(1-\beta)}],   \hspace{0.7cm} 0\leq t\leq T_{loc} \notag
\end{align}
with some  $C>0$ depending only on   the exponents. 
\end{theorem}
The proof of Theorem \ref{Th9} is very similar to Theorem \ref{Th8}. We then omit it.

\subsection{Strict solutions} 
This subsection  studies  strict solutions  to \eqref{E101}. We prove  existence and uniqueness of strict solutions and show their regularity, provided that the condition {\rm (Ga)} takes place.

\begin{theorem}[Existence of strict solutions] \label{Th6}
Let {\rm (Aa)}, {\rm (Ab)}, {\rm (F1)}, {\rm (F2a)} and {\rm (Ga)}   be satisfied. Let $\xi\in \mathcal D(A^\beta)$ a.s. such that $\mathbb E\|A^\beta\xi\|^2<\infty$. Assume further that there exists $\rho>0$ such that $F_2(x) \in \mathcal D(A^\rho)$ for $x\in \mathcal D(A^\eta)$ and 
\begin{equation}     \label{H20.6}
\mathbb E [\sup_{x\in \mathcal D(A^\eta)}\|A^\rho F_2(x)\|]^2<\infty. 
\end{equation} 
 Then, \eqref{E101} possesses a unique strict solution $X$ on $[0,T_{loc}],$  where $T_{loc}$ is some positive constant in $[0,T]$ depending only on the exponents and $\mathbb E \|F_1\|_{\mathcal F^{\beta,\sigma}(E)}^2$, $\mathbb E \|F_2(0)\|^2, $ $\mathbb E \|A^\beta\xi\|^2,$   $ \mathbb E \|G\|_{\mathcal F^{\beta+\frac{1}{2},\sigma}(\gamma(H;E))}^2.$  
 Furthermore, $X$ has the regularity
\begin{equation*}
X\in  \mathcal C^\gamma([\epsilon,T_{loc}];\mathcal D(A^\eta))\cap \mathcal C([0,T_{loc}];\mathcal D(A^\beta)) \hspace{1cm} \text{a.s.}
\end{equation*}
for any $0<\epsilon<T_{loc}, $ $ 0\leq \gamma<\min\{\beta+\delta-1, \frac{1}{2}-\eta,\frac{1+2\beta}{4}-\eta\}$ with the estimate
\begin{align}
\mathbb E\|AX(t)\|^2    
 \leq  &  C \mathbb E \|\xi\|^2 t^{-2}   + C   \mathbb E \|F_1\|_{\mathcal F^{\beta,\sigma}(E)}^2 t^{2(\beta-1)}       \label{H23.4}  \\
&    +C \mathbb E \|A^\delta G\|_{\mathcal F^{\beta+\frac{1}{2},\sigma}(\gamma(H;E))}^2 t^{2(\beta+\delta-1)}   \notag  \\
&   
+C \mathbb E[\sup_{x\in \mathcal D(A^\eta)}\|A^\rho F_2(x)\|]^2   t^{2\rho},    \hspace{1.5cm} 0< t\leq T_{loc}  \notag
\end{align}
with some $C>0$ depending on the exponents.
\end{theorem}
\begin{proof}
Thanks to Theorem \ref{Th8},  \eqref{E101} has a unique local mild solution $X$ in the function space:
\begin{equation*}
X\in  \mathcal C^\gamma([\epsilon,T_{loc}];\mathcal D(A^\eta))\cap \mathcal C([0,T_{loc}];\mathcal D(A^\beta)) \hspace{1cm} \text{a.s.}
\end{equation*}
for any $0<\epsilon<T_{loc}, $ $ 0\leq \gamma<\min\{\beta+\delta-1, \frac{1}{2}-\eta,\frac{1+2\beta}{4}-\eta\}$. 

First, let us show that $X$ is a local strict solution of \eqref{E101}. 
 We have
\begin{align*}
X(t) = &\Big[ S(t) \xi+\int_0^t  S(t-s)F_1(s)ds + \int_0^t  S(t-s)G(s) dW(s) \Big] \notag\\
&+\int_0^t  S(t-s)F_2(X(s))ds     \notag \\
=&X_1(t)+X_2(t), \hspace{1cm} 0\leq t\leq T_{loc}.      
\end{align*}
Since  all the assumptions of Theorem \ref{Th2} are  satisfied, it is possible to see  that 
\begin{align} 
X_1(t)=&\xi +\int_0^t F_1(s)ds-\int_0^t AX_1(s)ds      \label{H23.5}\\
&+ \int_0^t G(s)dW(s),  \hspace{1cm} 0<t\leq T_{loc},  \notag
\end{align}
and by \eqref{H13.6} 
\begin{align}
\mathbb E\|AX_1(t)\|^2    
 \leq  &  C \mathbb E \|\xi\|^2 t^{-2}   + C   \mathbb E \|F_1\|_{\mathcal F^{\beta,\sigma}(E)}^2 t^{2(\beta-1)}      \label{H23.6}  \\
&  +C \mathbb E \|A^\delta G\|_{\mathcal F^{\beta+\frac{1}{2},\sigma}(\gamma(H;E))}^2 t^{2(\beta+\delta-1)},    \hspace{0.5cm} 0< t\leq T_{loc}  \notag
\end{align}
with some $C>0$ depending on the exponents.

In the meantime,   \eqref{H10} gives that
\begin{align*}
&\int_0^t  \|AS(t-s)F_2(X(s))\|ds\\
&=\int_0^t \| A^{1-\rho} S(t-s)\| \|A^\rho F_2(X(s))\|ds \\
&\leq \iota_{1-\rho} \int_0^t  (t-s)^{\rho-1} \sup_{x\in \mathcal D(A^\eta)}\|A^\rho F_2(x)\|ds \\
&=\frac{\iota_{1-\rho}}{\rho}\sup_{x\in \mathcal D(A^\eta)}\|A^\rho F_2(x)\|  t^\rho<\infty  \hspace{1cm} \text{a.s.,} \quad 0\leq t\leq T_{loc}.
\end{align*}
Thereby, $\int_0^\cdot  AS(\cdot-s)F_2(X(s))ds $ is well-defined on $[0,T_{loc}]$.

Since $A$ is closed, 
$$AX_2=\int_0^t  AS(t-s)F_2(X(s))ds,  \hspace{1cm} 0\leq t\leq T_{loc}$$
and
\begin{equation}  \label{H23.7}
\mathbb E \|AX_2(t)\|^2  \leq \frac{\iota_{1-\rho}^2}{\rho^2}\mathbb E[\sup_{x\in \mathcal D(A^\eta)}\|A^\rho F_2(x)\|]^2   t^{2\rho},  \hspace{1cm} 0\leq t\leq T_{loc}.
\end{equation}
Thus,
\begin{align*}
\frac{dX_2(t)}{dt}&=\frac{d}{dt}  \int_0^t  S(t-s)F_2(X(s))ds\\
&=F_2(X(t))- \int_0^t AS(t-s)F_2(X(s))ds\\
&=F_2(X(t))-AX_2(t),  \hspace{1cm} 0\leq t\leq T_{loc}.
\end{align*}
This implies that
\begin{equation} \label{H24.5}
X_2(t)=\int_0^t F_2(X(s))ds-\int_0^t AX_2(s)ds,  \hspace{1cm} 0\leq t\leq T_{loc}.
\end{equation}
Combining \eqref{H23.5} and \eqref{H24.5} yields that $X=X_1+X_2$ is a local strict solution of \eqref{E101} on $[0,T_{loc}]$.

Let us now prove \eqref{H23.4}. Since
$$\mathbb E\|AX(t)\|^2 \leq 2 \mathbb E\|AX_1(t)\|^2+ 2 \mathbb E\|AX_2(t)\|^2,$$
\eqref{H23.4} follows from \eqref{H23.6} and \eqref{H23.7}.
We complete the proof.
\end{proof}

\subsection{Regular dependence of solutions on initial data}
This subsection  investigates regular dependence of local mild solutions on initial data. For this, we  use the following lemma. A proof of the lemma can be found in \cite{Ton1.5}.

\begin{lemma} \label{Thm3}
Let $b\geq a>0,  \mu>0 $ and $\nu>0$. 
Let $f\colon [0,\infty) \to [0,\infty)$ be a continuous and increasing function, and $\varphi\colon [a,b] \to [0,\infty)$ be a   bounded  function. Assume that  
$$\varphi(t) \leq f(t)+ a^{-\mu} \int_a^t (t-r)^{\nu-1}\varphi(r)dr, \quad\quad a\leq t\leq b.$$
Then, there exists $c>0$ such that 
$$\varphi(t)\leq c f(t),\hspace{1cm} a\leq s<t\leq b.$$
\end{lemma}

Let the assumptions in Theorem \ref{Th8} be satisfied. Denote by  $\mathcal B_1$ and $\mathcal B_2$ two bounded balls: 
\begin{align}
&\mathcal B_1=\{f\in \mathcal F^{\beta,\sigma}((0,T];E) \text{ a.s. such that }  \mathbb E\|f\|_{\mathcal F^{\beta,\sigma}(E)}^2 \leq R_1^2\},   \label{mathcal B1}\\
&\mathcal B_2=\{g\in \mathcal F^{\beta+\frac{1}{2},\sigma}((0,T];\gamma(H;E))\text{ a.s. such that }  \label{mathcal B2}\\
& \hspace{1.5cm} A^\delta g\in \mathcal F^{\beta+\frac{1}{2},\sigma}((0,T];\gamma(H;E)) \text{ a.s. and } \notag \\
& \hspace{1.5cm} \mathbb E\|A^\delta g\|_{\mathcal F^{\beta+\frac{1}{2},\sigma}(\gamma(H;E))}^2 \leq R_2^2\},  \notag
\end{align}
 of the spaces $\mathcal F^{\beta,\sigma}((0,T];E)$ and $\mathcal F^{\beta+\frac{1}{2},\sigma}((0,T];\gamma(H;E))$, respectively, where $R_1$ and $R_2$ are some positive constants. 
Denote by  $B_A$  a set of random variables:  
\begin{equation}  \label{BABall}
B_A=\{\zeta: \zeta\in \mathcal D(A^\beta) \, \text{   a.s. and    }  \,   \mathbb E \|A^\beta \zeta\|^2\leq R_3^2\}, \quad 0<R_3<\infty.
\end{equation}

 According to Theorem \ref{Th8}, for every $F_1\in \mathcal B_1, G\in \mathcal B_2$ and $\xi\in B_A,$ there exists a local mild solution of \eqref{E101} on some interval $[0,T_{loc}]$. Furthermore, by virtue of   Steps 1 and   2  in  the proof of Theorem 
\ref{Th6}, 
\begin{equation} \label{Eq41}
\begin{aligned}
&\text{ there is a time    }  T_{\mathcal B_1, \mathcal B_2, B_A}>0 \text{  such that   } \\
&[0,T_{\mathcal B_1, \mathcal B_2, B_A}]\subset [0,T_{loc}]   \,  \text{   for all   } \,  (F_1,G,\xi)\in \mathcal B_1\times \mathcal B_2\times B_A. 
\end{aligned}
\end{equation}
Indeed, in view of \eqref{Ph11}, \eqref{Ph12} and \eqref{Ph13}, $T_{loc}$ can be chosen to be any time $S$ satisfying the conditions:
\begin{align*}
12\iota_\beta^2 &c_{F_2}^2 \kappa^2  B( 1+2\beta-2\eta, 1-2\eta) S^{2(1+\beta-\eta)}+\frac{12 \iota_\beta^2 \mathbb E\|F_2(0)\|^2}{1-2\eta} S^{2(1-\beta)} \leq \frac{\kappa^2}{2},\\
12\iota_\beta^2 &c_{F_2}^2 \kappa^2  B( 1+2\beta-2\eta, 1-2\beta) S^{2(1+\beta-2\eta)}+\frac{12 \iota_\beta^2 \mathbb E\|F_2(0)\|^2}{1-2\beta} S^{2(1-\beta)}\leq \frac{\kappa^2}{2},
\end{align*}
and
\begin{align*}
&c_{F_2}^2  [\iota_\eta^2 B(1+2\beta-2\eta,1-2\eta)+\iota_\beta^2 B(1+2\beta-2\eta,1-2\beta)]   S^{2(1-\eta)} <1,
\end{align*}
where $\kappa$ is defined by \eqref{Ph4} (see also \eqref{Ph10}).  Consequently,   $T_{loc}$ can be chosen being a constant which  depends continuously on  $\mathbb E \|F_2(0)\|^2$,  $ \mathbb E \|A^\beta\xi\|^2$,  $ \mathbb E\|A^\delta G\|_{\mathcal F^{\beta+\frac{1}{2},\sigma}(\gamma(H;E))}^2$,  and  $\mathbb E\|F_1\|_{\mathcal F^{\beta,\sigma}(E)}^2$. 
  Thus, \eqref{Eq41} follows.

We are now ready to state continuous dependence of local mild solutions on $(F_1,G,\xi)$ in the sense specified in the following theorem.

\begin{theorem}\label{Th10}
Let {\rm (Aa)}, {\rm (Ab)}, {\rm (F1)}, {\rm (F2a)} and {\rm (Ga)}   be satisfied.
Let $X$ and $\bar X$ be the local mild solutions of  \eqref{E101} on $[0,T_{loc}]$ and $[0,\bar T_{loc}]$ for the data $(F_1,G,\xi)$ and $(\bar F_1,\bar G,\bar \xi)$ in $\mathcal B_1\times \mathcal B_2\times B_A$, respectively. Then, there exists  $C_{\mathcal B_1, \mathcal B_2, B_A}>0$ depending only on $\mathcal B_1, \mathcal B_2$ and $ B_A$ such that 
\begin{align}
&t^{2\eta}\mathbb E  \|A^\eta[X(t)-\bar X(t)]\|^2+t^{2\eta}\mathbb E\|A^{\beta} [X(t)-\bar X(t)]\|^2    \label{Eq22}\\
&+ \mathbb E\|X(t)-\bar X(t)\|^2   \notag\\
\leq &C_{\mathcal B_1, \mathcal B_2, B_A}[\mathbb E \|\xi-\bar \xi\|^2+ t^{2\beta}   \mathbb E\|F_1-\bar F_1\|_{\mathcal F^{\beta,\sigma}(E)}^2     \notag\\
&+ \mathbb E\|A^\delta (G-\bar G)\|_{\mathcal F^{\beta+\frac{1}{2},\sigma}(\gamma(H;E))}^2], \hspace{1cm} 0<t\leq T_{\mathcal B_1, \mathcal B_2, B_A}. \notag
\end{align}
\end{theorem}

\begin{proof}
This theorem is proved by using analogous arguments as in the proof of Theorem \ref{Th6}. Throughout the proof, we use a  notation $C_{\mathcal B_1, \mathcal B_2,B_A}$ to denote positive constants which are defined by the exponents, and $\mathcal B_1, \mathcal B_2$ and $  B_A$. So, it may change from occurrence to occurrence.

First, we  give an estimate for 
$$t^{2\eta}\mathbb E [ \|A^\eta[X(t)-\bar X(t)]\|^2+\|A^{\beta} [X(t)-\bar X(t)]\|^2].$$
For $0 \leq \theta <\frac{1}{2}$, by using \eqref{H12}, \eqref{H10} and {\rm (F2a)}, 
\begin{align*}
&t^\theta \|A^\theta[X(t)-\bar X(t)]\|\\
= & \Big\|t^\theta A^\theta S(t)(\xi-\bar \xi)  +\int_0^t t^\theta A^\theta S(t-s) [F_1(s)-\bar F_1(s)]ds  \notag\\
&+\int_0^t t^\theta A^\theta S(t-s)[F_2(X(s))-F_2(\bar X(s))]ds   \notag\\
&+\int_0^t t^\theta A^\theta S(t-s) [G(s)-\bar G(s)]dW(s)\Big\| \notag\\
\leq & \iota_\theta \|\xi-\bar \xi\|   +\iota_\theta \|F_1-\bar F_1\|_{\mathcal F^{\beta,\sigma}(E)} \int_0^t t^\theta (t-s)^{-\theta}s^{\beta-1}ds  \notag\\
&+\iota_\theta c_{F_2}\int_0^t  t^\theta (t-s)^{-\theta}  \|A^\eta[X(s)-\bar X(s)]\| ds\notag\\
&+\Big\|\int_0^t t^\theta A^\theta S(t-s) [G(s)-\bar G(s)]dW(s)\Big\| \notag\\
= & \iota_\theta \|\xi-\bar \xi\|+\iota_\theta B(\beta,1-\theta) \|F_1-\bar F_1\|_{\mathcal F^{\beta,\sigma}(E)}   t^\beta\notag\\
& +\iota_\theta c_{F_2}\int_0^t  t^\theta (t-s)^{-\theta}  \|A^\eta[X(s)-\bar X(s)]\| ds\notag\\
&+\Big\|\int_0^t t^\theta A^{\theta-\delta} S(t-s) [A^\delta G(s)-A^\delta \bar G(s)]dW(s)\Big\|, \hspace{0.7cm} 0<t\leq T_{\mathcal B_1, \mathcal B_2, B_A}. \notag
\end{align*}
Thus, by \eqref{H11},
\begin{align}
&\mathbb E\|t^\theta A^\theta[X(t)-\bar X(t)]\|^2\notag\\
\leq & 4\iota_\theta^2 \mathbb E \|\xi-\bar \xi\|^2+4\iota_\theta^2  B(\beta,1-\theta)^2 \mathbb E\|F_1-\bar F_1\|_{\mathcal F^{\beta,\sigma}(E)}^2    t^{2\beta}\notag\\
& +4\iota_\theta^2 c_{F_2}^2 t^{2\theta} \mathbb E \Big[ \int_0^t  (t-s)^{-\theta}  \|A^\eta[X(s)-\bar X(s)]\| ds\Big]^2\notag\\
&+4\mathbb E \Big\|\int_0^t t^\theta A^{\theta-\delta} S(t-s) [A^\delta G(s)-A^\delta \bar G(s)]dW(s)\Big\|^2 \notag\\
\leq & 4\iota_\theta^2 \mathbb E \|\xi-\bar \xi\|^2+4\iota_\theta^2  B(\beta,1-\theta)^2 \mathbb E \|F_1-\bar F_1\|_{\mathcal F^{\beta,\sigma}(E)}^2    t^{2\beta} \notag\\
&+4\iota_\theta^2 c_{F_2}^2 t^{2\theta+1} \int_0^t   (t-s)^{-2\theta}  \mathbb E \|A^\eta[X(s)-\bar X(s)]\|^2 ds\notag\\
&+4c(E) \iota_0^2 \|A^{\theta-\delta}\|^2 \mathbb E \|A^\delta (G-\bar G)\|_{\mathcal F^{\beta+\frac{1}{2},\sigma}(\gamma(H;E))}^2   \notag   \\
& \times \int_0^t t^{2\theta} e^{-2\nu (t-s)} s^{2\beta -1}ds \notag\\
\leq  & 4\iota_\theta^2 \mathbb E \|\xi-\bar \xi\|^2+4\iota_\theta^2    B(\beta,1-\theta)^2 t^{2\beta} \mathbb E\|F_1-\bar F_1\|_{\mathcal F^{\beta,\sigma}(E)}^2      \label{Eq19}\\
&+4c(E) \iota_0^2 \|A^{\theta-\delta}\|^2   \sup_{0\leq t<\infty} \int_0^t  e^{-2\nu (t-s)} s^{2\beta -1}ds  \notag\\
& \times  t^{2\theta}   \mathbb E \|A^\delta (G-\bar G)\|_{\mathcal F^{\beta+\frac{1}{2},\sigma}(\gamma(H;E))}^2    +4\iota_\theta^2 c_{F_2}^2 t^{2\theta+1}  \int_0^t   (t-s)^{-2\theta} \notag\\
& \times  \mathbb E \|A^\eta[X(s)-\bar X(s)]\|^2 ds, \hspace{1cm} 0<t \leq T_{\mathcal B_1, \mathcal B_2, B_A}.     \notag
\end{align}

Applying these estimates  with $\theta=\beta$ and $\theta=\eta$, it follows that 
\begin{align*}
&\mathbb E\| A^\beta[X(t)-\bar X(t)]\|^2\notag\\
\leq  & 4\iota_\beta^2 \mathbb E \|\xi-\bar \xi\|^2t^{-2\beta}+4\iota_\beta^2    B(\beta,1-\beta)^2  \mathbb E\|F_1-\bar F_1\|_{\mathcal F^{\beta,\sigma}(E)}^2\notag\\
&+4c(E) \iota_0^2 \|A^{\theta-\delta}\|^2   \sup_{0\leq t<\infty} \int_0^t  e^{-2\nu (t-s)} s^{2\beta -1}ds   \notag\\
& \times \mathbb E\|A^\delta(G-\bar G)\|_{\mathcal F^{\beta+\frac{1}{2},\sigma}(\gamma(H;E))}^2      +4\iota_\beta^2 c_{F_2}^2 t  \int_0^t   (t-s)^{-2\beta}  \notag\\
& \times  \mathbb E \|A^\eta[X(s)-\bar X(s)]\|^2 ds, \hspace{1cm} 0<t \leq T_{\mathcal B_1, \mathcal B_2, B_A},     \notag
\end{align*}
and
\begin{align*}
&t^{2\eta} \mathbb E\|A^\eta[X(t)-\bar X(t)]\|^2\notag\\
\leq  & 4\iota_\eta^2 \mathbb E \|\xi-\bar \xi\|^2+4\iota_\eta^2    B(\beta,1-\eta)^2 t^{2\beta} \mathbb E\|F_1-\bar F_1\|_{\mathcal F^{\beta,\sigma}(E)}^2\notag\\
&+4c(E) \iota_0^2 \|A^{\theta-\delta}\|^2   \sup_{0\leq t<\infty} \int_0^t  e^{-2\nu (t-s)} s^{2\beta -1}ds t^{2\eta} \notag\\
& \times  \mathbb E\|A^\delta(G-\bar G)\|_{\mathcal F^{\beta+\frac{1}{2},\sigma}(\gamma(H;E))}^2\notag   +4\iota_\eta^2 c_{F_2}^2 t^{2\eta+1}  \int_0^t   (t-s)^{-2\eta}  \notag\\
& \times\mathbb E \|A^\eta[X(s)-\bar X(s)]\|^2 ds, \hspace{1cm} 0<t \leq T_{\mathcal B_1, \mathcal B_2, B_A}.     \notag
\end{align*}

By putting 
$$q(t)=t^{2\eta}\mathbb E [ \|A^\eta[X(t)-\bar X(t)]\|^2+\|A^{\beta} [X(t)-\bar X(t)]\|^2], $$
we  obtain that
\begin{align}
q(t)\leq  & 4[\iota_\beta^2 t^{2(\eta-\beta)} +\iota_\eta^2 ] \mathbb E \|\xi-\bar \xi\|^2\notag\\
&+4[\iota_\beta^2    B(\beta,1-\beta)^2t^{2\eta}+\iota_\eta^2    B(\beta,1-\eta)^2 t^{2\beta} ]  \mathbb E \|F_1-\bar F_1\|_{\mathcal F^{\beta,\sigma}(E)}^2\notag\\
&+4c(E) \iota_0^2 \|A^{\theta-\delta}\|^2  (\iota_\beta^2+\iota_\eta^2)  \sup_{0\leq t<\infty} \int_0^t  e^{-2\nu (t-s)} s^{2\beta -1}ds t^{2\eta}   \notag\\
&\times \mathbb E\|A^\delta(G-\bar G)\|_{\mathcal F^{\beta+\frac{1}{2},\sigma}(\gamma(H;E))}^2 \notag\\
&+4 c_{F_2}^2 t^{2\eta+1}\int_0^t [ \iota_\beta^2 (t-s)^{-2\beta}+\iota_\eta^2  (t-s)^{-2\eta}]  \notag\\
&\times  \mathbb E\|A^\eta[X(s)-\bar X(s)]\|^2ds\notag\\
\leq  & 4[\iota_\beta^2 t^{2(\eta-\beta)} +\iota_\eta^2 ] \mathbb E \|\xi-\bar \xi\|^2\notag\\
&+4[\iota_\beta^2    B(\beta,1-\beta)^2t^{2\eta}+\iota_\eta^2    B(\beta,1-\eta)^2 t^{2\beta} ]  \mathbb E \|F_1-\bar F_1\|_{\mathcal F^{\beta,\sigma}(E)}^2\notag\\
&+4c(E) \iota_0^2 \|A^{\theta-\delta}\|^2  (\iota_\beta^2+\iota_\eta^2)  \sup_{0\leq t<\infty} \int_0^t  e^{-2\nu (t-s)} s^{2\beta -1}ds t^{2\eta}   \notag\\
&\times \mathbb E\|A^\delta(G-\bar G)\|_{\mathcal F^{\beta+\frac{1}{2},\sigma}(\gamma(H;E))}^2 \notag\\
&+4 c_{F_2}^2 t^{2\eta+1}\int_0^t [ \iota_\beta^2 (t-s)^{-2\beta}+\iota_\eta^2  (t-s)^{-2\eta}]  s^{-2\eta}q(s)  ds\notag\\
\leq  & C_{\mathcal B_1, \mathcal B_2, B_A}[\mathbb E \|\xi-\bar \xi\|^2+ t^{2\beta}   \mathbb E \|F_1-\bar F_1\|_{\mathcal F^{\beta,\sigma}(E)}^2       + t^{2\eta} \label{Eq47}\\
& \times \mathbb E \|A^\delta(G-\bar G)\|_{\mathcal F^{\beta+\frac{1}{2},\sigma}(\gamma(H;E))}^2] +4 c_{F_2}^2 t^{2\eta+1}\int_0^t [ \iota_\beta^2 (t-s)^{-2\beta}\notag\\
&+\iota_\eta^2  (t-s)^{-2\eta}]  s^{-2\eta}q(s) ds, \hspace{1cm} 0<t\leq T_{\mathcal B_1, \mathcal B_2, B_A}.    \notag 
\end{align}

The integral inequality  \eqref{Eq47} can be solved as follows. 
Let $\epsilon>0$  denote a small parameter. For $0\leq t\leq \epsilon,$ 
\begin{align*}
q(t) \leq   & C_{\mathcal B_1, \mathcal B_2, B_A}[\mathbb E \|\xi-\bar \xi\|^2+ t^{2\beta}   \mathbb E \|F_1-\bar F_1\|_{\mathcal F^{\beta,\sigma}(E)}^2       \notag\\
&\hspace{1.5cm} + t^{2\eta} \mathbb E \|A^\delta(G-\bar G)\|_{\mathcal F^{\beta+\frac{1}{2},\sigma}(\gamma(H;E))}^2] \notag\\
&+4 c_{F_2}^2 t^{2\eta+1}\int_0^t [ \iota_\beta^2 (t-s)^{-2\beta}+\iota_\eta^2  (t-s)^{-2\eta}]  s^{-2\eta} ds   \sup_{s\in [0,\epsilon]} q(s)\\
=& C_{\mathcal B_1, \mathcal B_2, B_A}[\mathbb E \|\xi-\bar \xi\|^2+ t^{2\beta}   \mathbb E \|F_1-\bar F_1\|_{\mathcal F^{\beta,\sigma}(E)}^2       \notag\\
&\hspace{1.5cm} + t^{2\eta} \mathbb E \|A^\delta(G-\bar G)\|_{\mathcal F^{\beta+\frac{1}{2},\sigma}(\gamma(H;E))}^2] \notag\\
&+4 c_{F_2}^2  [ \iota_\beta^2 B(1-2\eta,1-2\beta)t^{2(1-\beta)}+\iota_\eta^2  B(1-2\eta,1-2\eta) t^{2(1-\eta)}]   \notag\\
& \times  \sup_{s\in [0,\epsilon]} q(s)\\
\leq &C_{\mathcal B_1, \mathcal B_2, B_A}[\mathbb E \|\xi-\bar \xi\|^2+ \epsilon^{2\beta}   \mathbb E \|F_1-\bar F_1\|_{\mathcal F^{\beta,\sigma}(E)}^2       \notag\\
&\hspace{1.5cm} + \epsilon^{2\eta} \mathbb E \|A^\delta(G-\bar G)\|_{\mathcal F^{\beta+\frac{1}{2},\sigma}(\gamma(H;E))}^2] \notag\\
&+4 c_{F_2}^2  [ \iota_\beta^2 B(1-2\eta,1-2\beta)\epsilon^{2(1-\beta)}+\iota_\eta^2  B(1-2\eta,1-2\eta) \epsilon^{2(1-\eta)}]   \notag\\
& \times   \sup_{s\in [0,\epsilon]} q(s).
\end{align*}

Taking supremum on $[0,\epsilon]$, it yields that 
\begin{align*}
&\{1-4 c_{F_2}^2  [ \iota_\beta^2  B(1-2\eta,1-2\beta)\epsilon^{2(1-\beta)}+\iota_\eta^2   B(1-2\eta,1-2\eta) \epsilon^{2(1-\eta)}] \}  \\
&\times  \sup_{s\in [0,\epsilon]} q(s)\\
&\leq C_{\mathcal B_1, \mathcal B_2, B_A}[\mathbb E \|\xi-\bar \xi\|^2+ \epsilon^{2\beta}   \mathbb E \|F_1-\bar F_1\|_{\mathcal F^{\beta,\sigma}(E)}^2       \notag\\
&\hspace{1.5cm} + \epsilon^{2\eta} \mathbb E \|A^\delta(G-\bar G)\|_{\mathcal F^{\beta+\frac{1}{2},\sigma}(\gamma(H;E))}^2].
\end{align*}
If $\epsilon$ is taken sufficiently small so that 
\begin{equation} \label{P41}
c_{F_2}^2  [ \iota_\beta^2  B(1-2\eta,1-2\beta)\epsilon^{2(1-\beta)}+\iota_\eta^2   B(1-2\eta,1-2\eta) \epsilon^{2(1-\eta)}] \leq \frac{1}{8},
\end{equation}
 then
\begin{align}  
\sup_{s\in [0,\epsilon]} q(s)  \leq &C_{\mathcal B_1, \mathcal B_2, B_A}[\mathbb E \|\xi-\bar \xi\|^2+ \epsilon^{2\beta}   \mathbb E \|F_1-\bar F_1\|_{\mathcal F^{\beta,\sigma}(E)}^2       \label{P42}\\
&\hspace{1.5cm} + \epsilon^{2\eta} \mathbb E \|A^\delta(G-\bar G)\|_{\mathcal F^{\beta+\frac{1}{2},\sigma}(\gamma(H;E))}^2]. 
 \notag
\end{align}
In particular,
\begin{align}  
 q(\epsilon) \leq &C_{\mathcal B_1, \mathcal B_2, B_A}[\mathbb E \|\xi-\bar \xi\|^2+ \epsilon^{2\beta}   \mathbb E \|F_1-\bar F_1\|_{\mathcal F^{\beta,\sigma}(E)}^2       \label{P43}\\
&\hspace{1.5cm} + \epsilon^{2\eta} \mathbb E \|A^\delta(G-\bar G)\|_{\mathcal F^{\beta+\frac{1}{2},\sigma}(\gamma(H;E))}^2]. 
 \notag
\end{align}

Meanwhile, for $\epsilon<t \leq T_{\mathcal B_1, \mathcal B_2, B_A},$ 
\begin{align*}
q(t)&\\
\leq   & C_{\mathcal B_1, \mathcal B_2, B_A}[\mathbb E \|\xi-\bar \xi\|^2+ t^{2\beta}   \mathbb E \|F_1-\bar F_1\|_{\mathcal F^{\beta,\sigma}(E)}^2       \notag\\
&\hspace{1.5cm} + t^{2\eta} \mathbb E \|A^\delta(G-\bar G)\|_{\mathcal F^{\beta+\frac{1}{2},\sigma}(\gamma(H;E))}^2] \notag
\\
&+4 c_{F_2}^2 t^{2\eta+1}\int_0^\epsilon [ \iota_\beta^2 (t-s)^{-2\beta}+\iota_\eta^2  (t-s)^{-2\eta}]  s^{-2\eta} ds \sup_{s\in [0,\epsilon]} q(s) \\
&+4 c_{F_2}^2 t^{2\eta+1}\int_\epsilon^t [ \iota_\beta^2 (t-s)^{-2\beta}+\iota_\eta^2  (t-s)^{-2\eta}]  s^{-2\eta}q(s) ds \\
\leq   & C_{\mathcal B_1, \mathcal B_2, B_A}[\mathbb E \|\xi-\bar \xi\|^2+ t^{2\beta}   \mathbb E \|F_1-\bar F_1\|_{\mathcal F^{\beta,\sigma}(E)}^2       \notag\\
&\hspace{1.5cm} + t^{2\eta} \mathbb E \|A^\delta(G-\bar G)\|_{\mathcal F^{\beta+\frac{1}{2},\sigma}(\gamma(H;E))}^2] \notag
\\
&+4 c_{F_2}^2  [ \iota_\beta^2 B(1-2\eta,1-2\beta)t^{2(1-\beta)}+\iota_\eta^2  B(1-2\eta,1-2\eta) t^{2(1-\eta)}]    \\
&\times   \sup_{s\in [0,\epsilon]} q(s) \\
&+4 c_{F_2}^2  t^{2\eta+1}\int_\epsilon^t [ \iota_\beta^2 (t-s)^{2(\eta-\beta)}+\iota_\eta^2  ] (t-s)^{-2\eta} \epsilon^{-2\eta}q(s) ds \\
\leq   & C_{\mathcal B_1, \mathcal B_2, B_A}[\mathbb E \|\xi-\bar \xi\|^2+ t^{2\beta}   \mathbb E \|F_1-\bar F_1\|_{\mathcal F^{\beta,\sigma}(E)}^2       \notag\\
&\hspace{1.5cm} + t^{2\eta} \mathbb E \|A^\delta(G-\bar G)\|_{\mathcal F^{\beta+\frac{1}{2},\sigma}(\gamma(H;E))}^2] \notag
\\
&+4 c_{F_2}^2  [ \iota_\beta^2 B(1-2\eta,1-2\beta) T^{2(1-\beta)}+\iota_\eta^2  B(1-2\eta,1-2\eta) T^{2(1-\eta)}]   \\
&\times  \sup_{s\in [0,\epsilon]} q(s) +4 c_{F_2}^2 \epsilon^{-2\eta} T^{2\eta+1} [ \iota_\beta^2 T^{2(\eta-\beta)}+\iota_\eta^2  ] \int_\epsilon^t  (t-s)^{-2\eta} q(s) ds.
\end{align*}
Lemma \ref{Thm3} then provides that 
\begin{align*}
q(t) 
\leq & C_{\mathcal B_1, \mathcal B_2, B_A}[\mathbb E \|\xi-\bar \xi\|^2+ t^{2\beta}   \mathbb E \|F_1-\bar F_1\|_{\mathcal F^{\beta,\sigma}(E)}^2       \notag\\
&\hspace{1.5cm} + t^{2\eta} \mathbb E \|A^\delta(G-\bar G)\|_{\mathcal F^{\beta+\frac{1}{2},\sigma}(\gamma(H;E))}^2] \notag
\\
&+4 c_{F_2}^2  [ \iota_\beta^2 B(1-2\eta,1-2\beta)T^{2(1-\beta)}+\iota_\eta^2  B(1-2\eta,1-2\eta) T^{2(1-\eta)}]   \\
&\times  \sup_{s\in [0,\epsilon]} q(s), \hspace{1cm} \epsilon<t\leq T_{\mathcal B_1, \mathcal B_2, B_A}.
\end{align*}
Thanks to \eqref{P42},  
\begin{align*}
q(t) 
\leq & C_{\mathcal B_1, \mathcal B_2, B_A}[\mathbb E \|\xi-\bar \xi\|^2+ t^{2\beta}   \mathbb E \|F_1-\bar F_1\|_{\mathcal F^{\beta,\sigma}(E)}^2       \notag\\
&\hspace{1.5cm} + t^{2\eta} \mathbb E \|A^\delta(G-\bar G)\|_{\mathcal F^{\beta+\frac{1}{2},\sigma}(\gamma(H;E))}^2] \notag
\\
&+C_{\mathcal B_1, \mathcal B_2, B_A}[\mathbb E \|\xi-\bar \xi\|^2+ \epsilon^{2\beta}   \mathbb E \|F_1-\bar F_1\|_{\mathcal F^{\beta,\sigma}(E)}^2       \notag\\
&\hspace{1.5cm} + \epsilon^{2\eta} \mathbb E \|A^\delta(G-\bar G)\|_{\mathcal F^{\beta+\frac{1}{2},\sigma}(\gamma(H;E))}^2], \hspace{0.5cm} \epsilon<t\leq T_{\mathcal B_1, \mathcal B_2, B_A}.
\end{align*}
 Hence, 
\begin{align}
q(t) \leq & C_{\mathcal B_1, \mathcal B_2, B_A}[\mathbb E \|\xi-\bar \xi\|^2+ t^{2\beta}   \mathbb E \|F_1-\bar F_1\|_{\mathcal F^{\beta,\sigma}(E)}^2       \label{P44}\\
& + t^{2\eta} \mathbb E \|A^\delta(G-\bar G)\|_{\mathcal F^{\beta+\frac{1}{2},\sigma}(\gamma(H;E))}^2], \hspace{0.5cm} \epsilon<t\leq T_{\mathcal B_1, \mathcal B_2, B_A}. \notag
\end{align}

Thus,  \eqref{P43} and \eqref{P44} highlight that 
\begin{align}
t^{2\eta}&\mathbb E [ \|A^\eta[X(s)-\bar X(s)]\|^2+\|A^{\beta} [X(s)-\bar X(s)]\|^2]\notag\\
=&q(t)\notag\\
\leq &C_{\mathcal B_1, \mathcal B_2, B_A}[\mathbb E \|\xi-\bar \xi\|^2+ t^{2\beta}   \mathbb E \|F_1-\bar F_1\|_{\mathcal F^{\beta,\sigma}(E)}^2       \label{Eq20}\\
& + t^{2\eta} \mathbb E \|A^\delta(G-\bar G)\|_{\mathcal F^{\beta+\frac{1}{2},\sigma}(\gamma(H;E))}^2],
   \hspace{1cm}   0<t\leq T_{\mathcal B_1, \mathcal B_2, B_A}.      \notag
\end{align}

Second, let us give an estimate for $\mathbb E\|X(t)-\bar X(t)\|^2$. Take $\theta=0$ in the equality  \eqref{Eq19}: 
\begin{align*}
\mathbb E&\|X(t)-\bar X(t)\|^2 \\
\leq  &  4\iota_0 \mathbb E \|\xi-\bar \xi\|^2+ 4\iota_0   B(\beta,1)^2 t^{2\beta} \mathbb E\|F_1-\bar F_1\|_{\mathcal F^{\beta,\sigma}(E)}^2\notag\\
&+4c(E) \iota_0^2 \|A^{-\delta}\|^2   \sup_{0\leq t<\infty} \int_0^t  e^{-2\nu (t-s)} s^{2\beta -1}ds  \notag\\
& \times \mathbb E\|A^\delta(G-\bar G)\|_{\mathcal F^{\beta+\frac{1}{2},\sigma}(\gamma(H;E))}^2\\
&
+4\iota_0^2 c_{F_2}^2 t  \int_0^t     \mathbb E \|A^\eta[X(s)-\bar X(s)]\|^2 ds\\
\leq  &  4\iota_0 \mathbb E \|\xi-\bar \xi\|^2+ 4\iota_0   B(\beta,1)^2 t^{2\beta} \mathbb E\|F_1-\bar F_1\|_{\mathcal F^{\beta,\sigma}(E)}^2\notag\\
&+4c(E) \iota_0^2 \|A^{-\delta}\|^2   \sup_{0\leq t<\infty} \int_0^t  e^{-2\nu (t-s)} s^{2\beta -1}ds \notag\\
& \times \mathbb E\|A^\delta(G-\bar G)\|_{\mathcal F^{\beta+\frac{1}{2},\sigma}(\gamma(H;E))}^2  \notag\\
&+4\iota_0^2  c_{F_2}^2 t\int_0^t   s^{-2\eta}q(s)ds.\notag
\end{align*}
By \eqref{Eq20}, 
 \begin{align*}
 t \int_0^t   s^{-2\eta}q(s)ds
 \leq & C_{\mathcal B_1, \mathcal B_2, B_A} t\int_0^t   s^{-2\eta}
[\mathbb E \|\xi-\bar \xi\|^2+ s^{2\beta}   \mathbb E \|F_1-\bar F_1\|_{\mathcal F^{\beta,\sigma}(E)}^2 \\
 &+ s^{2\eta} \mathbb E \|A^\delta(G-\bar G)\|_{\mathcal F^{\beta+\frac{1}{2},\sigma}(\gamma(H;E))}^2]ds\notag\\
  \leq &C_{\mathcal B_1, \mathcal B_2, B_A} t^{2(1-\eta)}\mathbb E \|\xi-\bar \xi\|^2\notag\\
  &+C_{\mathcal B_1, \mathcal B_2, B_A} t^{2(1+\beta-\eta)} \mathbb E\|F_1-\bar F_1\|_{\mathcal F^{\beta,\sigma}(E)}^2   \notag\\
&+ C_{\mathcal B_1, \mathcal B_2, B_A} t^2  \mathbb E\|A^\delta(G-\bar G)\|_{\mathcal F^{\beta+\frac{1}{2},\sigma}(\gamma(H;E))}^2.
  \end{align*}
Therefore, 
 \begin{align}
 & \mathbb E\|X(t)-\bar X(t)\|^2       \label{Eq21} \\
& \leq C_{\mathcal B_1, \mathcal B_2, B_A}[\mathbb E \|\xi-\bar \xi\|^2+ t^{2\beta}   \mathbb E \|F_1-\bar F_1\|_{\mathcal F^{\beta,\sigma}(E)}^2       \notag\\
 &\hspace{0.5cm} +  \mathbb E\|A^\delta(G-\bar G)\|_{\mathcal F^{\beta+\frac{1}{2},\sigma}(\gamma(H;E))}^2],     \hspace{1cm} 0<t\leq  T_{\mathcal B_1, \mathcal B_2, B_A}.     \notag
\end{align}
Thus,  \eqref{Eq20} and \eqref{Eq21} imply  \eqref{Eq22}.
 It completes the proof.
\end{proof}

\section{Applications}  \label{section6}
This section  gives some examples to illustrate our results. We first present some stochastic PDEs and then formulate them as the Cauchy problems for abstract stochastic parabolic evolution equations of 
the form \eqref{E2}.  We only concentrate  on  strict solutions to the equations.

\subsection{Example 1 (A neurophysiology model)} 
  Walsh \cite{Walsh} proposed  a model arising in neurophysiology as follows.  Nerve cells operating by a mixture of chemical, biological and electrical properties are regarded as long thin cylinders, which act much like electrical cables. 

Denote by $V(t,x)$ the electrical potential at time $t$ and  position $x$. If  such a cylinder is identified with the interval $[0,L],$ then $V$ satisfies a nonlinear equations coupled with a system of ordinary differential equations, called the Hodgkin-Huxley equations.  In some certain ranges of the values of the potential, the equations are approximated by the cable equation:
 \begin{equation} \label{P31}
 \begin{cases}
\frac{\partial V}{\partial t}=\frac{\partial^2 V}{\partial x^2} -V, \\
 V\colon [0,\infty)\times [0,L] \to \mathbb R,  \\
 t\geq 0, x\in [0,L].
 \end{cases}
 \end{equation}
 
 Consider \eqref{P31} with an additional external force, called impulses of current. Neurons receive these impulses of current via synapses on their surface. Denote by $F(t)$ the current arriving at time $t$, then  the electrical potential satisfies the equation: 
  \begin{equation*} 
\begin{cases}
\frac{\partial V}{\partial t}=\frac{\partial^2 V}{\partial x^2} -V +F(t), \\
 V\colon [0,\infty)\times [0,L] \to \mathbb R,  \\
 t\geq 0, x\in [0,L].
 \end{cases}
 \end{equation*}
 
Suppose now that the current $F$ is perturbed by infinitely many  white noises. Suppose further that the intensity of each white noise depends separately on both position $x$ and time $t$. In other words, $F$ is perturbed by the sum $\Sigma_{j=1}^\infty e_j(x) g_j(t) B_j(t)$ of noises, where $\{e_j\}_{j=1}^\infty$ and $\{g_j\}_{j=1}^\infty$ are two sequences of functions, and $\{B_j\}_{j=1}^\infty$
is a family of independent real-value standard Wiener processes on a complete filtered  probability space  $(\Omega, \mathcal F,\mathcal F_t,\mathbb P).$

In order to handle the resulting equation, we assume that $\{e_j\}_{j=1}^\infty$ is an orthonormal and complete basis of some Hilbert space, say $H_0$; and $g_j(t) = g_0(t)$ for all $j$. It is known that (e.g., \cite{prato}) the series $\Sigma_{j=1}^\infty e_j(x)  B_j(t)$ converges to a cylindrical Wiener process on a separable Hilbert space $H\supseteq H_0$, where the embedding of $H_0$ into $H$ is a Hilbert-Schmidt operator. Denote by $W$ the  cylindrical Wiener process.
 This leads  to a stochastic partial differential equation:
  \begin{equation} \label{P32}
  \begin{cases}
\begin{aligned}
&\frac{\partial V}{\partial t}(t,x)=\frac{\partial^2 V}{\partial x^2}(t,x) -V(t,x) +F(t) &+  G(t)  \dot W(t),\\
& & \text{ in } (0,\infty)\times [0,L],\\
& \frac{\partial V}{\partial x}(0,t)= \frac{\partial V}{\partial x}(L,t)=0, &\text{ on } (0,\infty),\\
&V(0,x)=V_0(x), &\text{ on } (0,L).
\end{aligned}
 \end{cases}
 \end{equation}

The first author studied mild solutions to  \eqref{P32} in $L^2([0,L])$ (see \cite{Ton1}). Let us now handle \eqref{P32}  in the Banach space $(E,\|\cdot\|)=(L^p([0,L]),\|\cdot\|_{L^p})$ for some $2\leq p<\infty$. Assume further that  
\begin{description}
    \item [{\rm (Ex1.1)}] For every $T>0$,  $F$ is  measurable  from $[0,T]$ to $ L^p([0,L])$  and satisfies
  $$  
          F\in \mathcal F^{\beta, \sigma}((0,T];L^p([0,L]))  \hspace{0.7cm} \text{ for some }  0<\sigma<\beta< \frac{1}{2}.$$
  \item [{\rm (Ex1.2)}] For every $T>0$,  $G$ is  square integrable, measurable  from  $[0,T]$ to  the space $\gamma(H;L^p([0,L])) $ and satisfies
  $$  
          A^\delta G\in \mathcal F^{\beta+\frac{1}{2}, \sigma} ((0,T];\gamma(H;L^p([0,L])))  \hspace{0.7cm} \text{ for some }  1-\beta<\delta\leq 1,
 $$
 where $A$ is the realization of the operator $-\frac{\partial^2}{\partial x^2}+I$
 in $E$ under the Neumann boundary conditions on the boundary $\{0,L\}$. 
  \item [{\rm (Ex1.3)}] $V_0$ is an $\mathcal  F_0$-measurable  random variable.
\end{description}

It is known that  $A$ is a sectorial operator (see \cite{yagi}).  Furthermore, we have the following lemma.
\begin{lemma}\label{lemma2}
The operator $A$ is a positive definite self-adjoint operator on $H$ with  domain 
$$\mathcal D(A)=H_N^2([0,L]):=\{u\in H^2([0,L]) \text{ such that } \frac{\partial u}{\partial x}(0)= \frac{\partial u}{\partial x}(L)=0\}.$$ 
 In addition, the domains of fractional powers $A^\theta, 0\leq \theta\leq 1,$ are characterized by 
\begin{equation*}
\mathcal D(A^\theta)=
\begin{cases}
H^{2\theta}([0,L]) \hspace{1cm} \text{ for } 0\leq \theta<\frac{3}{4},\\
H_N^{2\theta}([0,L])  \hspace{1cm} \text{ for } \frac{3}{4}< \theta\leq 1.
\end{cases}
\end{equation*}
\end{lemma}

Using $A$, we formulate  \eqref{P32} as the Cauchy problem for an abstract stochastic linear equation
\begin{equation} \label{P33}
\begin{cases}
dV+AVdt=F(t)dt+G(t)dW(t), \hspace{1cm} 0<t<\infty,\\
V(0)=V_0
\end{cases}
\end{equation}
in $E$. 
 Theorems \ref{Th2} and  \ref{Th3} are then available to \eqref{P33}.

\begin{theorem} 
Let {\rm (Ex1.1)}-{\rm (Ex1.3)} be satisfied.
  Then, there exists a unique  strict solution $V$  to \eqref{P32}. Furthermore, $V$  possesses  the regularity
 $$AV\in \mathcal C((0,T];L^p([0,L])) \hspace{1cm} \text{a.s.}$$
with the estimate
\begin{align*}
\mathbb E & \|V(t)\|^2  + t^2 \mathbb E \|AV(t)\|^2 \\
 \leq &  C[\mathbb E \|V_0\|^2 +      \|F_1\|_{\mathcal F^{\beta,\sigma}(L^p([0,L]))}^2   t^{2\beta}  +\|A^\delta G\|_{\mathcal F^{\beta+\frac{1}{2},\sigma}(\gamma(H;L^p([0,L])))}^2   \notag\\
& \quad \times  \{      t^{2\beta}  +t^{2(\beta+\delta)}   \}    ], \hspace{1cm} 0\leq t\leq T     \notag  
\end{align*}
for any $T>0$. 
In addition, when
 $V_0\in \mathcal D(A^\beta) $ a.s. and $\mathbb E\|A^\beta V_0\|<\infty,$ 
 the strict solution $V$  has the maximal space-time regularities
$$
 V \in \mathcal C([0,\infty);\mathcal D(A^\beta)) \cap \mathcal C^{\gamma_1}([0,\infty);L^p([0,L]))
\hspace{1cm} \text{a.s.,}
$$
$$AV \in \mathcal C^{\gamma_2}([\epsilon,\infty);L^p([0,L]))    \hspace{1cm} \text{a.s.}$$
for any  $ 0<\gamma_1<\beta$, $ 0<\gamma_2<\beta+\delta-1, 0<\gamma_2\leq \sigma $ and $\epsilon <\infty,$ and satisfies the estimate
\begin{align*}
\mathbb E \|A^\beta V(t)\|^2 
 \leq  &C[e^{-2\nu t}\mathbb E \|A^\beta V_0\|^2 +  \|F_1\|_{\mathcal F^{\beta,\sigma}(L^p([0,L]))}^2     \\
& \quad +    \|A^\delta G\|_{\mathcal F^{\beta+\frac{1}{2},\sigma}(\gamma(H;L^p([0,L])))}^2  t^{2\beta}], \hspace{1.5cm} 0\leq t\leq T  \notag
\end{align*}
for any $T>0$. 
 Here,  $C>0$ is some constant depending only on the exponents.
\end{theorem}

\subsection{Example 2}
Let us generate the equation \eqref{P32}. Consider an equation of the form
\begin{equation} \label{E17}
\begin{cases}
\begin{aligned}
\frac{\partial X(x,t)}{\partial t}=&\sum_{i,j=1}^n \frac{\partial }{\partial x_j} [a_{ij}(x)\frac{\partial }{\partial x_i}X] -b(x)X+f(t) \varphi(x) \\
& +G(t) \dot W(t) \hspace{3cm}  \text { in } \mathcal O \times (0,T),
\end{aligned}\\
\begin{aligned}
& X=0  \hspace{4cm} &    \text {on } \partial\mathcal O \times (0,T),\\
& X(x,0)=X_0(x)  \hspace{4cm}&    \text { in }  \mathcal O  
\end{aligned}
\end{cases}
\end{equation}
in a bounded domain $\mathcal O\subset \mathbb R^n$ with $\mathcal C^2$ boundary $\partial\mathcal O$ $(n=1,2,\dots)$.  Here, 
\begin{description}
 \item [{\rm (Ex2.1)}]
$a_{ij}\in L^\infty (\mathcal O, \mathbb R) $ for $ 1\leq i,j\leq n$, and  satisfies 
$$ \sum_{i,j=1}^n a_{ij}(x) z_i z_j \geq a_0 \|z\|_{\mathbb R^n}^2, \hspace{1cm} z=(z_1,\dots,z_n)\in \mathbb R^n, \text{   a.e. }  x\in \mathcal O
$$
for some $a_0>0.$
 \item [{\rm (Ex2.2)}] $b\in L^\infty (\mathcal O, \mathbb R)$ and satisfies 
  $  b(x) \geq b_0$   a.e.  $ x\in \mathcal O$ for some $b_0>0$.
\end{description}
The functions $f$ and $\varphi$ are random and non-random real-valued functions, respectively, and $G(t) \dot W(t)$ denotes space-time white noise making effect on the system.

To solve \eqref{E17}, we formulate it as a problem of the form \eqref{linear}. 
Set the underlying space  $E=H^{-1}(\mathcal O)=[\mathring{H}^1(\mathcal O)]'$, where
$$\mathring{H}^1(\mathcal O)=\{u\in H^1(\mathcal O); u|_{\partial \mathcal O}=0\}.$$
  The norm in $E$ is denoted by $\|\cdot\|$. Since $E$ is the dual space of the Hilbert space $\mathring{H}^1(\mathcal O)$, $E$ is a Hilbert space. Hence $E$ is an M-type $2$ Banach space.

Assume that
\begin{description}
  \item [{\rm (Ex2.3)}] $W$ is a cylindrical   Wiener process  on a separable Hilbert space $H$, which is defined on a complete filtered probability space $(\Omega, \mathcal F,\mathcal F_t,\mathbb P).$
   \item [{\rm (Ex2.4)}] For some $ 0<\sigma<\beta< \frac{1}{2}, $ 
 $$f\in \mathcal F^{\beta, \sigma} ((0,T];\mathbb R)  \hspace{1cm} \text{ a.s.  and }  \hspace{0.3cm} \mathbb E\|f\|_{\mathcal F^{\beta, \sigma}(\mathbb R)}^2<\infty.$$
   \item [{\rm (Ex2.5)}]  $\varphi \in   H^{-1}(\mathcal O).$
   \item [{\rm (Ex2.6)}]  $G$ is  an adapted, square integrable  process from  $[0,T]\times \Omega$ to  the space  $\gamma(H;$$H^{-1}(\mathcal O)), $ which is $H$-strongly measurable. In addition, there exists $1-\beta<\delta\leq 1$ such that 
  $$  
          A^\delta G\in \mathcal F^{\beta+\frac{1}{2}, \sigma} ((0,T];\gamma(H;H^{-1}(\mathcal O))) \hspace{1cm} \text{ a.s.}
 $$
and 
$$\mathbb E\|A^\delta G\|_{\mathcal F^{\beta+\frac{1}{2}, \sigma} (\gamma(H;H^{-1}(\mathcal O))}^2<\infty.$$
    \item [{\rm (Ex2.7)}] $X_0$ is an $\mathcal  F_0$-measurable  random variable.
\end{description}

Let $A$ be a realization of the differential operator 
$$\,{-}\,\sum_{i,j=1}^n \frac{\partial }{\partial x_j} [a_{ij}(x)\frac{\partial }{\partial x_i}] +b(x)$$ in $E$ under the Dirichlet boundary conditions. According to  \cite{yagi}, it is known that $\mathcal D(A)= \mathring{H}^1(\mathcal O)$  
and  $A$  is a sectorial operator.

Using $A$,  \eqref{E17} can be written as a problem of the form \eqref{linear} with
$F_1(t)=f(t) \varphi(x).$
It is obvious that {\rm (Ex2.4)}, {\rm (Ex2.5)} and {\rm (Ex2.6)} imply {\rm (F1)} and {\rm (Ga)} for $F$ and $G$, respectively.  Theorems \ref{Th2} and  \ref{Th3} are then available to \eqref{E17}.

\begin{theorem}
Let {\rm (Ex2.1)}-{\rm (Ex2.7)} be satisfied. Then, there exists a unique  strict solution  of \eqref{E17}  in $H^{-1}(\mathcal O)$   possessing the regularity
$$ AX\in \mathcal C([0,T];H^{-1}(\mathcal O))\hspace{1cm} \text{ a.s.} $$
with the estimate
\begin{align*}
\mathbb E & \|X(t)\|^2  + t^2 \mathbb E \|AX(t)\|^2 \\
 \leq &  C[\mathbb E \|X_0\|^2 +      \mathbb E\|f\|_{\mathcal F^{\beta,\sigma}(\mathbb R)}^2 \|\varphi\|^2  t^{2\beta}+\mathbb E\|A^\delta G\|_{\mathcal F^{\beta+\frac{1}{2},\sigma}(\gamma(H;H^{-1}(\mathcal O)))}^2   \notag\\
& \quad \times  \{      t^{2\beta}  +t^{2(\beta+\delta)}   \}    ], \hspace{1cm} 0\leq t\leq T.     \notag  
\end{align*}

 Furthermore,  when
 $X_0\in \mathcal D(A^\beta) $ a.s. such that  $\mathbb E\|A^\beta X_0\|^2<\infty,$ 
 the strict solution $X$  has the maximal space-time regularities
$$
 X \in \mathcal C([0,T];\mathcal D(A^\beta)) \cap \mathcal C^{\gamma_1}([0,T];H^{-1}(\mathcal O))
\hspace{1cm} \text{a.s.,}
$$
$$AX \in \mathcal C^{\gamma_2}([\epsilon,T];H^{-1}(\mathcal O))    \hspace{1cm} \text{a.s.}$$
for any  $ 0<\gamma_1<\beta$, $ 0<\gamma_2<\beta+\delta-1, 0<\gamma_2\leq \sigma $ and $0<\epsilon <T,$ 
and satisfies  the estimate
\begin{align*}
\mathbb E \|A^\beta X(t)\|^2 
 \leq  &C[e^{-2\nu t}\mathbb E \|A^\beta X_0\|^2 +  \mathbb E \|f\|_{\mathcal F^{\beta,\sigma}(H^{-1}(\mathcal O))}^2  \|\varphi\|^2  \\
& \quad +   \mathbb E \|A^\delta G\|_{\mathcal F^{\beta+\frac{1}{2},\sigma}(\gamma(H;H^{-1}(\mathcal O)))}^2  t^{2\beta}], \hspace{1cm} 0\leq t\leq T.  \notag
\end{align*}
Here,  $C>0$ is some constant depending only on the exponents.
\end{theorem}

\subsection{Example 3}
Consider a nonlinear  problem
\begin{equation} \label{P48}
\begin{cases}
\begin{aligned}
\frac{\partial X(x,t)}{\partial t}=&\sum_{i,j=1}^n \frac{\partial }{\partial x_j} [a_{ij}(x)\frac{\partial }{\partial x_i}u] +f_1(X)+t^{\beta-1}f_2(t) \varphi(x)  \\
  &+t^{\beta-\frac{1}{2}} g(t) \dot W(t)     \hspace{1cm} \text { in } \mathbb R^n \times (0,T),\\
 X(x,0)=&X_0(x)      \hspace{4cm} \text { in }  \mathbb R^n. 
\end{aligned}
\end{cases}
\end{equation}
Here,  
\begin{description}
  \item [{\rm (Ex3.1)}] $W$ is a cylindrical   Wiener process  on a separable Hilbert space $H$, which is defined on a complete filtered probability space $(\Omega, \mathcal F,\mathcal F_t,\mathbb P).$
   \item [{\rm (Ex3.2)}]  $a_{ij}\in L^\infty (\mathbb R^n,\mathbb R)$ for $ 1\leq i,j\leq n.$ In addition,
there exists $ a_0>0 $ such that  
$$\sum_{i,j=1}^n a_{ij}(x) z_i z_j \geq a_0 \|z\|_{\mathbb R^n}^2, \hspace{0.3cm} z=(z_1,\dots,z_n)\in \mathbb R^n, \text{a.e. } x\in \mathbb R^n.
$$
\item  [{\rm (Ex3.3)}]
$\varphi \in H^{-1}(\mathbb R^n)$ and  $f_1$ is measurable function from a domain $ \mathcal D(f_1)$ of $ H^{-1}(\mathbb R^n)$ to $ \mathbb R.$
\item [{\rm (Ex3.4)}]
 $ f_2\in \mathcal C^\sigma ([0,T];\mathbb R) \text{ and } g\in \mathcal C^\sigma ([0,T];\gamma(H;H^{-1}(\mathbb R^n))). $ In addition,
$f_2(0)=0 $ and $g_2(0)=0 $ for some $ 0<\sigma<\beta<\frac{1}{2}.$
   \item [{\rm (Ex3.5)}] $X_0$ is an $\mathcal  F_0$-measurable  random variable.
 \end{description}

 Let $A$ be a realization of the differential operator 
$-\sum_{i,j=1}^n \frac{\partial }{\partial x_j} [a_{ij}(x)\frac{\partial }{\partial x_i}] +1$ in $(E,\|\cdot\|)=(H^{-1}(\mathbb R^n),\|\cdot\|_{H^{-1}(\mathbb R^n)}).$ Thanks to \cite[Theorem 2.2]{yagi}, the realization $A$ is a sectorial operator of $H^{-1}(\mathbb R^n)$   with domain
$\mathcal D(A)=H^1(\mathbb R^n).$
 As a consequence, $(-A)$ generates an analytical semigroup on $H^{-1}(\mathbb R^n)$.

Using $A$, the equation \eqref{P48} is formulated as a problem of the form \eqref{E101} in the Hilbert space $H^{-1}(\mathbb R^n),$ where  $F_1, F_2$ and $G$ are defined as follows.
The functions $F_1\colon (0,T] \to H^{-1}(\mathbb R^n)$ and $G\colon  (0,T] \to \gamma(H;H^{-1}(\mathbb R^n))$ are defined by
$$F_1(t)=t^{\beta-1} f_2(t) \varphi(x), \quad G(t)=t^{\beta-\frac{1}{2}} g(t).$$
Remark \ref{rm1} provides  that   
$$F_1 \in \mathcal F^{\beta,\sigma} ((0,T];H^{-1}(\mathbb R^n)), \quad G \in \mathcal F^{\beta+\frac{1}{2},\sigma} ((0,T];\gamma(H;H^{-1}(\mathbb R^n))). $$ 
In the meantime,  $F_2$ is defined by
$$F_2(u)=u+f_1(u).$$

Assume further that 
\begin{description}
  \item [{\rm (Ex3.6)}]
For some $c>0$ and $\max\{0, 2\eta-\frac{1}{2}\}<\beta<\eta, $ $\mathcal D(f_1)=\mathcal D(A^\eta)$ and 
$$
\|f_1(u)-f_1(v)\| \leq c \|A^\eta (u-v)\|, \hspace{0.5cm}  u,v \in  \mathcal D(A^\eta).
$$
 \end{description}

It is easily seen that  the  assumptions {\rm (Aa)}, {\rm (Ab)}, {\rm (F1)}, {\rm (F2a)} and {\rm (Ga)}  are satisfied. 
According to Theorem \ref{Th6},  we have the following theorem.

\begin{theorem}
Let {\rm (Ex3.1)}-{\rm (Ex3.6)} be satisfied. Suppose further that there exists $\rho>0$ such that
\begin{equation*}   
\mathbb E [\sup_{x\in \mathcal D(A^\eta)}\|A^\rho F_2(x)\|]^2<\infty. 
\end{equation*} 
Let  $X_0 \in \mathcal D(A^\beta)$ a.s. such that $\mathbb E\|A^\beta X_0\|^2<\infty$.  Then, \eqref{P48} possesses a unique local strict solution $X$ on some interval $[0,T_{loc}]$. Furthermore, $X$ has the regularity
\begin{equation*}
X\in  \mathcal C^\gamma([\epsilon,T_{loc}];\mathcal D(A^\eta))\cap \mathcal C([0,T_{loc}];\mathcal D(A^\beta)) \hspace{1cm} \text{a.s.}
\end{equation*}
for any $0<\epsilon<T_{loc}, $ $ 0\leq \gamma<\min\{\beta+\delta-1, \frac{1}{2}-\eta,\frac{1+2\beta}{4}-\eta\},$  and satisfies  the estimate
\begin{align*}
\mathbb E\|AX(t)\|^2    
 \leq  &  C \mathbb E \|X_0\|^2 t^{-2}   + C   \|F_1\|_{\mathcal F^{\beta,\sigma}(H^{-1}(\mathbb R^n))}^2 t^{2(\beta-1)}       \\
&    +C  \|A^\delta G\|_{\mathcal F^{\beta+\frac{1}{2},\sigma}(\gamma(H;H^{-1}(\mathbb R^n)))}^2 t^{2(\beta+\delta-1)}   \notag  \\
&   
+C [\sup_{x\in \mathcal D(A^\eta)}\|A^\rho F_2(x)\|]^2   t^{2\rho},    \hspace{1.5cm} 0< t\leq T_{loc},  \notag
\end{align*}
where  $C>0$ is some constant depending on the exponents.
\end{theorem}


\begin{thebibliography}{9}
\bibitem {Brzezniak} Z. Brze\'{z}niak,  
      Stochastic partial differential equations in M-type 2 Banach spaces, 
      Potential Anal. 4  (1995)   1-45.  
  \bibitem {Brzezniak1} Z. Brze\'{z}niak,   On stochastic convolution in Banach spaces and applications, 
   Stochastics Stochastics Rep. 61 (1997) 245-295.   
  \bibitem {Brzezniak1.5}  Z. Brze\'{z}niak, 
 Some remarks on Ito and Stratonovich integration in 2-smooth Banach spaces,
      Probabilistic Methods in Fluids, World Sci. Publ. (2003) 48-69. 
 \bibitem {Brzezniak2}  Z. Brze\'{z}niak,   J.M.A.M. van Neerven,   
      Stochastic convolution in separable Banach spaces and the stochastic linear Cauchy problem,
   Studia Math. 143 (2000)  43-74.   
\bibitem {Brzezniak3} Z. Brze\'{z}niak,  J.M.A.M. van Neerven,  M.C. Veraar,  L. Weis,   
 It\^{o}'s formula in UMD Banach spaces and regularity of solutions of the Zakai equation, 
J. Differential Equations 245 (2008)  30-58.    
 \bibitem {CurtainFalb} R.F. Curtain,    P.L. Falb, 
       Stochastic differential equations in Hilbert space, 
      J. Differential Equations 10  (1971)    412-430. 
\bibitem{prato-Flandoli}  G. Da Prato,  F. Flandoli,       
     Pathwise uniqueness for a class of SDE in Hilbert spaces and applications, 
      J. Funct. Anal. 259 (2010)  243-267. 
\bibitem{pratoFlandoliPriolaRockner}   G. Da Prato,  F. Flandoli,  E. Priola, M. R\"{o}ckner,   
       Strong uniqueness for stochastic evolution equations in Hilbert spaces perturbed by a bounded measurable drift,  
      Ann. Probab. 41  (2013) 3306-3344. 
\bibitem{pratoFlandoliPriolaRockner2}  G. Da Prato,  F. Flandoli, E. Priola, M. R\"{o}ckner, 
       Strong uniqueness for stochastic evolution equations with unbounded measurable drift term,
       J. Theoret. Probab. 28 (2015)  1571-1600.  
\bibitem{prato} G. Da Prato, J. Zabczyk,     
      Stochastic Equations in Infinite Dimensions,    
      Cambridge, 1992.  
\bibitem{Dawson0} D.A. Dawson,  
       Stochastic evolution equations, 
      Math. Biosci. 15  (1972)  287-316. 
\bibitem {Dettweiler} E. Dettweiler,     
      On the martingale problem for Banach space valued stochastic differential equations,  
     J. Theoret. Probab.  2 (1989) 159-191. 
\bibitem {Dettweiler3} E. Dettweiler,  
      Stochastic integration relative to Brownian motion on a general Banach space, 
  Do\u{g}a Mat.  15 (1991) 58-97. 
\bibitem {Gawarecki} L. Gawarecki,V. Mandrekar, 
       Stochastic Differential Equations in Infinite Dimensions with Applications to Stochastic Partial Differential Equations,  
        Probability and its Applications, Springer, Heidelberg, 2011.  
\bibitem{Hairer}     M. Hairer,  
An introduction to stochastic PDEs,   ArXiv e-prints  (2009) [arXiv:0907.4178]. 
\bibitem{Haak}   B.H. Haak,  J.M.A.M. van Neerven,   
Uniformly $\gamma$-radonifying families of operators and the stochastic Weiss conjecture,   
Oper. Matrices 6 (2012) 767-792.  
\bibitem {Hille} E. Hille,  
       Functional Analysis and Semi-Groups,  
       Amer. Math. Soc. Coll. Publ.,  Amer. Math. Soc. 31, 1948. 

\bibitem{Martin} O. Martin,  Uniqueness for stochastic evolution equations in Banach spaces,
      Dissertationes Math. (Rozprawy Mat.) 426 (2004) 63 pp.   




 \bibitem{Osaki}    K. Osaki, A. Yagi,    
Global existence for a chemotaxis-growth system in $\mathbb R^2$,
  Adv. Math. Sci. Appl.  12  (2002)  587-606.  
\bibitem{Pisier}  G. Pisier,    
         Probabilistic methods in the geometry of Banach spaces,   
         Probability and Analysis  1206  (1986)  167-241.  
\bibitem{PrevotRockner} C. Pr\'{e}v\^{o}t, M. R\"{o}ckner,   
          A Concise Course on Stochastic Partial Differential Equations, 
         Lecture Notes in Mathematics  1905, Springer, Berlin, 2007.  
\bibitem{Ton1} T.V. T\d{a},  
  Regularity of solutions of abstract linear evolution equations, 
  Lith. Math. J.   56 (2016) 268-290.

\bibitem{Ton1.5} T.V. T\d{a},  
  Note on abstract stochastic semilinear evolution equations, 
  J. Korean Math. Soc. (to appear). Doi: 10.4134/JKMS.j160311

\bibitem{Ton1.6}	T.V. T\d{a}, A. Yagi, 
 Non-autonomous stochastic evolution equations in Banach spaces of martingale type 2: strict solutions and maximal regularity,
 Discrete Contin. Dyn. Syst. (to appear).

\bibitem{Ton2}  T.V. T\d{a},  Y. Yamamoto,  A. Yagi,  
Strict solutions to  stochastic parabolic evolution equations in M-type $2$ Banach spaces, 
Funkcial. Ekvac.  (to appear).

\bibitem {van1} J.M.A.M. van Neerven,  M.C. Veraar,   L. Weis, 
  Stochastic integration in UMD Banach spaces, 
 Ann. Probab.  35  (2007) 1438-1478.  
\bibitem {van2} J.M.A.M. van Neerven,  M.C. Veraar,  L. Weis, 
 Stochastic evolution equations in UMD Banach spaces, 
 J. Funct. Anal.  255  (2008) 940-993.  
\bibitem {van3}  J.M.A.M. van Neerven,  M.C. Veraar,  L. Weis, 
     Maximal $L^p$-regularity for stochastic evolution equations, 
       SIAM J. Math. Anal. 44 (2012) 1372-1414. 
\bibitem {van4} J.M.A.M. van Neerven,  M.C. Veraar,  L. Weis, 
 Stochastic maximal $L^p$-regularity, 
 Ann. Probab. 40 (2012)  788-812.  
\bibitem {van5} J.M.A.M. van Neerven,  M.C. Veraar,  L. Weis, 
 Maximal $\gamma$-regularity, 
   J. Evol. Equ.  15  (2015)  361-402.    
\bibitem {van6} J.M.A.M. van Neerven,   J. Zhu,  
 A maximal inequality for stochastic convolutions in 2-smooth Banach spaces,
    Electron. Commun. Probab. 16 (2011) 689-705. 
  \bibitem {yagi0} A. Yagi,     
      Fractional powers of operators and evolution equations of parabolic type,  
      Proc. Japan Acad. Ser. A Math. Sci.
      64 (1988)  227-230.   
\bibitem {yagi1} A. Yagi, 
      Parabolic evolution equations in which the coefficients are the generators of infinitely differentiable semigroups,  
      Funkcial. Ekvac.
      32  (1989) 107-124.    
\bibitem {yagi2} A. Yagi,  
      Parabolic evolution equations in which the coefficients are the generators of infinitely differentiable semigroups, II,  
      Funkcial. Ekvac.
      33  (1990) 139-150.   
\bibitem {yagi} A. Yagi,   
      Abstract Parabolic Evolution Equations and their Applications, 
       Springer-Verlag, Berlin, 2010. 
\bibitem {Yosida} K. Yosida,  
 On the differentiability and the representation of one-parameter semigroups of linear operators,  
      J. Math. Soc. Japan  1  (1948)  15-21. 
\bibitem {Walsh} J.B. Walsh,    
An Introduction to Stochastic Partial Differential Equations, 
 \'{E}cole d'\'{e}t\'{e} de probabilit\'{e}s de Saint-Flour, XIV-1984, 265-439, Lecture Notes in Mathematics  1180, Springer, Berlin, 1986.  
\end{thebibliography}
\end{document}